\documentclass[11pt,reqno]{amsart}   	
\usepackage{subfigure}									
\usepackage{amsmath}
\usepackage{amssymb, amsfonts}
\usepackage{amsthm}
\usepackage{hyperref}
\usepackage{mathtools}
\usepackage{tikz}
\usetikzlibrary{cd}
\usetikzlibrary{decorations.markings,decorations.pathmorphing}
\usepackage{enumitem}
\usepackage{makecell}
\setcellgapes{4pt}

\usepackage[margin=1in]{geometry}

\newtheorem{thm}{Theorem}[subsection]
\newtheorem*{thm*}{Theorem}
\newtheorem{lem}[thm]{Lemma}
\newtheorem{cor}[thm]{Corollary}
\newtheorem{prop}[thm]{Proposition}
\newtheorem{conj}[thm]{Conjecture}

\newtheorem{thmIntro}{Theorem}
\newtheorem{corIntro}[thmIntro]{Corollary}

\theoremstyle{definition}
\newtheorem{define}[thm]{Definition}
\newtheorem{defthm}[thm]{Definition-Theorem}
\newtheorem*{define*}{Definition}
\newtheorem{ex}[thm]{Example}

\newtheorem*{nota*}{Notation}
\newtheorem{rem}[thm]{Remark}
\newtheorem*{rem*}{Remark}

\newcommand{\D}{\mathcal{D}}
\newcommand{\U}{\mathcal{U}}

\newcommand{\X}{\mathcal{X}}
\newcommand{\Y}{\mathcal{Y}}
\newcommand{\T}{\mathcal{T}}
\renewcommand{\S}{\mathcal{S}}

\newcommand{\F}{\mathcal{F}}

\newcommand{\Hom}{\mathrm{Hom}}
\newcommand{\Ext}{\mathrm{Ext}}
\newcommand{\End}{\mathrm{End}}

\newcommand{\image}{\mathrm{Im}}
\renewcommand{\dim}{\mathrm{dim}}
\newcommand{\supp}{\mathrm{supp}}
\renewcommand{\ker}{\mathrm{ker}}
\newcommand{\coker}{\mathrm{coker}}
\newcommand{\rk}{\mathrm{rk}}

\newcommand{\mods}{\mathsf{mod}}

\newcommand{\Db}{\mathcal{D}^b}

\newcommand{\tors}{\mathsf{tors}}

\newcommand{\sbrick}{\mathsf{sbrick}}
\newcommand{\brick}{\mathsf{brick}}

\newcommand{\add}{\mathsf{add}}
\newcommand{\Filt}{\mathsf{Filt}}

\newcommand{\Fac}{\mathsf{Fac}}
\newcommand{\Sub}{\mathsf{Sub}}

\newcommand{\cone}{\mathsf{cone}}
\newcommand{\cocone}{\mathsf{cocone}}

\newcommand{\lperp}[1]{\prescript{\perp}{}{#1}}

\newcommand{\inv}{\mathrm{inv}}
\newcommand{\src}{\mathrm{src}}
\newcommand{\tar}{\mathrm{tar}}
\newcommand{\suppo}{\supp^{\circ}}

\newcommand{\covers}{{\,\,\,\cdot\!\!\!\! >\,\,}}
\newcommand{\covered}{{\,\,<\!\!\!\!\cdot\,\,\,}}

\newcommand{\des}{\mathrm{des}}
\newcommand{\asc}{\mathrm{asc}}
\usepackage{soul}

\title[Pairwise Compatibility for 2-Simple Minded Collections II]{Pairwise Compatibility for 2-Simple Minded Collections II: Preprojective Algebras and Semibrick Pairs of Full Rank}
\author{Emily Barnard}
\address{Department of Mathematical Sciences, DePaul Universsity, Chicago, IL 60604, USA}\email{e.barnard@depaul.edu}

\author{Eric J. Hanson}
\address{Department of Mathematics, Brandeis University, Waltham, MA 02453, USA}\email{ehanson4@brandeis.edu}
\date{May 16, 2022}


\begin{document}

\noindent This version of the article has been accepted for publication, after peer review (when applicable) and is subject to Springer Nature’s AM terms of use, but is not the Version of Record and does not reflect post-acceptance improvements, or any corrections. The Version of Record is available online at: https://doi.org/10.1007/s00026-022-00585-4.

\bigskip

\maketitle

\begin{abstract}
    Let $\Lambda$ be a finite-dimensional associative algebra over a field. A \emph{semibrick pair} is a finite set of $\Lambda$-modules for which certain Hom- and Ext-sets vanish. A semibrick pair is \emph{completable} if it can be enlarged so that a generating condition is satisfied. We prove that if $\Lambda$ is $\tau$-tilting finite with at most 3 simple modules, then the completability of a semibrick pair can be characterized using conditions on pairs of modules. We then use the weak order to construct a combinatorial model for the semibrick pairs of preprojective algebras of type $A_n$. From this model, we deduce that any semibrick pair of size $n$ satisfies the generating condition, and that the dimension vectors of any semibrick pair form a subset of the column vectors of some $c$-matrix. Finally, we show that no ``pairwise" criteria for completability exists for preprojective algebras of Dynkin diagrams with more than 3 vertices.
\end{abstract}


\setcounter{tocdepth}{1}
\tableofcontents


\section{Introduction}
Let $\Lambda$ be a finite-dimensional algebra over an arbitrary field $K$. We denote by $\mods\Lambda$ the category of finitely-generated (right) $\Lambda$-modules and by $\Db(\mods\Lambda)$ the corresponding bounded derived category.
Recall that an object $S \in \mods\Lambda$, or more generally $S \in \Db(\mods\Lambda)$, is called a \emph{brick} if $\End_\Lambda(S)$ is a division algebra.
Following \cite{asai}, we call a (possibly empty) collection of hom-orthogonal bricks a \emph{semibrick}.
Let $\U$ and $\D$ be semibricks and write $\U[1]=\{T[1]\in  \Db(\mods\Lambda): T\in \U\}$. 
Then we say $\D\sqcup\U[1]$ is a \emph{semibrick pair} if for each $S\in \U$ and each $T\in \D$, we have $\Hom_\Lambda(S, T) = 0$ and $ \Ext_\Lambda^1(S,T) =0$.
 If in addition the bricks in $\D\sqcup \U[1]$ ``generate'' the bounded derived category $\Db(\mods\Lambda)$, then $\D\sqcup \U[1]$ is a \emph{2-term simple-minded collection}. (This is made more precise in Definition~\ref{def:semibrick_pair} and Remark~\ref{rmk:2-term}.) In this paper, we investigate when semibrick pairs can be enlarged to 2-term simple minded collections. More precisely, we have the following definitions.

\begin{define}
Let $\Lambda$ be a finite-dimensional algebra, and let $\D\sqcup\U[1]$ be a semibrick pair.
\begin{enumerate}
\item We say that $\D\sqcup\U[1]$ is \emph{completable} provided that there exists a 2-term simple minded collection $\D'\sqcup \U'[1]$ such that $\D \subseteq \D'$ and $\U \subseteq \U'$.
\item We say that $\D\sqcup\U[1]$ is \emph{pairwise completable} provided that for all $S \in \D$ and $T \in \U$ there exists a 2-term simple minded collection $\D_S^T\sqcup \U_S^T[1]$ with $S \in \D_S^T$ and $T\in \U_S^T$.
\item We say that $\Lambda$ has the \emph{pairwise 2-simple minded completability property}\footnote{The word \emph{compatibility} is used in place of completability in \cite{HI_smc}. We have chosen to use the term completability since, \emph{a priori}, determining whether a semibrick pair is completable is not characterized internally.} provided that each pairwise completable semibrick pair is completable.
\end{enumerate}
\end{define}

The pairwise 2-simple minded completability property is quite natural. For example, both representation finite hereditary algebras \cite{IT_exceptional} and Nakayama algebras \cite{HI_picture} have the pairwise 2-simple minded completability property. More specifically, in \cite{GM_noncrossing}, 2-term simple minded collections were classified using a combinatorial model for certain special Nakayama algebras called \emph{tiling algebras}. Not only do tiling algebras have the pairwise 2-simple minded completability property, but this pairwise condition can be described in terms of a (non)crossing condition for certain arcs in a disc. This also holds true more generally for arbitrary Nakayama algebras \cite{HI_picture} and ($\tau$-tilting finite, monomial) quotients of hereditary algebras of type $A$ and $\widetilde{A}$ \cite{HI_smc}. We create a similar model for preprojective algebras of type $A$ in Section~\ref{sec:typeA} of this paper.

From the perspective of representation theory, 2-term simple minded collections are in bijection with many other classes of objects which satisfy ``pairwise conditions". These include $\tau$-tilting pairs \cite{AIR}, 2-term silting complexes \cite{aihara}, and canonical join representations of functorially finite torsion classes \cite{BCZ,asai}. It is therefore surprising that there exist $\tau$-tilting finite algebras which do not have this property. For example, a recent paper by Igusa and the second author \cite{HI_smc} shows that a \linebreak $\tau$-tilting finite gentle algebra (whose quiver contains no loops or 2-cycles) has the pairwise 2-simple minded completability property if and only if its quiver contains no vertex of degree 3 or 4.

The purpose of the present paper is to explore the combinatorial structures which encode completability.
 We focus in particular on the preprojective algebras $\Pi_W$ of Dynkin type $W$ because they have a close connection to the weak (Bruhat) order.
 The weak order on $W$ is isomorphic to the lattice of torsion classes of $\Pi_W$ \cite[Theorem~0.2]{mizuno}.
 Moreover each algebraic quotient of $\Pi_W$ gives rise to a lattice quotient of the weak order \cite{DIRRT}.
 Further connections to cluster algebras, rowmotion, and geometry are discussed in Section~\ref{sec:motivation}.
 

\subsection{Organization and Main Results}

The contents of this paper are as follows. In Section~\ref{sec:motivation} we discuss several sources of motivation for studying the pairwise 2-simple minded completability property. In Section~\ref{sec:background}, we review necessary background information pertaining to semibrick pairs, 2-term simple minded collections, torsion classes, and mutation.

In Section~\ref{sec:rank3}, we prove our first main result.

\begin{thmIntro}[Theorem \ref{thm:3 vertex}]\label{Main_A}
Let $\Lambda$ be a $\tau$-tilting finite algebra such that $\mods\Lambda$ contains at most three simple modules (up to isomorphism). Then $\Lambda$ has the pairwise 2-simple minded completability property.
\end{thmIntro}

Theorem~\ref{Main_A} also allows us to characterize the completability and pairwise completability of a semibrick pair in terms of \emph{wide subcategories} (Theorem~\ref{thm:wide}).

In Section~\ref{sec:arcs}, we lay the foundation for our combinatorial model. We first review the weak order on the Weyl group $A_n$ and Reading's use of ``noncrossing arc diagrams'' to model the descents of permutations \cite{reading_arcs}. We then introduce \emph{2-colored noncrossing arc diagrams} (Definition~\ref{def:2-colored}), which are pairs of noncrossing arc diagrams satisfying some compatibility condition.

In Section~\ref{sec:typeA}, we consider preprojective algebras of Dynkin type A (Definition~\ref{def:preproj}). By extending the bijection between noncrossing arc diagrams and the semibricks over these algebras established in \cite{BCZ}, we obtain our third main result.

\begin{thmIntro}[Theorem~\ref{arc_pairwise_compat}, simplified]\label{Main_B}
    Let $\Pi_{A_n}$ be a preprojective algebra of Dynkin type $A$. Then there is a bijection between the set of semibrick pairs for $\Pi_{A_n}$ and the set of 2-colored noncrossing arc diagrams on $n+1$ nodes. Moreover, a semibrick pair $\D\sqcup \U[1]$ is a 2-term simple minded collection if and only if the corresponding 2-colored noncrossing arc diagram encodes the ascents and descents of some permutation.
\end{thmIntro}

From Theorem~\ref{Main_B}, we deduce several interesting consequences. Two notable ones are as follows. (See Definition-Theorem \ref{defthm:cVect} for an explanation of $c$-vectors and $c$-matrices.)

\begin{corIntro}[Corollaries~\ref{cor:Main_C1} and~\ref{cor:Main_C2}]\label{Main_C}
    Let $\Pi_{A_n}$ be a preprojective algebra of Dynkin type $A$, and let $\D \sqcup \U[1]$ be a semibrick pair. Then
    \begin{enumerate}
        \item $\D\sqcup \U[1]$ is a 2-term simple minded collection if and only if $|\D| + |\U| = n$.
        \item There exists a $c$-matrix of $\Pi_{A_n}$ which contains the dimension vectors of the bricks in $\D$ and the negatives of the dimension vectors of the bricks in~$\U$.
    \end{enumerate}
\end{corIntro}

\begin{rem}\label{rem:size_equals_rank}\
\begin{enumerate}
    \item For any 2-term simple minded collection $\D \sqcup \U[1]$ (over any algebra), the sum $|\D|+ |\U|$ is always equal to $\rk(\Lambda)$ \cite[Corollary~5.5]{KY}. Although our proof of Corollary~\ref{Main_C}(1) is specific to the type $A$ preprojective algebra, we do not know of any $\tau$-tilting finite algebra $\Lambda$ which admits a semibrick pair $\D\sqcup \U[1]$ which is not a 2-simple minded collection but satisfies $|\D| + |\U| = \rk(\Lambda)$.
    
    \item We note that if we replace $\Pi_{A_n}$ with one of the counterexamples to the pairwise 2-simple minded completability property found in \cite{HI_smc}, then the conclusion of Corollary~\ref{Main_C}(2) will not hold. This seems to indicate the the $c$-vectors and $c$-matrices are better behaved for preprojective algebras than they are in general.
\end{enumerate}
\end{rem}

In Section~\ref{sec:preprojective}, we study the pairwise 2-simple minded completability property for preprojective algebras of Dynkin type. In particular, we prove our final main theorem.

\begin{thmIntro}[Theorem~\ref{thm:preproj}]\label{Main_D}
    Let $W$ be a finite Weyl group. Then the preprojective algebra of type $W$ has the pairwise 2-simple minded completability property if and only if $W$ is of type $A_1$, $A_2$, or~$A_3$.
\end{thmIntro}


\section{Motivation}\label{sec:motivation}
    In this section, we give an overview of our motivation for this paper. In particular, we give examples and interpretations of the pairwise 2-simple minded completability property in representation theory, combinatorics, and geometry.

\subsection{Picture groups and picture spaces}

Our original motivation comes from the study of \emph{picture groups} and \emph{picture spaces}.
The picture group of an algebra was first defined by Igusa--Todorov--Weyman \cite{ITW} in the (representation finite) hereditary case and later generalized to \linebreak $\tau$-tilting finite algebras by the second author and Igusa \cite{HI_picture}.
It is a finitely presented group whose relations encode the structure of the lattice of torsion classes.
Recently, picture groups for valued Dynkin quivers of finite type were shown to be closely related to maximal green sequences \cite{IT_mgs}.
The corresponding picture space is the classifying space of the \emph{($\tau$)-cluster morphism category} of the algebra.
This category encodes the geometry of the support $\tau$-rigid pairs of the algebra and was first defined by Igusa--Todorov \cite{IT_exceptional} in the hereditary case and later generalized by Buan--Marsh \cite{BM}.
Using techniques developed in \cite{igusa_category}, the second author and Igusa have shown that the picture group and picture space have isomorphic (co-)homology when the algebra $\Lambda$ has the pairwise 2-simple minded completability property (plus one technical condition outlined in \cite{HI_picture}).

\subsection{Cover relations in lattices of torsion lattices}

In this paper we largely focus on the connection with torsion classes, where pairwise conditions are quite natural. In this context, Asai has shown that each 2-term simple minded collection corresponds to a set of bricks which ``label'' the upper and lower cover relations of a torsion class in the lattice $\tors \Lambda$ \cite{asai}.
(We review background on torsion classes and the lattice $\tors \Lambda$ in Section~\ref{tors_sec}. For the definition of cover relation, see Definition~\ref{cover_relation}.)
More precisely, following the construction in \cite{BCZ}, we say that a brick $S$ \emph{labels} an upper cover relation $\T\covered \T'$ in the lattice $\tors \Lambda$ provided that $\T' = \Filt(\T\cup S)$.
That is, $\T'$ is the closure of $\T\cup \{S\}$ under iterative extensions.
(The brick $S$ is called a \emph{minimal extending module}. See Definition~\ref{def:minextending} and Theorem~\ref{min_ext}.)
Dually, a brick $S$ \emph{labels} a lower cover relation $\T \covers \T''$ if $S$ labels the corresponding relation $(\T'')^\perp \covered \T^\perp$ in the lattice of torsion free classes.
In our notation, the set $\U[1]$ corresponds to the set of bricks labeling the upper cover relations for some torsion class $\T$, and $\D$ is the set of bricks labeling its lower cover relations.
 Asai's result then says that every 2-term simple minded collection appears as the labels of the upper and lower cover relations for some torsion class. More precisely, if $\D\sqcup \U[1]$ is a 2-term simple minded collection, then $\Filt\Fac(\D)$ is the unique torsion class with cover relations labeled by $\D\sqcup \U[1]$. Moreover, if the lattice $\tors\Lambda$ is finite, then the association $\D\sqcup \U[1] \mapsto \Filt\Fac(\D)$ is a bijection between 2-term simple minded collections and torsion classes. See Theorem~\ref{torsion_to_simple} for additional details.

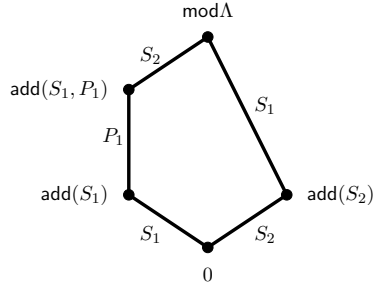
\begin{figure}[h]
\centering
\scalebox{.7}{
\begin{tikzpicture}
\draw[fill] (2,0) circle [radius=.1];
\draw[fill] (.5,1) circle [radius=.1];
\draw[fill] (3.5,1) circle [radius=.1];
\draw[fill] (.5,3) circle [radius=.1];
\draw[fill] (2,4) circle [radius=.1];

\draw [black, line width=.65 mm] (2,0) to (.5,1);
\draw [black, line width=.65 mm] (2,0) to (3.5,1);
\draw [black, line width=.65 mm] (.5,1) to (.5,3);
\draw [black, line width=.65 mm] (2,4) to (.5,3);
\draw [black, line width=.65 mm] (2,4) to (3.5,1);

\node [below] at (2,-.25) {$0$};
\node [left] at (.25,1) {$\add(S_1)$};
\node [left] at (.25,3) {$\add(S_1, P_1)$};
\node [above] at (2,4.25) {$\mods \Lambda$};
\node [right] at (3.75,1) {$\add(S_2)$};

\node [left] at (1.25,.25) {$S_1$};
\node [left] at (1.25,3.65) {$S_2$};
\node [left] at (.55,2.1) {$P_1$};
\node [right] at (2.75,.25) {$S_2$};
\node [right] at (2.75,2.7) {$S_1$};

\end{tikzpicture}}
\caption{\label{fig_brick_labeling} An example of the brick labeling for the lattice
$\tors A_2$ 
(of the hereditary algebra $KA_2$ where $A_2 =(1\rightarrow 2)$).}
\end{figure} 

\begin{ex}
In Figure~\ref{fig_brick_labeling} we display the brick labeling of the lattice of torsion classes for the hereditary algebra $KA_2$ where $A_2 =(1\rightarrow 2)$.
Consider the torsion class $\add(S_1, P_1)$.
The corresponding 2-term simple minded collection is $\D=P_1$ and $\U[1]=S_2$.
In general, we recall from Remark~\ref{rem:size_equals_rank} that $|\D|+|\U|$ is equal to the number of simples in $\mods \Lambda$.
\end{ex}

Now we can rephrase the pairwise 2-simple minded completability property in terms of the brick labeling of $\tors \Lambda$.
An algebra $\Lambda$ has the pairwise 2-simple minded completability property if whenever a semibrick pair $\D\sqcup \U[1]$ is not contained in the set of bricks labeling the upper and lower cover relations for some torsion class, there is a pair of bricks $S\in \D$ and $T[1] \in \U[1]$ such that no torsion class $\T$ has a lower cover relation labeled by $S$ and an upper cover relation labeled by $T$.

The bricks which label only the upper (or only the lower) cover relations of a torsion class are characterized by a pairwise condition, namely each pair of bricks is hom-orthogonal \cite{BCZ}. It also follows from the definition of a torsion class that if $S$ labels an upper cover relation and $T$ labels a lower cover relation, then $\Hom_\Lambda(S,T) = 0$.
Therefore it is surprising that bricks which label a mixture of upper and lower covers do \emph{not} generally satisfy a pairwise condition.

\subsection{Dynamical combinatorics}
One motivation for simultaneously studying the upper and lower cover relations of a torsion class comes from analyzing the so-called ``$\kappa$-map''.

The lattice of torsion classes is known to be (completely) semidistributive \cite{DIRRT}, and the labeling of cover relations by bricks corresponds to the labeling by join-irreducible elements \cite{BCZ,asai}. In particular, we can consider the map $\overline{\kappa}^d$ which sends a torsion class $\T$ with upper cover relations labeled by $\U$ to the (unique) torsion class $\overline{\kappa}^d(\T)$ with lower cover relations labeled by $\U$. Historically, the $\overline{\kappa}^d$ map is sometimes called ``rowmotion'' or ``Kreweras complement'' or simply ``kappa''.
In the context of Coxeter-Catalan combinatorics, the dynamics of $\overline{\kappa}^d$ provide the only known uniform bijection from the set of noncrossing partitions of type $W$ to the set of non-nesting partitions of type $W$, where $W$ is a finite Weyl group \cite{AST}.
This Coxeter-Catalan perspective was translated into representation theory by \cite{IngTho} and \cite{ringel2016} among others; and further explored very recently in \cite{TW}, \cite{Williams_Thomas19} and \cite{BTZ}.
The map $\overline{\kappa}^d$ appears to have important lattice-theoretic implications, as seen in  \cite{RST19}, and homomesic properties as explored in~\cite{Hopkins19}. 

The equivalence of the usual lattice-theoretic definition of $\overline{\kappa}^d$ in terms of join- and meet-irreducible elements and the one given here can be found in \cite{BTZ}.

In case the category $\mods\Lambda$ contains no cycles (so that if $M_1$ and $M_2$ are nonisomorphic indecomposable modules, then at least one of $\Hom_\Lambda(M_1,M_2)$ and $\Hom_\Lambda(M_2,M_1)$ is zero), Thomas and Williams show in \cite{TW} that $\overline{\kappa}^d$ can be computed using so-called ``flips''. However, there are many interesting classes of algebras whose module categories do contain cycles, such as the ``preprojective algebras'' considered in this paper. It remains an interesting question to find an efficient algorithm for computing $\overline{\kappa}^d$ for such algebras.

In general, this problem can be considered as finding a ``completion" of a semibrick. To explain this, let us restrict to the case where the lattice $\tors\Lambda$ is finite. In this case, given a semibrick $\D$, we have $\overline{\kappa}^d(\Filt\Fac(\D)) = \Filt\Fac(\U)$, where $\U$ is the (unique) semibrick making $\D\sqcup \U[1]$ into a 2-term simple minded collection. The semibrick $\U$ may be difficult to compute, but when our algebra satisfies the pairwise 2-simple minded completability property, the computation can be carried out as follows:
\begin{enumerate}
    \item Choose an ordering of $\D = \{D_1,\ldots,D_k\}$ and let $\mathcal{S}_0$ be the set of bricks in $\mods\Lambda$ (up to isomorphism).
    \item For $1 \leq j \leq k$, let $\mathcal{S}_j$ be the set of bricks $U \in \mathcal{S}_{j-1}$ for which $D_j \sqcup U[1]$ is a completable semibrick pair.
    \item The set $\mathcal{S}_k$ will contain a unique semibrick $\U$ of size $n-k$, where $n$ is the number of (non-isomorphic) simple modules in $\mods\Lambda$. This is the semibrick for which $\D\sqcup \U[1]$ is a 2-term simple minded collection; or, equivalently, for which $\overline{\kappa}^d(\Filt\Fac(\D)) = \Filt\Fac(\U)$.
\end{enumerate}
When the pairwise 2-term simple minded completability property does not hold, this algorithm will no longer be sufficient. Indeed, in step 2, there is in general no guarentee that if both $D_1\sqcup U$ and $D_2 \sqcup U$ are completable then $D_1\sqcup D_2 \sqcup U[1]$ is completable as well. Moreover, even if we replace step 2 with the requirement that $\left(\bigsqcup_{i = 1}^j D_i\right)\sqcup U[1]$ be completable for each $j$, there is no guarantee that the set $\mathcal{S}_k$ will contain a unique semibrick of size $n-k$. We do not, however, know of a $\tau$-tilting finite algebra where this is not the case (see Section \ref{sec:discussion}).

\subsection{$c$-vectors}\label{sec:cVect}
Another place that 2-term simple minded collections appear is in the ``wall-and-chamber structures'' associated to finite-dimensional algebras. Given a finite-dimensional algebra, King's stability conditions \cite{king} can be used to define a collection of codimension-1 subspaces of Euclidean space $\mathbb{R}^n$. These codimension-1 subspaces are referred to as ``walls'' and the closure of a connected component of the complement of the walls is called a ``chamber'' or ``region''. See e.g. \cite{IOTW,BST}. For some algebras, such as the ``preprojective algebras'' considered in this paper, the corresponding wall-and-chamber structure is actually a (simplicial) hyperplane arrangement.

The wall-and-chamber structure of an algebra comes with a natural choice of base region. Given a wall $H$, this induces a choice of normal vector $\vec{n}_H$ so that $\vec{n}_H \cdot v < 0$ for any vector $v$ in the base region. This defines a notion of the ``positive side'' and ``negative side'' of a wall.

Now let $R$ be a region. We say a wall $H$ is a \emph{lower facet} (resp. \emph{upper facet}) of $R$ if $H\cap R$ is $(n-1)$-dimensional and $R$ lies on the positive side (resp. negative side) of $H$. Now denote
\begin{eqnarray*}
    \mathfrak{d}(R) &=& \{\vec{n}_H: H \text{ is a lower facet of }R\}\\
    \mathfrak{u}(R) &=& \{-\vec{n}_H: H\text{ is an upper facet of }R\}
\end{eqnarray*}
The set of vectors $\mathfrak{d}(R) \cup \mathfrak{u}(R)$ is referred to as a \emph{$c$-matrix} of the algebra, and the individual vectors are referred to as \emph{$c$-vectors}. See Definition-Theorem \ref{defthm:cVect}. We note that under this formulation, a $c$-matrix is a set of vectors, not an actual matrix. See Section \ref{tors_sec} for additional discussion.

For simplicity, suppose $\Lambda$ is a basic, elementary algebra (so that no indecomposable projective module appears more than once in the direct sum decomposition of $_\Lambda\Lambda$ and the endomorphism ring of any simple module is isomorphic to the field $K$). A result of Treffinger \cite{treffinger} then shows that for any region $R$, there exists a 2-term simple minded collection $\D\sqcup \U[1]$ so that $\mathfrak{d}(R)$ consists of the dimension vectors of the bricks in $\D$ and $\mathfrak{u}(R)$ consists of the negatives of the dimension vectors of the bricks in $\U$. One may then ask about a ``pairwise characterization'' of $c$-matrices. More specifically, let $\mathcal{M}$ be a set of $c$-vectors and suppose that for all pairs $v_i, v_j \in \mathcal{M}$ there exists a $c$-matrix containing $v_i$ and $v_j$. Having $c$-matrices characterized by pairwise conditions would then mean that there exists a $c$-matrix containing $\mathcal{M}$. 

For representation finite hereditary algebras, bricks can only share a dimension vector if they are isomorphic. For these algebras, the pairwise 2-simple minded completability property is thus equivalent to $c$-matrices being characterized by pairwise conditions. (This is actually the approach used in \cite{IT_exceptional} to prove the pairwise 2-simple minded completability property for hereditary algebras of finite type.) In general, however, even algebras which satisfy the pairwise 2-simple minded completability property may not have $c$-matrices characterized by pairwise conditions. See Remark~\ref{rem:notPair} for an example. Even so, for preprojective algebras of type $A$, we show in Corollary~\ref{cor:Main_C2} that it is possible to relate $c$-matrices directly to semibrick pairs.


\section{Background}\label{sec:background}
        Let $\Lambda$ be a finite-dimensional algebra over an arbitrary field $K$. We will assume that $\Lambda$ is basic. We denote by $\mods\Lambda$ the category of finitely generated (right) $\Lambda$-modules and by $\Db(\mods\Lambda)$ the bounded derived category of $\mods\Lambda$ with shift functor $[1]$. 
    We denote by $\rk(\Lambda)$ the number of simple modules in $\mods\Lambda$ up to isomorphism.
    
    Throughout this paper, we assume all subcategories are full and closed under isomorphism. Thus we can identify subcategories with the set of (isoclasses of) objects they contain. Given a subcategory $\mathcal{C}\subseteq\mods\Lambda$ we denote by $\add(\mathcal{C})$ (resp. $\Fac(\mathcal{C}),\Sub(\mathcal{C}))$ the subcategory of $\mods\Lambda$ consisting of direct summands (resp. factors, submodules) of finite direct sums of the objects in $\mathcal{C}$. Likewise, we denote by $\Filt(\mathcal{C})$ the subcategory of $\mods\Lambda$ consisting of objects $M \in \mods\Lambda$ for which there exists a finite filtration
    $$0 = M_0 \subsetneq M_1 \subsetneq \cdots \subsetneq M_k = M$$
    so that $M_i/M_{i-1} \in \mathcal{C}$ for all $i$. Finally, we say two objects $X,Y \in \Db(\mods\Lambda)$ are \emph{hom-orthogonal} if $\Hom_\Lambda(X,Y) = 0 = \Hom_\Lambda(Y,X)$.

\subsection{Semibrick pairs and 2-term simple minded collections}\label{tors_sec}

We denote by $\brick\Lambda$ and $\sbrick\Lambda$ the sets of bricks and semibricks in $\mods\Lambda$ (up to isomorphism).
Recall that a semibrick is a collection of bricks, each pair of which is hom-orthogonal.

\begin{define}\label{def:semibrick_pair}
    Let $\Lambda$ be an arbitrary finite-dimensional algebra. Let $\D,\U\in\sbrick\Lambda$ and let $\X = \D\sqcup \U[1]$.
    \begin{enumerate}
        \item\ If $\Hom_\Lambda(\D,\U) = 0 = \Ext_\Lambda^1(\D,\U)$, then $\X$ is called a \emph{semibrick pair}.
        \item\  If in addition the smallest triangulated subcategory of $\Db(\mods\Lambda)$ containing $\X$ which is closed under direct summands is $\Db(\mods\Lambda$), then $\X$ is called a \emph{2-term simple minded collection}.
    \end{enumerate}
\end{define}
\begin{rem}\label{rmk:2-term}
\normalfont
The original definition of a simple minded collection comes from \cite{al-nofayee}, and requires that $\X$ be a collection of hom-orthogonal bricks in $\Db(\mods\Lambda)$ with $\Hom_{\Db(\mods\Lambda)}(\X,\X[<0]) = 0$.
A 2-term simple minded collection additionally satisfies that each homology of $\X$ vanishes outside of degree 0 and -1.
In \cite[Remark~4.11]{BY}, it is shown that each $X_i$ in a 2-term simple minded collection is either a module or a shift of a module in $\Lambda$, so we take this as our definition here.
\end{rem}

\begin{ex}
Let $\Lambda$ be the hereditary algebra of type $A_2$ from Figure~\ref{fig_brick_labeling}, and consider $P_1\sqcup S_2[1]$.
We observe that there are no non-zero homomorphisms $P_1\to S_2$, and $\Ext_\Lambda^1(P_1,S_2)=0$ because $P_1$ is projective.
Therefore $P_1\sqcup S_2[1]$ is a semibrick pair.
Since $S_1$ is the cokernel of the map $S_2 \hookrightarrow P_1$, we obtain both simple modules after closing under triangles.
Hence $P_1 \sqcup S_2[1]$  ``generates'' $\Db(\mods\Lambda)$.
Therefore $P_1\sqcup S_2[1]$ is a 2-term simple minded collection.
\end{ex}

For simplicity, let us now suppose that the algebra $\Lambda$ is elementary, meaning that $\End_\Lambda(S) \cong K$ for any simple module $S \in \mods\Lambda$. Choose an ordering $P_1,\ldots,P_{\rk(\Lambda)}$ on the (isomorphism classes of) indecomposable projective modules in $\mods \Lambda$. For $j \in \{1,\ldots,\rk(\Lambda)\}$, denote by $e_j$ the $j$-th standard basis vector of $\mathbb{Q}^{\rk(\Lambda)}$. Then for $M \in \mods\Lambda$, the \emph{dimension vector} of $M$ is given by
$$\underline{\dim}(M) = \sum_{j = 1}^{\rk(\Lambda)} e_j \cdot \dim_k\Hom_\Lambda(P_j,M).$$

We now wish to describe the $c$-vectors and $c$-matrices of an elementary algebra. Rather than give the original definition of Fu~\cite{fu}, we use the following characterization from \cite{treffinger}.

\begin{defthm}\label{defthm:cVect}
    Let $\Lambda$ be an elementary algebra.
    \begin{enumerate}
        \item Let $\D \sqcup \U[1]$ be a 2-term simple minded collection in $\mods\Lambda$. Then we say $$\{\underline{\dim}(S):S \in \D\} \cup \{-\underline{\dim}(T): T \in \U\}$$
        is a \emph{$c$-matrix} of $\Lambda$.
        \item A vector $v \in \mathbb{Q}^{\rk(\Lambda)}$ is called a \emph{$c$-vector} if there exists a $c$-matrix $\mathcal{M}$ with $v \in \mathcal{M}$.
    \end{enumerate}
\end{defthm}

We emphasize that under this formulation, a $c$-matrix is a collection of vectors, rather than an actual matrix.

\begin{rem}\label{rem:IThereditary}\
\begin{enumerate}
    \item The name $c$-vector comes from the relationship between representation theory and cluster algebras. Indeed, if $Q$ is an acyclic quiver, then the $c$-vectors of the path algebra $KQ$ coincide with the $c$-vectors of the cluster algebra of type $Q$.
    
    \item Ingalls and Thomas showed in \cite{IngTho} that the $c$-vectors of representation-finite hereditary algebras are characterized by pairwise conditions. This, and the fact that $c$-vectors and bricks are in bijection for such algebras, led to Igusa and Todorov's proof that these algebras satisfy the pairwise 2-simple minded completability property in \cite{IT_exceptional}.
\end{enumerate}
\end{rem}

\subsection{Torsion classes}
Now let $\T,\F$ be (full, closed under isomorphism) subcategories of $\mods\Lambda$. Then the pair $(\T,\F)$ is called a \emph{torsion pair} if each of the following holds:
\begin{enumerate}
\item $\Hom_\Lambda (M,N) = 0$ for all $M\in \T$ and $N\in \F$.
\item $\Hom_\Lambda(M,-)|_\F  = 0$ implies that $M\in \T$.
\item $\Hom_\Lambda (-,N)|_\T = 0$ implies that $N\in \F$.
\end{enumerate}
For a torsion pair $(\T,\F)$, we say that $\T$ is a \emph{torsion class}, and $\F$ is a \emph{torsion free class}.
Let $\T^\perp$ denote the subcategory $\{X\in \mods \Lambda: \Hom_\Lambda (T,X)=0 \text{ for all $T\in \T$}\}$ and define $\lperp{\F}$ analogously.
Note that $\T^\perp = \F$ and $\T = \lperp{\F}$ when $(\T,\F)$ is a torsion pair.
In particular, $\T \cap \F = 0$. It is well known that a subcategory is a torsion class if and only if it is closed under isomorphisms, quotients and extensions. Dually, a subcategory is a torsion free class if and only if it is closed under subobjects and extensions.
See \cite[Proposition~VI.1.4]{ASS}.

We partially order the set of all torsion classes of $\mods \Lambda$ by inclusion (i.e. $\T \le \T'$ provided that $\T \subseteq \T'$) and we denote this poset $\tors \Lambda$.
\begin{define}\label{cover_relation}
A \emph{cover relation} in a poset is a pair $\T\covered \T'$ satisfying $\T< \T'$, and for all $\S$ such that $\T<\S\le \T'$, we have $\S=\T'$. The cover relation $\T\covered \T'$ is an \emph{upper} cover relation for $\T$ and a \emph{lower} cover relation for $\T'$.
\end{define}
\begin{rem}
It is well known that the poset $\tors \Lambda$ is a \emph{lattice}, in which the smallest upper bound for torsion classes $\T$ and $\T'$ is $\Filt(\T\cup \T')$ and the greatest lower bound is $\T\cap \T'$.
See \cite{GM_lattice,IRTT}.
In this paper we do not use the lattice properties of $\tors \Lambda$.
\end{rem}

\begin{rem}
The torsion free classes of $\mods \Lambda$ can also be partially ordered by inclusion.
Indeed, this poset is anti-isomorphic to $\tors\Lambda$.
So there is a cover relation torsion classes $\T\covered \T'$ if and only if $(\T')^\perp\covered \T^\perp$.
\end{rem}
We restrict our attention to algebras for which the lattice $\tors\Lambda$ is finite. (Finiteness allows us to reframe completability in terms of mutation. See Theorem~\ref{thm:mutation}.)
By \cite{DIJ} the following are equivalent:
\begin{enumerate}
    \item $\tors \Lambda$ is finite.
    \item There are only finitely many (isoclasses of) bricks in $\mods\Lambda$.
    \item There are only finitely many support $\tau$-tilting pairs for $\Lambda$; that is, $\Lambda$ is \emph{$\tau$-tilting finite}.
    \item Every torsion class in $\mods\Lambda$ is \emph{functorially finite}.
\end{enumerate}
Examples of algebras satisfying these properties include representation-finite algebras and preprojective algebras of Dynkin type (see \cite[Theorem~0.2]{mizuno}).
\begin{ex}
The poset of torsion classes $\tors \Lambda$, for $\Lambda$ the hereditary algebra $KA_2$, is displayed in Figure~\ref{fig_brick_labeling}.
\end{ex}

\subsection{Brick labeling}
The goal of this section is to introduce a certain labeling of the cover relations of $\tors \Lambda$ which encodes all of the 2-term simple minded collections. We will follow the construction in \cite{BCZ}, although we remark that brick labeling is also defined independently in \cite{asai,BST,DIRRT}.

Given a torsion class $\T$ we would like to label each upper cover relation $\T\covered \T'$ by a module $M$ which is ``minimal'' such that closing $\T\cup M$ under extensions produces $\T'$.
The following definition characterizes such modules.

\begin{define}\label{def:minextending}\cite[Definition 1.1 and Definition 2.10]{BCZ}
    Let $\T$ be a torsion class. A module $M$ is a \emph{minimal extending module} for $\T$ if:
    \begin{enumerate}
        \item Every proper factor of $M$ is in $\mathcal{T}$.
        \item If $M \hookrightarrow X \twoheadrightarrow T$ is a nonsplit exact sequence and $T \in \T$, then $X \in \T$.
        \item $\Hom_\Lambda(\T,M) = 0$.
    \end{enumerate}
    Dually, let $\F$ be a torsion free class. Then $M$ is a \emph{minimal coextending module} for $\F$ if:
    \begin{enumerate}
        \item Every proper submodule of $M$ is in $\F$.
        \item If $F \hookrightarrow X \twoheadrightarrow M$ is a nonsplit exact sequence and $F \in \F$, then $X \in \F$.
        \item $\Hom_\Lambda(M,\F) = 0$.
    \end{enumerate}
\end{define}

The following lemma is well known (\cite[Lemma~2.1 and Lemma~2.2]{BCZ}).
\begin{lem}
Suppose that $\S$ is a set of indecomposable modules satisfying: If $M\in \S$ and $N$ is an indecomposable factor of $M$, then $N\in \S$.
Then $\Filt(\S)$ is a torsion class.
In particular, if $M$ is a minimal extending module of $\T$, then $\Filt(\T\cup M)$ is a torsion class.
\end{lem}

The following is quoted from \cite[Theorem~1.2 and Theorem~1.3]{BCZ}.
\begin{thm}\label{min_ext}
Let $\Lambda$ be an arbitrary finite-dimensional algebra.
\begin{enumerate}
\item For each cover relation $\T\covered \T'$ there is a unique (up to isomorphism) minimal extending module $M$ for $\T$.
Moreover, $\T'=\Filt(\T\cup M)$.
\item A module $M$ is a minimal extending module for some torsion class $\T$ if and only if $M$ is a brick.
\end{enumerate}
\end{thm}

\begin{rem}
For each cover relation $\F\covered \F'$ in the lattice of torsion free classes, there is a minimal coextending module $M$ such that $\Filt(\F \cup M) = \F'$.
Moreover, $M$ is a minimal extending module for the cover relation $\T\covered \T'$ if and only if it is a minimal \emph{coextending} module for ${\T'}^\perp\covered \T^\perp$.
\end{rem}

As in Figure~\ref{fig_brick_labeling}, we visualize labeling each cover relation of $\tors \Lambda$ with the corresponding minimal extending module.
Then we can read off every 2-term simple minded collection for $\Db(\mods\Lambda)$ as the sets of bricks labeling the upper and the lower cover relations for a given torsion class.
To make this precise, we write $\U(\T)$ for the set of minimal extending modules of $\T$ and $\D(\T)$ for the set of minimal coextending modules of $\T^\perp$.
Note that the bricks in $\D(\T)$ label the lower cover relations of $\T$ in $\tors \Lambda$.

\begin{ex}
Let $\Lambda$ be the hereditary algebra $KA_2$ from Figure~\ref{fig_brick_labeling}, and let $\T=\add(P_1, S_1)$.
Then $\U(\T)=S_2$ and $\D(\T)=P_1$.
\end{ex}

\begin{thm}\label{torsion_to_simple} Let $\Lambda$ be an arbitrary finite-dimensional algebra.

    \begin{enumerate}
        \item Let $\U$ be a collection of (isoclasses of) modules in $\mods\Lambda$. Then there exists a torsion pair $(\T,\F)$ for which $\U = \U(\T)$ if and only if $\U$ is a semibrick. Moreover, if $\Lambda$ is $\tau$-tilting finite, then the map $(\T,\F)\mapsto \U(\T)$ is a bijection from the set of torsion pairs to $\sbrick \Lambda$. The inverse map is $\U \mapsto (^\perp\U,\Filt\Sub(\U))$.
        \item  Let $\D$ be a collection of (isoclasses of) modules in $\mods\Lambda$. Then there exists a torsion pair $(\T,\F)$ for which $\D = \D(\T)$ if and only if $\D$ is a semibrick. Moreover, if $\Lambda$ is $\tau$-tilting finite, then the map $(\T,\F)\mapsto\mapsto \D(\T)$ is a bijection from the set of torsion pairs to $\sbrick \Lambda$. The inverse map is $\D \mapsto (\Filt\Fac(\D),\D^\perp)$.
        \item If $\Lambda$ is $\tau$-tilting finite, then the map $(\T,\F) \mapsto \D(\T) \sqcup \U(\T)[1]$ is a bijection from the set of torsion pairs to the set of 2-term simple minded collections for $\Lambda$. The inverse map is given by $\D\sqcup\U[1]\mapsto(\Filt\Fac(\D),\Filt\Sub(\U)) = (\lperp{\U},\D^\perp).$
    \end{enumerate}
\end{thm}
\begin{proof}
The first and second items follow from \cite[Theorem~1.8]{BCZ}.
The third item is from\linebreak \cite[Theorem~3.3]{asai}.

\end{proof}
\begin{rem}\label{rem:torsion_to_simple}
Suppose that $\Lambda$ is $\tau$-tilting finite. Then as an immediate consequence of Theorem~\ref{torsion_to_simple}, given a semibrick pair $\D\sqcup \U[1]$, there exists unique semibricks $\D'$ and $\U'$ such that $\D'\sqcup \U[1]$ and $\D\sqcup \U'[1]$ are both 2-term simple minded collections.
\end{rem}

The following will be used frequently throughout this paper.

\begin{prop}\label{prop:maximal}
    Let $\Lambda$ be an arbitrary finite-dimensional algebra and let $\D\sqcup\U[1]$ be a 2-term simple minded collection for $\Lambda$. Then
    \begin{enumerate}
        \item $\D\sqcup \U[1]$ is maximal in the sense that it is not properly contained in any semibrick pair.
        \item $|\D| + |\U| = \rk(\Lambda)$.
    \end{enumerate}
\end{prop}

\begin{proof}
    Item (2) is \cite[Corollary 5.5]{KY}. To prove item (1), we note that we have a torsion pair $(\Filt\Fac(\D),\Filt\Sub(\U))$, even if $\Lambda$ is not $\tau$-tilting finite (see \cite[Theorem 3.3]{asai}). Now suppose for a contradiction that $\D\sqcup\U[1]$ is properly contained in $\D'\sqcup\U'[1]$ and let $S \in (\D'\setminus\D)\cup(\U'\setminus\U)$. Since $\D'\sqcup\U'[1]$ is a semibrick pair, we know that $\Hom_\Lambda(\D,S) = 0 = \Hom_\Lambda(S,\U)$. However, this means that $S$ has neither a nonzero torsion part nor a nonzero torsion free part with respect to the torsion pair $(\Filt\Fac(\D),\Filt\Sub(\U))$. We conclude that $S = 0$, a contradiction. 
\end{proof}

\begin{rem}
    We note that Theorem \ref{torsion_to_simple} and Proposition \ref{prop:maximal} imply that, when $\Lambda$ is $\tau$-tilting finite, $|\D| \leq \rk(\Lambda)$ for all semibricks $\D$. Moreover, if $|\D| = \rk(\Lambda)$, then $\D$ is the collection of simple modules.
\end{rem}

We conclude with the following consequence of Proposition \ref{prop:maximal}.

\begin{cor}\label{cor:2vertex}
    Let $\Lambda$ be a  $\tau$-tilting finite algebra with $\rk(\Lambda) \leq 2$. Then $\Lambda$ has the 2-simple minded pairwise completability property.
\end{cor}

\begin{proof}
    We observe that if $\rk(\Lambda) = 1$ then $\mods\Lambda$ contains only a single brick (up to isomorphism). Thus there is nothing to show.
    
    Now suppose $\rk(\Lambda) = 2$ and let $\D\sqcup\U[1]$ be a pairwise completable semibrick pair. If $\D = \emptyset$ or $\U = \emptyset$, then $\D\sqcup\U[1]$ is completable by Remark \ref{rem:torsion_to_simple}. Otherwise, for all $S \in \D$ and $T \in \U$, Proposition \ref{prop:maximal}(2) implies that $S \sqcup T[1]$ is a 2-term simple minded collection. Proposition \ref{prop:maximal}(1) then implies that $\D = S$ and $\U = T$, so $\D\sqcup\U[1]$ is completable.
\end{proof}

\subsection{Mutation and completability}
In this section we recall the definitions of mutation and the notion of mutation compatibility for semibrick pairs.
We will use these definitions to determine when a semibrick pair is completable.

First we recall the notion of a left-approximation, following \cite{Kleiner}.
Let $\S$ be a subcategory of $\mods \Lambda$ that is closed under direct sums and extensions.
Given a module $T\in \mods \Lambda$, a morphism $f:T\to S_T$ with $S_T\in \S$ is said to be a \emph{left $\S$-approximation} if for every morphism $j:T\to S$ with $S\in\S$ there exists $h:S_T\to S$ such that $j=hf$.
We say that $f$ is a \emph{minimal left-approximation} if $f$ is a left minimal morphism; i.e. if for every endomorphism $s: S_T \to S_T$ satisfying $sf= f$ is an isomorphism.
The notion of a right $\S$-approximation is defined dually.
\begin{rem}
Suppose that $f:T\to S_T$ is a left $\S$-approximation and $S_T$ is a brick.
Then every nonzero endomorphism $s:S_T\to S_T$ is an isomorphism.
Therefore, $f$ is minimal.
\end{rem}

\begin{define}\cite[Definition 3.2]{HI_smc}\label{mutation}
    Let $\Lambda$ be $\tau$-tilting finite and let $\X = \D \sqcup \U[1]$ be a semibrick pair.
    \begin{enumerate}
        \item Let $S \in \D$. If for all $T \in \U$ there exists a left minimal ($\Filt S$)-approximation $g_{ST}^+:T \rightarrow S_T$ which is either mono or epi, we say $\X$ is \emph{singly left mutation compatible at $S$}. In this case, there is a new semibrick pair $\X' = \mu^+_{S}(\X)$, called the \emph{left mutation} of $\X$ at $S$, given as follows:
    \begin{enumerate}
        \item $\mu^+_S(S)=S[1]$.
        \item For all $T\neq S \in \D$, we have $\mu^+_S(T)$ is equal to $\cone(g_{ST}^+)$, where $g_{ST}^+:T[-1]\rightarrow S_T$ is a left minimal $(\Filt S$)-approximation.
        In particular, there is an exact sequence \linebreak $S_T \hookrightarrow \mu^+_S(T) \twoheadrightarrow T$.
        \item For all $T \in \U$, we have $\mu^+_S(T[1])$ is equal to $\cone(g_{ST}^+)$, where $g_{ST^+}:T\rightarrow S$ is a left minimal ($\Filt S$)-approximation.
        In particular, if $g_{ST^+}$ is mono, then $\mu^+_S(T[1])$ is $\coker(g_{ST}^+)$, and if $g_{ST}^+$ is epi, then $\mu^+_S(T[1])$ is $\ker(g_{ST}^+[1])$.
    \end{enumerate}
    \item Let $S \in \U$. If for all $T \in \D$ there exists a right minimal ($\Filt S$)-approximation $g_{ST}^-:S_T \rightarrow T$ which is either mono or epi, we say $\X$ is \emph{singly right mutation compatible at $S$}. In this case, there is a new semibrick pair $\X' = \mu^-_{S}(\X)$, called the \emph{right mutation} of $\X$ at $S$, given as follows:
    \begin{enumerate}
        \item $\mu^-_S([1]) = S.$
        \item For all $T\neq S \in \U$, we have $\mu_S^-(T[1])$ is equal to $\cocone(g_{ST}^-)[1]$, where $g_{ST}^-:S_T\rightarrow T[1]$ is a right minimal ($\Filt S$)-approximation.
        In particular, there is an exact sequence $T \hookrightarrow \mu^-_S(T)[-1] \twoheadrightarrow S_T$.
        \item For all $T \in \D$, we have $\mu_S^-(T)$ is equal to $\cocone(g_{ST}^-)[1]$, where $g_{ST}^-:S_T\rightarrow T$ is a right minimal ($\Filt S$)-approximation.
        In particular, if $g_{ST}^-$ is mono then $\mu^-_S(T)$ is $\coker(g_{ST}^-)$, and if $g_{ST}^-$ is epi then $\mu^-_S(T)$ is $\ker(g_{ST}^-)[1]$.
    \end{enumerate}
    \item We say $\X$ is \emph{singly left mutation compatible} if $\X$ is singly left mutation compatible at every $S \in \D$. Likewise, we say $\X$ is \emph{singly right mutation compatible} if $\X$ is singly right mutation compatible at every $S \in \U$.
    \end{enumerate}
\end{define}
We remark that left and right mutation are dual in the sense that if $\X$ is singly left mutation compatible at $S$ then $\mu^+_S(\X)$ is singly right mutation compatible at $S$ and $\mu^-_S\circ\mu^+_S(\X) = \X$.
The same sentence is true if we switch ``left'' and ``right'', and in this case $\mu^+_S\circ\mu^-_S(\X) = \X$.

\begin{rem}\label{rem:mutation}
    The mutation formulas in Definition \ref{mutation} are based on the formulas for the mutation of simple minded collections from \cite{KY}, specialized to the 2-term case. It is shown in \cite{BY} that these formulas send 2-term simple minded collections to 2-term simple minded collections. In particular, suppose $\D\sqcup \U[1]$ is a 2-term simple minded collection and let $S \in \D$ and $T \in \U$. Then since both $\mu^+_S(\D\sqcup \U[1])$ and $\mu^-_T(\D\sqcup \U[1])$ are 2-term simple minded collections, a left minimal $(\Filt S)$-approximation $g_{ST}^+:T\rightarrow S_T$ and a right minimal $(\Filt T)$-approximation $g_{TS}^-:T_S \rightarrow S$ are each either mono or epi. This means every pairwise completable semibrick pair is both singly left and singly right mutation compatible. Moreover, if $\D'\sqcup \U'[1]$ is a completable semibrick pair and $S' \in \D'$  (resp. $T' \in \U'$), then $\mu_{S'}^+(\D'\sqcup \U'[1])$ (resp. $\mu_{T'}^-(D'\sqcup \U'[1])$) is also completable.
\end{rem}

\begin{rem}
Suppose that $\X=\D\sqcup\U[1]$ is a 2-term simple minded collection.
Recall from Theorem~\ref{torsion_to_simple} that there exists a torsion class $\T$ such that $\U=\U(\T)$ and $\D=\D(\T)$.
(This means $\U$ is the set of minimal extending modules for $\T$ and $\D$ is the set of minimal coextending modules for $\T^\perp$).
In particular, each brick $S\in \D$ labels a lower cover relation $\T\covers \T'$ in $\tors \Lambda$.
The new semibrick pair $\X' = \mu^+_{S}(\X)$ is also a 2-term simple minded collection, and it corresponds to $\D(\T')\sqcup \U(\T')$.
Therefore, left mutation at $S$ corresponds to moving down by the cover relation $\T\covers \T'$.
\end{rem}

\begin{define}\cite[Definition 3.7]{HI_smc}\label{def:mut_compat}
Let $\Lambda$ be $\tau$-tilting finite and let $\X = \D \sqcup \U[1]$ be a semibrick pair.
We say $\X$ is \emph{mutation compatible} if one of the following hold.
\begin{enumerate}
    \item $\X = \U[1]$; that is, $\D = \emptyset$.
    \item $\X = \D$; that is, $\U = \emptyset$.
    \item $\X$ is singly left mutation compatible and there exist a sequence $\mu^+_{S_1},\ldots,\mu^+_{S_k}$ of left mutations and a semibrick $\U'$ so that  $\mu^+_{S_\ell}\circ\cdots\circ\mu^+_{S_1}(\X)$ is singly left mutation compatible for all $1 \leq \ell \leq k$ and $\mu^+_{S_k}\circ\cdots\circ\mu^+_{S_1}(\X) = \U'[1]$.
\end{enumerate}
\end{define}

\begin{rem}
    In Section~\ref{sec:rank3}, we will discuss \emph{wide subcategories} of $\mods\Lambda$. In particular, it will be a consequence of Lemma~\ref{lem:wide} that the semibrick $\U'$ in Definition~\ref{def:mut_compat}(3) will consist precisely of the simple objects in the smallest wide subcategory containing the bricks in $\D \cup \U$. In particular, we will have $\D \cup \U \subseteq \Filt(\U')$. For example, let $A_3 = (1\rightarrow 2 \rightarrow 3)$. Then $\X := S_1 \sqcup P_1[1]$ is a semibrick pair for $\mods KA_3$. Mutating at $S_1$, we obtain $\U' := \mu_{S_1,\X}^+(\X) = P_2[1] \oplus S_1[1]$.
\end{rem}

The following is \cite[Theorem 3.9]{HI_smc}. We include a proof here for completeness.
\begin{thm}\label{thm:mutation}
Let $\Lambda$ be $\tau$-tilting finite, and let $\X$ be a semibrick pair.
Then $\X$ is completable if and only if it is mutation compatible.
\end{thm}
\begin{proof}
Suppose that $\X$ is completable.
Theorem~\ref{torsion_to_simple} implies that there is a torsion class $\T$ such that $\U=\U(\T)$ and $\D=\D(\T)$.
Consider any chain of cover relations ending at the zero torion class:
\[\T\covers \T_1 \covers\cdots \covers \T_k = 0.\]
This chain in $\tors \Lambda$ corresponds to a sequence of left-mutations $\mu^+_{S_k}\circ\cdots\circ\mu^+_{S_1}(\X) = \U'[1]$, where $\U'$ is the set of all simple modules for $\Lambda$.
Such a chain exists because $\tors \Lambda$ is finite.

Conversely, suppose that $\X$ is mutation compatible, and there is a sequence of left-mutations which take $\X$ to a semibrick pair $\X'=\U'[1]$.
By the first item of Theorem~\ref{torsion_to_simple}, there is a torsion class $\T'$ whose extending modules are precisely those in the set $\{M: M[1]\in U'[1]\}.$
That is, $\U(\T')=\U'$.
The third item of Theorem~\ref{torsion_to_simple} implies that $\D(\T')\sqcup \U'[1]$ is a 2-term simple minded collection.
Now follow the sequence of right mutations: $$\mu^-_{S_1}\circ\cdots\circ\mu^-_{S_k}(\D(\T')\sqcup \U'[1]).$$
The resulting 2-term simple minded collection contains $\X$.
\end{proof}

\begin{rem}
\label{finite}
\normalfont
Theorem~\ref{thm:mutation} allows us to test the completability of a semibrick pair by performing a series of mutations.
Note that the theorem depends on the fact that $\tors \Lambda$ is finite.
\end{rem}

The following is an immediate consequence of Theorem \ref{thm:mutation}.

\begin{cor}\label{cor:mutation}
    Let $\Lambda$ be $\tau$-tilting finite. Then the following are equivalent.
    \begin{enumerate}
        \item $\Lambda$ has the pairwise 2-simple minded completability property.
        \item Every pairwise completable semibrick pair is mutation compatible.
        \item For all pairwise completable semibrick pairs $\D\sqcup \U[1]$ with $\D \neq \emptyset$ and for all $S \in \D$, the semibrick pair $\mu^+_S(\D\sqcup \U[1])$ is pairwise completable.
    \end{enumerate}
\end{cor}

Moreover, we have the following.

\begin{cor}\label{cor:mutationcompatiblepassdown}
	Let $\X = \D \sqcup \U[1]$ be mutation compatible. Then for $S \in \D$, the semibrick pair $\mu^+_{S,\X}(\X)$ is mutation compatible. Likewise, for $T \in \U$, the semibrick pair $\mu^-_{T,\X}(\X)$ is mutation compatible.
\end{cor}

\begin{proof}
	If $\X = \D \sqcup \U[1]$ is contained in the 2-simple minded collection $\Y$, then $\mu^+_{S,\X}(\X)$ is contained in the 2-simple minded collection $\mu^+_{S,\Y}(\Y)$. The result for right mutation is completely analogous.
\end{proof}

Although singly left (or right) mutation compatibility is often straightforward to verify, mutation compatibility is in general much more opaque. Even when a semibrick pair consists of a pair of bricks $\X = S \sqcup T[1]$ it is often non-trivial to determine whether $\X$ is mutation compatible. See e.g. \cite[Remark~3.13]{HI_smc}.

\section{Semibrick pairs of rank 3}\label{sec:rank3}
   In this section, we first give an alternative formulation of the pairwise 2-simple minded completability property in terms of semibrick pairs of rank 3. We then consider semibrick pairs of full rank, i.e., for which $|\D| + |\U| = \rk(\Lambda)$. Finally, we use these results to reformulate the pairwise 2-simple minded completability property in terms of wide subcategories.  (We recall the definition of wide subcategories before Remark~\ref{rem:wide} below.)
   
We begin with the following.

\begin{prop}\label{prop:size 3 implies size n}
    Let $\Lambda$ be a $\tau$-tilting finite algebra. Then the following are equivalent.
    \begin{enumerate}
        \item $\Lambda$ has the pairwise 2-simple minded completability property.
        \item Every pairwise completable semibrick pair $\D\sqcup\U[1]$ with $|\D| + |\U| = 3$ is completable.
    \end{enumerate}
\end{prop}

\begin{proof}
    We need only show that (2) implies (1). Let $\D\sqcup\U[1]$ be a pairwise completable semibrick pair and let $S \in \D$. By Corollary \ref{cor:mutation}, it suffices to show that $\D'\sqcup \U'[1] := \mu_S^+(\D\sqcup\U[1])$ is pairwise completable.
    
    Let $S' \in \D'$ and $T' \in \U'$. If $T' = S$, then either there exists $T \in \D$ so that $S'\sqcup T'[1] = \mu^+_S(S\sqcup T)$ or there exists $T \in \U$ so that $S'\sqcup T'[1] = \mu^+_S(S\sqcup T[1])$. In either case, $S\sqcup T$ (resp. $S\sqcup T[1]$) is completable by assumption. Remark \ref{rem:mutation} then implies that $S'\sqcup T'[1]$ is completable.
    
    Otherwise $T' \neq S$ and either there exists $R\in \D$ and $T \in \U$ so that $S'\sqcup T'[1]\sqcup S[1] = \mu^+_S(S\sqcup R\sqcup T[1])$ or there exists $R,T \in \U$ so that $S'\sqcup T'[1]\sqcup S[1] = \mu^+_S(S\sqcup R[1]\sqcup T[1])$. In either case, $S\sqcup R\sqcup T[1]$ (resp. $S\sqcup R[1]\sqcup T[1]$) is completable by the assumption of (2). Remark \ref{rem:mutation} then implies that $S'\sqcup T'[1]$ is completable.
\end{proof}

Before turning to semibrick pairs of full rank, we give an alternative description of the subcategories $\Filt\Fac(\D)$.

\begin{lem}\cite[Lemma 2.2]{BCZ}\label{lem:filtfac}
    Let $\mathcal{I}$ be a class of indecomposable modules.
    \begin{enumerate}
        \item If $\mathcal{I}$ is closed under taking indecomposable direct summands of factors, then $\Filt(\mathcal{I})$ is closed under factors.
        \item If $\mathcal{I}$ is closed under taking indecomposable direct summands of submodules, then $\Filt(\mathcal{I})$ is closed under submodules.
    \end{enumerate}
\end{lem}

\begin{cor}\label{cor:littleFac}
    Let $\mathcal{S}$ be a semibrick.
    \begin{enumerate}
        \item Let $\mathcal{S}^-$ be the class of indecomposable factors of the bricks in $\mathcal{S}$. Then $\Filt\Fac(\mathcal{S}) = \Filt(\mathcal{S}^-)$.
        \item Let $\mathcal{S}_-$ be the class of indecomposable submodules of the bricks in $\mathcal{S}$. Then $\Filt\Sub(\mathcal{S}) = \Filt(\mathcal{S}_-)$.
    \end{enumerate}
\end{cor}

\begin{proof}
    We prove (1) as the proof of (2) is nearly identical. Let $X \in \Fac(\mathcal{S)}$. Then there exists a positive integer $m$ and a epimorphism $\mathcal{S}^m \twoheadrightarrow X$. Now observe that $\mathcal{S}^m \in \Filt(\mathcal{S}^-)$, which is closed under factors by Lemma~\ref{lem:filtfac}. Thus $X \in \Filt(\mathcal{S}^-)$. This proves the result.
\end{proof}

We now use Corollary \ref{cor:littleFac} to prove the following general result, which can also be found in \cite[Lemma 2.7(1)]{asai}. Recall from Remark \ref{rem:torsion_to_simple} that given an arbitrary semibrick pair $\D\sqcup\U[1]$, there exist unique semibricks $\D'$ and $\U'$ so that $\D'\sqcup \U[1]$ and $\D\sqcup \U'[1]$ are 2-term simple minded collections.

\begin{prop}\label{prop:exists surjection}
    Let $\Lambda$ be a $\tau$-tilting finite algebra and let $\X = \D\sqcup\U[1]$ be a semibrick pair.
    \begin{enumerate}
        \item Let $\D'$ be the unique semibrick for which $\D'\sqcup \U[1]$ is a 2-term simple minded collection. Then for every $S \in \D$ there exists $R \in \D'$ which is a quotient of $S$.
        \item Let $\U'$ be the unique semibrick for which $\D \sqcup \U'[1]$ is a 2-term simple minded collection. Then for every $T \in \U$, there exists $U \in \U'$ which is a submodule of $T$.
    \end{enumerate}
\end{prop}

\begin{proof}
    (1) Let $S \in \D$. We first observe that since $\Hom_\Lambda(S,\U) = 0$, we have $S$ is in the torsion class associated with $\D'\sqcup \U[1]$.
    By Theorem~\ref{torsion_to_simple}, this implies $S\in \Filt\Fac(\D')$. By Corollary~\ref{cor:littleFac}, this means there is a filtration
     $$S = S_k \supset S_{k-1} \supset \cdots \supset S_0 = 0$$
    so that each $S_i/S_{i-1}$ is an (indecomposable) factor of some $R_i \in \D'.$ Let $q: R_k \rightarrow S/S_{k-1}$ be the quotient map. We claim that $q$ is an isomorphism and thus $R_k$ is a quotient of $S$. To see this, we apply the functor $\Hom_\Lambda(-,\U)$ to the short exact sequence
    $$\ker q \hookrightarrow R_k \twoheadrightarrow S/S_{k-1}.$$
    This gives the exact sequence
    $$0 = \Hom_\Lambda(R_k,\U) \rightarrow \Hom_\Lambda(\ker q,\U) \rightarrow \Ext_\Lambda^1(S/S_{k-1},\U) \rightarrow \Ext_\Lambda^1(R_k,\U) = 0.$$
    Thus $\Hom_\Lambda(\ker q,\U) = 0$ if and only if $\Ext_\Lambda^1(S/S_{k-1},\U)=0$. To see that this is the case, we apply the functor $\Hom_\Lambda(-,\U)$ to the short exact sequence
    $$S_{k-1} \hookrightarrow S \twoheadrightarrow S/S_{k-1},$$
    which gives the exact sequence
    $$0 = \Hom_\Lambda(S_{k-1},\U) \rightarrow \Ext_\Lambda^1(S/S_{k-1},\U) \rightarrow \Ext_\Lambda^1(S,\U) = 0,$$
    where the first term is zero since $S_{k-1} \in \Filt\Fac(\D')$. This means $\Hom_\Lambda(\ker q,\U) = 0$; however, if $\ker q \subsetneq R_k$, then $\ker q \in \Filt\Sub(\U)$ by the definition of a minimal coextending module. We conclude that $\ker q = 0$, that is, $q$ is an isomorphism.
    
   (2) Let $T \in \U$. We first observe that since $\Hom_\Lambda(\D,T) = 0$, we have $T$ in the torsion free class $(\D', \U'[1])$.
   By Theorem~\ref{torsion_to_simple}, this implies that $T\in \Filt\Sub(\U')$. By Corollary \ref{cor:littleFac}, this means there is a filtration
    $$T = T_k \supset T_{k-1} \supset \cdots \supset T_0 = 0$$
    so that each $T_i/T_{i-1}$ is an (indecomposable) submodule of some $U_i \in \U'$. Let $\iota: U_1 \rightarrow T_1$ be the inclusion map. We claim that $\iota$ is an isomorphism and thus $U_1$ is a submodule of $T$. 
    
    We first show that $\Ext_\Lambda^1(\D,T_{i}) = 0$ for all $i \in \{1,\ldots,k\}$ using a backward induction argument on $i$. We already know this for $i = k$. Thus let $1\leq i < k$ and assume $\Ext_\Lambda^1(\D,T_{i+1}) = 0$. Applying $\Hom_\Lambda(\D,-)$ to the short exact sequence
    $$T_{i} \hookrightarrow T_{i+1} \twoheadrightarrow T_{i+1}/T_{i}$$
    then gives an exact sequence
    $$0 = \Hom_\Lambda(\D,T_{i+1}/T_{i})\rightarrow \Ext_\Lambda^1(\D,T_{i}) \rightarrow \Ext_\Lambda^1(\D,T_{i+1}) = 0,$$
    where the first term is zero since $T_{i+1}/T_i \in \Filt\Sub(\U')$. We conclude that $\Ext_\Lambda^1(\D,T_i) = 0$ for all $i$. In particular, $\Ext_\Lambda^1(\D,T_1) = 0$.
    
    We now apply the functor $\Hom_\Lambda(\D,-)$ to the short exact sequence
    $$T_1 \hookrightarrow U_1 \twoheadrightarrow \coker\iota,$$
    which gives us an exact sequence
    $$0 = \Hom_\Lambda(\D,U_1) \rightarrow \Hom_\Lambda(\D,\coker \iota) \rightarrow \Ext_\Lambda^1(\D,T_1) = 0.$$
    Thus we have $\Hom_\Lambda(\D,\coker\iota) = 0$. However, if $\coker\iota$ is a proper quotient of $U_1$, then $\coker\iota \in \Filt\Fac(\D)$ by the definition of a minimal extending module. We conclude that $\coker\iota = 0$, that is, $\iota$ is an isomorphism.
\end{proof}

As a consequence of Proposition \ref{prop:exists surjection}, we can prove that certain semibrick pairs of full rank are 2-term simple minded collections.

\begin{thm}\label{thm:no common quotients}
    Let $\Lambda$ be a $\tau$-tilting finite algebra and let $\D\sqcup\U[1]$ be a semibrick pair with $|\D| + |\U| = \rk(\Lambda)$.
    \begin{enumerate}
        \item Let $\D'$ be the unique semibrick for which $\D'\sqcup \U[1]$ is a 2-term simple minded collection. For each $S \in \D$, let $\D'(S) = \{R \in \D'|\exists S \twoheadrightarrow R\}$. If $\D'(S)\cap \D'(S') = \emptyset$ for all $S \ncong S' \in \D$, then $\D = \D'$. In particular, $\D \sqcup \U[1]$ is a 2-term simple minded collection.
        \item Let $\U'$ be the unique semibrick for which $\D\sqcup \U'[1]$ is a 2-term simple minded collection. For each $T \in \U$, let $\U'(T) = \{U \in \U'|\exists U \hookrightarrow T\}$. If $\U'(T)\cap \U'(T') = \emptyset$ for all $T \ncong T' \in \U$, then $\U = \U'$. In particular, $\D \sqcup \U[1]$ is a 2-term simple minded collection.
    \end{enumerate}
\end{thm}

\begin{proof}
    We prove only (1) as the proof of (2) is entirely analogous. By assumption, we know $|\D| = |\D'|$ and by Proposition \ref{prop:exists surjection}, we know each $\D'(S)$ is nonempty. Thus the assumption that $\D'(S)\cap \D'(S') = \emptyset$ for all $S \ncong S' \in \D$ implies that every $T \in \D'$ is contained in some $\D'(S)$. We conclude that
    $$\Filt\Fac(\D') \subseteq \Filt\Fac(\D) \subseteq \Filt\Fac(\D')$$
    and therefore $\D = \D'$ by Theorem 
    \ref{torsion_to_simple}(3).
\end{proof}

Note that if $\D$ (resp. $\U$) consists of a single brick then the hypotheses of Theorem \ref{thm:no common quotients}(1) (resp. Theorem \ref{thm:no common quotients}(2)) are satisfied automatically. This implies the following.

\begin{cor}\label{cor:single brick}
    Let $\Lambda$ be a $\tau$-tilting finite algebra and let $\D\sqcup\U[1]$ be a semibrick pair with $|\D| + |\U| = \rk(\Lambda)$. If $|\U| = 1$ or $|\D| = 1$, then $\D\sqcup\U[1]$ is a 2-term simple minded collection.
\end{cor}

As a consequence, we obtain our first main result.

\begin{thm}[Theorem~\ref{Main_A}]\label{thm:3 vertex}
    Let $\Lambda$ be a $\tau$-tilting finite algebra with $\rk(\Lambda) \leq 3$. Then $\Lambda$ has the pairwise 2-simple minded completability property.
\end{thm}

\begin{proof}
For $\rk(\Lambda) < 3$, this result is contained in Corollary \ref{cor:2vertex}, so suppose $\rk(\Lambda) = 3$. Then any pairwise completable semibrick pair $\D\sqcup \U[1]$ with $|\D| + |\U| = 3$ is a 2-term simple minded collection by Corollary \ref{cor:single brick}. The result then follows from Proposition \ref{prop:maximal}(1).
\end{proof}

\begin{rem}
Together, Proposition~\ref{prop:size 3 implies size n} and Theorem~\ref{thm:3 vertex} help to explain a pattern in the known counterexamples to the pairwise 2-simple minded completability property. Namely, the counterexamples to the pairwise 2-simple minded completability property in both \cite{HI_smc} and Section~\ref{sec:preprojective} of the present paper come from semibrick pairs $\D\sqcup \U[1]$ satisfying $|\D| + |\U| = 3 < \rk(\Lambda)$.
\end{rem}

We conclude this section by giving an alternative characterization of the pairwise 2-simple minded completability property in terms of \emph{wide subcategories}. Recall that a subcategory $\mathcal{W} \subseteq \mods\Lambda$ is called \emph{wide} if it is closed under extensions, kernels, and cokernels (that is, $\mathcal{W}$ is an exact embedded abelian category). A well-known result of Ringel \cite{ringel} implies that there is a bijection between $\sbrick\Lambda$ and the set of wide subcategories of $\mods\Lambda$ given by $\D \mapsto \Filt(\D)$. Moreover, the semibrick $\D$ consists of the simple objects in $\Filt(\D)$.

It is shown by Jasso \cite{jasso} (see also \cite[Thm. 4.12]{DIRRT}) that the wide subcategory $\mathcal{W} = \Filt(\D)$ is equivalent to $\mods\Lambda_\mathcal{W}$ for some $\tau$-tilting finite algebra $\Lambda_\mathcal{W}$ satisfying $\rk(\mathcal{W}):= \rk(\Lambda_\mathcal{W}) = |\D|.$ (Recall that we have assumed $\Lambda$ to be $\tau$-tilting finite.) This allows us to consider semibrick pairs and 2-term simple minded collections for $\mathcal{W}$.

\begin{rem}\label{rem:wide}
    Let $\mathcal{W} \subseteq \mods\Lambda$ be a wide subcategory and let $\D\sqcup \U[1]$ be a semibrick pair for $\mathcal{W}$. Then $\D\sqcup \U[1]$ is also a semibrick pair for $\mods\Lambda$. Moreover, the property of single left (right) mutation compatibility and the mutation formulas in Definition~\ref{mutation} are agnostic to whether $\D\sqcup \U[1]$ is considered as a semibrick pair for $\mathcal{W}$ or $\mods\Lambda$.
\end{rem}

\begin{rem}
    We recall from Section~\ref{sec:cVect} that each chamber in the wall-and-chamber structure of $\Lambda$ corresponds to some 2-term simple-minded collection. In this interpretation, mutation of 2-term simple-minded collections then corresponds to moving between adjacent chambers by crossing a wall. (See e.g. \cite[Corollary~4.3, Theorem~4.9]{BY} and \cite[Lemma~4.2]{BST}.) Now let $\mathcal{W} \subseteq \mods\Lambda$ a wide subcategory. It is then shown in \cite[Section~4.1]{asai_wall} that the wall-and-chamber structure of $\mods \Lambda_\mathcal{W}$ appears in the neighborhood of some cone in the wall-and-chamber structure of $\mods\Lambda$. (This cone is the positive span of the $g$-vectors which allow for $\mathcal{W}$ to be realized as the $\tau$-perpendicular subcategory of some $\tau$-rigid pair in the sense of \cite{jasso}.) In particular, this implies that mutation in $\mathcal{W}$ is ``governed" by mutation in $\mods\Lambda$. Our Remark~\ref{rem:wide} can be seen as the analogous statement rephrased in terms of semibrick pairs.
\end{rem}

\begin{lem}\label{lem:wide}
    Let $\Lambda$ be $\tau$-tilting finite and let $\D\sqcup\U[1]$ be semibrick pair. Let $\mathcal{W}$ be any wide subcategory of $\mods\Lambda$ containing $\D\sqcup \U[1]$. If $\D\sqcup\U[1]$ is singly left mutation compatible at $S \in \D$, denote $\mu^+_S(\D\sqcup\U[1]):= \D_S \sqcup \U_S[1]$. Then $\mathcal{W}$ contains $\D_S\sqcup \U_S$. Likewise, if $\D\sqcup\U[1]$ is singly right mutation compatible at $t \in \D$, denote $\mu^-_T(\D\sqcup\U[1]):= \D_T \sqcup \U_T[1]$. Then $\mathcal{W}$ contains $\D_T\sqcup \U_T$.
\end{lem}

\begin{proof}
    We observe from Definition \ref{mutation} that the bricks in both $\D_S\sqcup \U_S$ and $\D_T\sqcup \U_T$ can be formed from the bricks in $\D\sqcup \U$ by taking extensions, kernels, and cokernels. The result is then immediate from the definition of a wide subcategory.
\end{proof}

We now note that if $\D\sqcup \U[1]$ is a semibrick pair, then the ``smallest wide subcategory'' containing $\D\sqcup \U$ is a well-defined notion. Indeed, the intersection of arbitrarily many wide subcategories is again a wide subcategory and $\D\sqcup \U \subseteq \mods\Lambda$, which is wide in itself.
    
    \begin{thm}\label{thm:wide}
        Let $\Lambda$ be a $\tau$-tilting finite algebra and let $\D\sqcup \U[1]$ be a semibrick pair with $|\D| + |\U| \leq 3.$ Then the following are equivalent.
        \begin{enumerate}
            \item The smallest wide subcategory containing $\D\sqcup\U$ has rank $|\D| + |\U|$.
            \item $\D \sqcup \U[1]$ is mutation compatible.
        \end{enumerate}
    \end{thm}
    
    \begin{proof}
    First let $\mathcal{W}$ be the smallest wide subcategory containing $\D\sqcup \U$ and suppose $\rk(\mathcal{W}) = |\D| + |\U|$. Since $|\D| + |\U| \leq 3$, Theorem~\ref{thm:3 vertex} implies that $\D\sqcup \U[1]$ is a 2-term simple minded collection for $\mathcal{W}$. By Theorem \ref{thm:mutation} and Remark \ref{rem:wide}, this implies that $\D\sqcup\U[1]$ is mutation compatible.
    
    Now suppose $\D\sqcup \U[1]$ is mutation compatible. If $\U = 0$, then the smallest wide subcategory containing $\D$ is $\Filt(\D)$ and we are done. Otherwise, there exists a semibrick $\U'$ and a sequence of left mutations transforming $\D\sqcup \U[1]$ into $\U'[1]$. By Lemma \ref{lem:wide}, this implies that $\Filt(\U')$ is the smallest wide subcategory containing $\D\sqcup \U$. Since mutation preserves the size of a semibrick pair, this proves the result.
    \end{proof}

By combing Theorem~\ref{thm:wide} with Corollary~\ref{cor:mutation}, we obtain the following.
    
    \begin{cor}
        Let $\Lambda$ be $\tau$-tilting finite. Then the following are equivalent.
        \begin{enumerate}
            \item For every semibrick pair $\D\sqcup\U[1]$ with $|\D| + |\U| = 3$, if for all $\D'\subseteq \D$ and $\U'\subseteq \U$ with $|\D'| + |\U'| = 2$ the smallest wide subcategory containing $\D'\sqcup \U'$ has rank 2, then the smallest wide subcategory containing $\D\sqcup \U$ has rank 3.
            \item $\Lambda$ has the 2-simple minded compatibility property.
        \end{enumerate}
    \end{cor}
    
We conclude this section with one additional observation which will be useful in the sequel.

\begin{thm}\label{thm:pairwise_wide}
    Let $\Lambda$ be a $\tau$-tilting finite algebra which satisfies the pairwise 2-simple minded completability property, and let $\mathcal{W} \subseteq\mods\Lambda$ be a wide subcategory. Then $\mathcal{W}$ has the pairwise 2-simple minded completability property (when considered as a module category in its own right).
\end{thm}

\begin{proof}
    Let $\D\sqcup \U[1]$ be a pairwise completable semibrick pair for $\mathcal{W}$. Then $\D\sqcup \U[1]$ is also a pairwise completable semibrick pair for $\mods\Lambda$ by Remark~\ref{rem:wide}. By assumption, this means $\D\sqcup \U[1]$ is completable as a semibrick pair for $\mods\Lambda$. Now let $\mathcal{V}$ be the smallest wide subcategory of $\mods\Lambda$ containing $\D\sqcup \U$. By construction, we note that $\mathcal{V}$ is also the smallest wide subcategory of $\mathcal{W}$ which contains $\D \sqcup \U$. Theorems~\ref{thm:mutation} and~\ref{thm:wide} then imply that $\D\sqcup \U[1]$ is completable as a semibrick pair for $\mathcal{W}$.
\end{proof}

\section{2-colored noncrossing arc diagrams}\label{sec:arcs}
    In this section, we introduce 2-colored noncrossing arc diagrams. These are adapted from the arc diagrams of \cite{reading_arcs}, and are designed to simultaneously encode the ascents and descents of a permutation in the Weyl group $A_n$.

\subsection{The weak order on $A_n$}
\label{sec:weak_order background}
In this section, we review the weak order on $A_n$.
Recall that the type-A Weyl group of rank $n$ is isomorphic to the symmetric group on the set $[n+1]:=\{1,2,\ldots, n+1\}$.
For the remainder of the paper, we denote this group by $A_n$.

We write $w\in A_{n}$ in its one-line notation as $w= w_1\ldots w_{n+1}$ where $w_i=w(i)$.
For example, we write $213$ for the permutation where $1\mapsto 2$, $2\mapsto 1$ and $3\mapsto3$.
An \emph{inversion} of $w$ is a pair $(p,q)$ satisfying: $q>p$ and $q$ proceeds $p$ in the word $w_1\ldots w_{n+1}$.
The \emph{inversion set} of $w$, denoted $\inv(w)$, is the set of all such pairs $(p,q)$.

The weak order is a partial order on $A_n$ where $w\le v$ if and only if $\inv(w)\subseteq \inv(v)$.
The Hasse diagram for $A_2$ is shown in Figure~\ref{S_3}.
In particular, $w\covered v$ if and only if $\inv(w) \subset \inv(v)$ and $\inv(v) \setminus \inv(w)$ has precisely one element.
This unique inversion is a so called \emph{descent} for $w$.
Descents (defined below) will play an important role in our proof of Theorem~\ref{Main_C}.
\begin{define}
Let $w$ be a permutation in $A_n$.
\begin{enumerate}
\item A \emph{descent} for a permutation $w$ is a pair of positive integers $(p,q)$ such that $p<q$ and there exists $i \in [n]$ with $q=w_{i}$ and $p=w_{i+1}$.
We write $\des(w)$ for the set of all descents of $w$.
\item An \emph{ascent} for $w$ is a pair $(r,s)$ such that $r<s$ and there exists $j \in [n]$ with $r=w_j$ and $s=w_{j+1}$. 
We write $\asc(w)$ for the set of all ascents of $w$.
\end{enumerate}
\end{define}

\begin{rem}\label{rem:ascDes}
Each descent of $w$ corresponds bijectively to an element $v\covered w$.
Each ascent corresponds bijectively with an element $u\covers w$.
Hence $|\des(w)|+|\asc(w)|=n$ for each $w\in A_n$.
\end{rem}
\begin{rem}\label{uniqueness_des_and_asc}
Assume $A \sqcup D$ is a disjoint union of pairs $A, D \subseteq [n+1]\times [n+1]$.
If $|A|+|D|= n$ then there is at most one permutation $w \in A_n$ such that $\asc(w)= A$ and $\des(w)=D$.
\end{rem}

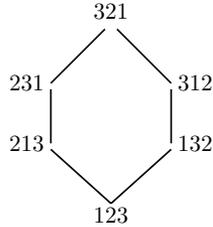
\begin{figure}[h]
  \centering
  \scalebox{.8}{
\begin{tikzpicture}
\draw [black, thick] (0,0) -- (-1,.9);
\draw [black, thick] (0,0) -- (1,.9);
\draw [black, thick] (1,1) -- (1,1.9);
\draw [black, thick] (-1,1) -- (-1,1.9);
\draw [black, thick] (1,2) -- (.1,2.9);
\draw [black, thick] (-1,2) -- (-.1,2.9);
\node at (0,-.2) {$123$};
\node at (1.4,1) {$132$};
\node at (-1.4,1) {$213$};
\node at (1.4,2) {$312$};
\node at (-1.4,2) {$231$};
\node at (0,3.2) {$321$};
\end{tikzpicture}
}
 \caption{The weak order on $A_2$.}
        \label{S_3}
\end{figure}

\subsection{Noncrossing arc diagrams}
\label{sec:arc-diagrams}
In this section, we describe a combinatorial model for the permutations in $A_n$ using green (resp. red) noncrossing arc diagrams.

\begin{rem}
We adapt red and green arc diagrams from the noncrossing arc diagrams first defined in \cite{reading_arcs}.
Our definition for a green noncrossing diagram below coincides with the noncrossing arc diagrams defined in that paper, except that the arcs in our diagram are oriented.
\end{rem}

\begin{define}\label{def:red_arc}
Consider $n+1$ nodes arranged in a vertical column and labeled by the numbers $1, 2, \ldots, n+1$ in increasing order from bottom to top. An \emph{arc} (on $n+1$ nodes) is a directed curve with distinct endpoints in $[n+1]$ which travels monotonically upward or downward and only intersects the $n+1$ labeled nodes at its endpoints.
A \emph{green arc} $\alpha$ travels monotonically \emph{downward} from its top endpoint, $\src(\alpha)$, to its bottom endpoint, $\tar(\alpha)$.
A \emph{red arc} $\beta$ travels monotonically \emph{upward} from its bottom endpoint, $\src(\beta)$, to its top endpoint, $\tar(\beta)$.
For each node between its endpoints, a given arc passes either to the left or to the right.
We consider each arc only up to combinatorial equivalence.
That is, an arc is characterized by its color, its endpoints, and on which side the arc passes each node (either to the left or to the right).
\end{define}

Examples of arcs are shown in Figure \ref{fig: arcs for A4}. These arcs can be considered as green by orienting them downward and can be considered as red by orienting them upward.

\begin{rem}\label{rem:redgreen}
The terms \emph{green} and \emph{red} are chosen to agree with the standard nomenclature for maximal green sequences (see e.g. \cite{keller}). Indeed, in Section \ref{sec:arcs_to_bricks}, we will relate arcs (of arbitrary color) to bricks in the preprojective algebra. Under this correspondence, given a 2-term simple minded collection $\D\sqcup \U[1]$, we wish to visualize the bricks in $\D$ as a set of ``red'' arcs and $\U$ as a set of ``green'' arcs. 
For $S \in \D$, the left mutation $\mu^+_S$ is sometimes called a \emph{green-to-red} or \emph{reddening} mutation. 
Observe that left mutation take $S$ to $S[1] \in \mu^+_S(\D\sqcup \U[1])$, and thus the arc corresponding to $S$ should change from ``green'' to ``red''.

The upshot of this is that ``green'' and ``red'' can be considered as abstract properties of arcs rather than as colors. Thus to make this document more accessible, we will typically draw green arcs as solid blue and red arcs as dashed orange. Arrows on these arcs (see e.g. Figure \ref{fig:2-coloredDiagram}) indicate whether each arc travels monotonically upward or monotonically downward. When the color of an arc is not relevant, we generally draw it in black and with no arrows. (see e.g. Figure \ref{fig: left and right position}). 
\end{rem}

\begin{define}\label{nc_diagram}
A \emph{green (resp. red) noncrossing arc diagram} (on $n+1$ nodes) is a (possibly empty) set of green (resp. red) arcs (on $n+1$ nodes) which can be drawn so that each pair of arcs satisfy the following compatibility conditions:
\begin{enumerate}[label={(C\arabic*)}]
\item\label{C1} $\alpha$ and $\beta$ do not share a bottom endpoint or a top endpoint;
\item\label{C2} $\alpha$ and $\beta$ do not cross in their interiors.
\end{enumerate}
\end{define}
We note that two arcs may share an endpoint, but it must be the top endpoint of one and the bottom endpoint of the other.
For example, the four diagrams shown in Figure \ref{fig: left and right position} can be considered as green noncrossing arc diagrams by orienting the arcs downward.

We now describe a map $\delta$ from the set of permutations in $A_n$ to the set of green noncrossing arc diagrams on $n+1$ nodes.
Given $w=w_1\ldots w_{n+1}$, we plot the point $(i,w_i)$ in $\mathbb{R}^2$.
We connect $(i, w_i)$ to $(i+1, w_{i+1})$ with a straight line segment whenever $w_i>w_{i+1}$.
(That is, whenever the pair $w_i$ and $w_{i+1}$ are a descent.)
Finally, we move all of the points into a vertical line, bending the straight line segments so that they become the arcs in our diagram. 
See Figure~\ref{fig:delta}, or see \cite[Figure 4]{reading_arcs} for a larger example.

\begin{center}
\begin{figure}
\begin{tikzpicture}

\draw[thick,color=orange](-1,2)--(-0.5,0);
\draw[thick,color=orange](0,1)--(0.5,-2);

\node[draw,fill=white] at (-1,2) {5};
\node[draw,fill=white] at (-0.5,0) {3};
\node[draw,fill=white] at (0,1) {4};
\node[draw,fill=white] at (0.5,-2) {1};
\node[draw,fill=white] at (1,-1) {2};

\draw[thick,->] (2,0) -- (4,0) node[midway,above] {\text{squish}};

\begin{scope}[shift = {(6,0)}]


\begin{scope}[decoration={markings,mark = at position 0.15 with {\arrow{<}},mark=at position 0.5 with {\arrow{<}},mark=at position 0.85 with {\arrow{<}}}]
\draw[thick,orange,smooth,postaction=decorate](0,-2) to [out = 150, in = -150] (0,-0.5) to [out = 30, in = -30] (0,1);
\end{scope}

\begin{scope}[decoration={markings,mark=at position 0.25 with {\arrow{<}},mark=at position 0.75 with {\arrow{<}}}]
\draw[thick,orange,smooth,postaction=decorate](0,0) to [out = 150, in = -150] (0,2);
\end{scope}

\begin{scope}[decoration={markings,mark=at position 0.5 with {\arrow{>}}}]
\end{scope}

\draw[fill] (0,1) circle [radius=.07];
\draw[fill] (0,-2) circle [radius=.07];
\draw[fill] (0,0) circle [radius=.07];
\draw[fill] (0,2) circle [radius=.07];
\draw[fill] (0,-1) circle [radius=.07];

\end{scope}
\end{tikzpicture}
\caption{An example of the bijection $\delta$.}\label{fig:delta}
\end{figure}
\end{center}

\begin{thm}\cite[Theorem~3.1]{reading_arcs}
The map $\delta$ from the set of permutations in $A_n$ to the set of noncrossing green arc diagrams on $n+1$ nodes is a bijection.
\end{thm}

By switching each instance of the word ``descent'' with ``ascent'' in the paragraph above, and connecting the points $(i,w_i)$ to $(i+1,w_{i+1})$ whenever $w_i<w_{i+1}$, we immediately obtain the following corollary.

\begin{cor}
There is a bijection which we denote $\bar{\delta}$ from the set of permutations in $A_n$ to the set of noncrossing red arc diagrams on $n+1$ nodes which sends the ascents of a permutation $w$ to a set of compatible red arcs.
\end{cor}

\subsection{2-colored noncrossing arc diagrams}
We now wish to extend our combinatorial model to simultaneously encode both the descents and ascents of a permutation. To do so, we introduce \emph{2-colored noncrossing arc diagrams}, which consist of a green noncrossing arc diagram and a red noncrossing arc diagram satisfying some compatibility condition. In order to formulate this condition, we first need the following definitions.

\begin{define}\label{def:support and sub}
Let $\alpha$  and $\beta$ be arcs on $n+1$ nodes.
\begin{enumerate}
    \item The \emph{support} of $\alpha$, written $\supp(\alpha)$, is the set the set of nodes between (and including) the endpoints of $\alpha$.
We write $\suppo(\alpha)$ for the set of nodes strictly between (i.e. not including) $\alpha$'s endpoints.
    \item We say that $\alpha$ has full support if its bottom endpoint is 1 and its top endpoint is $n+1$.
    \item We say that \emph{the supports of $\alpha$ and $\beta$ overlap} provided that $|\supp(\alpha)\cap \supp(\beta)|>1$.
\end{enumerate}
\end{define}

Examples of arcs with full support are shown in Figure \ref{fig: arcs for A4}.

\begin{figure}[h]
\centering
\scalebox{.8}{
\begin{tikzpicture}
\draw [black] (0,-3) to (0,3);
\draw [black] (-4,-3) to (-4,3);
\draw [black] (4,-3) to (4,3);
\draw[black] (-7,0) to (7,0);


\node at (-5.5,1.7) {\scalebox{1.45}{\includegraphics{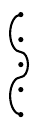}}};

\node at (-2,1.7) {\scalebox{1.45}{\includegraphics{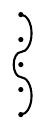}}};

\node at (2,1.7) {\scalebox{1.45}{\includegraphics{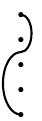}}};

\node at (5.5,1.7) {\scalebox{1.45}{\includegraphics{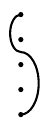}}};


\node at (-5.5,-1.5) {\scalebox{1.45}{\includegraphics{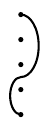}}};

\node at (-2,-1.5) {\scalebox{1.45}{\includegraphics{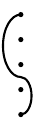}}};

\node at (1.68,-1.5) {\scalebox{1.45}{\includegraphics{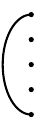}}};

\node at (5.8,-1.5) {\scalebox{1.45}{\includegraphics{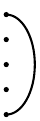}}};
\end{tikzpicture}}
\caption{The arcs with full support on $5$ nodes.}
\label{fig: arcs for A4}
\end{figure}

\begin{define}\label{def:leftof}
Let $\alpha$ and $\beta$ be two arcs which do not cross and do not have the same support, and suppose $|\supp(\alpha)\cap \supp(\beta)| > 1$.
We say that $\beta$ is \emph{left} of $\alpha$ provided that 
\begin{enumerate}
\item If $i\in \src(\beta)\cup \tar(\beta)$ and $i\in \suppo(\alpha)$, then $i$ is on the left side of
$\alpha$.
\item If $i\in \src(\alpha)\cup \tar(\alpha)$ and $i\in \suppo(\beta)$ then $i$ is on the right side of $\beta$.
\end{enumerate}
Whenever $\beta$ is left of $\alpha$, we may equivalently say that $\alpha$ is \emph{right} of $\beta$.
\end{define}

\begin{rem}\label{rmk:left_for_proof}
Note that if $\beta$ is left of $\alpha$, the fact that $\alpha$ and $\beta$ do not cross implies that 
if $i\in \suppo(\alpha)\cap\suppo(\beta)$ and $i$ is left of $\beta$, then $i$ is also left of $\alpha$.
Similarly if $i\in \suppo(\alpha)\cap\suppo(\beta)$ and $i$ is right of $\alpha$, then $i$ is also right of $\beta$.
\end{rem}

Examples of arcs which are left of one another are shown in Figure \ref{fig: left and right position}. We are now ready to construct our combinatorial model.

 \begin{figure}[h]
 \centering
\scalebox{.65}{
\begin{tikzpicture}

\draw[fill] (-5,1) circle [radius=.07];
\draw[fill] (-5,-2) circle [radius=.07];
\draw[fill] (-5,0) circle [radius=.07];
\draw[fill] (-5,2) circle [radius=.07];
\draw[fill] (-5,-1) circle [radius=.07];
\draw [black, thick] (-5,-2) to [out= 150, in =-150 ] (-5,2);
\draw [black, ultra thick, dashed] (-5,-1) to [out= 30, in =-30 ] (-5,1);

\draw[fill] (-2,1) circle [radius=.07];
\draw[fill] (-2,-2) circle [radius=.07];
\draw[fill] (-2,0) circle [radius=.07];
\draw[fill] (-2,2) circle [radius=.07];
\draw[fill] (-2,-1) circle [radius=.07];
\draw [black, thick] (-2,-2) to [out= 150, in =-150 ] (-2,2);
\draw [black, ultra thick, dashed] (-2,-1) to [out= 150, in =-150 ] (-2,1);

\draw[fill] (0,1) circle [radius=.07];
\draw[fill] (0,-2) circle [radius=.07];
\draw[fill] (0,0) circle [radius=.07];
\draw[fill] (0,2) circle [radius=.07];
\draw[fill] (0,-1) circle [radius=.07];

\draw [black, thick] (0,-2) to [out= 150, in =210 ] (0,-.5);
\draw [black, thick] (0,-.5) to [out = 30, in = -30] (0,.5);
\draw [black, thick] (0,.5) to [out= 150, in =-150 ] (0,2);
\draw [black, ultra thick, dashed] (0,-1) to [out= 30, in =-30 ] (0,1);

\draw[fill] (2,1) circle [radius=.07];
\draw[fill] (2,-2) circle [radius=.07];
\draw[fill] (2,0) circle [radius=.07];
\draw[fill] (2,2) circle [radius=.07];
\draw[fill] (2,-1) circle [radius=.07];

\draw [black, thick] (2,-1) to [out= 150, in =210 ] (2,.5);
\draw [black, thick] (2, .5) to [out= 30, in =-30 ] (2,2);
\draw [black, ultra thick, dashed] (2,-2) to [out= 30, in =-30 ] (2,0);

\draw[fill] (4.5,1) circle [radius=.07];
\draw[fill] (4.5,-2) circle [radius=.07];
\draw[fill] (4.5,0) circle [radius=.07];
\draw[fill] (4.5,2) circle [radius=.07];
\draw[fill] (4.5,-1) circle [radius=.07];

\draw [black, thick] (4.5,-2) to [out= 150, in =-150 ] (4.5,0);
\draw [black, ultra thick, dashed] (4.5,-1) to [out= 30, in =-30 ] (4.5,2);

\end{tikzpicture}
}
\caption{In each diagram, the solid arc is left of the dashed arc. }
\label{fig: left and right position}
\end{figure}
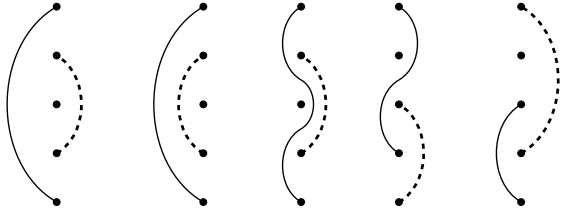

\begin{define}\label{def:2-colored}
    Let $\mathcal{G}$ be a green noncrossing arc diagram on $n+1$ nodes and let $\mathcal{R}$ be a red noncrossing arc diagram on $n+1$ nodes.
    \begin{enumerate}
        \item We say that $(\mathcal{G},\mathcal{R})$ is a \emph{2-colored noncrossing arc diagram} if for all $\alpha \in \mathcal{G}$ and $\beta \in \mathcal{R}$, we have that $\supp(\alpha) \neq \supp(\beta)$ and we can draw $\alpha$ and $\beta$ together so that:
       \begin{enumerate}[label={(TC\arabic*)}]
        \item\label{TC1} $\alpha$ and $\beta$ do not cross in their interiors.
        \item\label{TC2} $\src(\alpha) \neq \src(\beta)$ and $\tar(\alpha) \neq \tar(\beta)$.
        
        \item\label{TC3} If $\tar(\alpha) = \src(\beta)$, then $\alpha$ is left of $\beta$.
        \item\label{TC4} If $\tar(\beta) = \src(\alpha)$, then $\beta$ is left of $\alpha$.
    \end{enumerate}
        \item If $(\mathcal{G},\mathcal{R})$ is a 2-colored noncrossing arc diagram on $n+1$ nodes and there exists a permutation $w \in A_n$ so that $\mathcal{G} \subseteq \delta(w)$ and $\mathcal{R} \subseteq \overline{\delta}(w)$, we say that $(\mathcal{G},\mathcal{R})$ is \emph{completable}. When these containments are both equalities, we say that $(\mathcal{G},\mathcal{R})$ is \emph{complete}.
    \end{enumerate}
\end{define}

\begin{rem}
    It does not follow \emph{a priori} that every completable 2-colored noncrossing arc diagram is a subset of a complete 2-colored noncrossing arc diagram. Indeed, for $w \in A_n$, we know that $\delta(w)$ (resp. $\overline{\delta}(w)$) is a green (resp. red) noncrossing arc diagram, but not that the pair $(\delta(w),\overline{\delta}(w))$ is a 2-colored noncrossing arc diagram. (The axioms~\ref{TC1}-\ref{TC4} do not follow immediately.) We will show that this is in fact the case in Corollary~\ref{cor:miz_sec}.
\end{rem}

\begin{ex}\label{ex:arcDiagramNotComplete}
 In Figure~\ref{fig:2-coloredDiagram}, from left to right we draw a 2-colored noncrossing arc diagram that is not completable, and two complete 2-colored noncrossing arc diagrams.
To see that the leftmost diagram is not completable, we note that the only possible permutations 
$w \in A_n$ for which $\delta(w)$ and $\overline{\delta}(w)$ could contain the arcs in this diagram are 31452 and 14523. (These are the only permutations which have both (1,4) and (4,5) as ascents and (2,5) as a descent.) 
However, the red arc between the nodes 1 and 4 in $\overline{\delta}(31452)$ passes to the right of the node 3. Likewise, the green arc from 5 to 2 in $\delta(14523)$ passes to the left of 3.
As neither such arc appears in the left diagram, we conclude that it is not completable.
\end{ex}

 \begin{figure}[h]
 \centering
\begin{tikzpicture}

\begin{scope}[decoration={markings,mark=at position 0.25 with {\arrow{<}},mark=at position 0.75 with {\arrow{<}}}]
\draw[very thick,blue,smooth,postaction=decorate](0,-1) to [out = 30, in = -30] (0,2);
\end{scope}

\begin{scope}[decoration={markings,mark=at position 0.25 with {\arrow{>}},mark=at position 0.75 with {\arrow{>}}}]
\draw[very thick,orange,dashed,smooth,postaction=decorate](0,-2) to [out = 150, in = -150] (0,1);
\end{scope}

\begin{scope}[decoration={markings,mark=at position 0.5 with {\arrow{>}}}]
\draw[very thick,orange,dashed,smooth,postaction=decorate] (0,1) to (0,2);
\end{scope}

\draw[fill] (0,1) circle [radius=.07];
\draw[fill] (0,-2) circle [radius=.07];
\draw[fill] (0,0) circle [radius=.07];
\draw[fill] (0,2) circle [radius=.07];
\draw[fill] (0,-1) circle [radius=.07];


\begin{scope}[decoration={markings,mark=at position 0.25 with {\arrow{>}},mark=at position 0.75 with {\arrow{>}}}]
\draw[very thick,blue,smooth,postaction=decorate](4,2) to [out = -30, in = 90] (5,0.5) to [out=-90, in = 30] (4,-1);
\draw[very thick,blue,smooth,postaction=decorate](4,0) to [out=-150, in = 90] (3.25,-1) to [out = -90, in = 150] (4, -2);
\end{scope}

\begin{scope}[decoration={markings,mark=at position 0.5 with {\arrow{>}}}]
\draw[very thick,orange,dashed,smooth,postaction=decorate] (4,1) to (4,2);
\end{scope}

\begin{scope}[decoration={markings,mark = at position 0.15 with {\arrow{>}},mark=at position 0.5 with {\arrow{>}},mark=at position 0.85 with {\arrow{>}}}]
\draw[very thick,orange,dashed,smooth,postaction=decorate](4,-2) to [out = 120, in = -150] (4,-0.5) to [out = 30, in = -30] (4,1);
\end{scope}

\draw[fill] (4,1) circle [radius=.07];
\draw[fill] (4,-2) circle [radius=.07];
\draw[fill] (4,0) circle [radius=.07];
\draw[fill] (4,2) circle [radius=.07];
\draw[fill] (4,-1) circle [radius=.07];


\begin{scope}[decoration={markings,mark=at position 0.25 with {\arrow{>}},mark=at position 0.75 with {\arrow{>}}}]
\draw[very thick,orange,dashed,smooth,postaction=decorate](8,-2) to [out = 150, in = -90] (7,-0.5) to [out=90, in = -150] (8,1);
\end{scope}

\begin{scope}[decoration={markings,mark=at position 0.5 with {\arrow{>}}}]
\draw[very thick,orange,dashed,smooth,postaction=decorate] (8,1) to (8,2);
\draw[very thick,orange,dashed,smooth,postaction=decorate] (8,-1) to (8,0);
\end{scope}

\begin{scope}[decoration={markings,mark = at position 0.15 with {\arrow{<}},mark=at position 0.5 with {\arrow{<}},mark=at position 0.85 with {\arrow{<}}}]
\draw[very thick,blue,smooth,postaction=decorate](8,-1) to [out = 150, in = -150] (8,0.5) to [out = 30, in = -30] (8,2);
\end{scope}

\draw[fill] (8,1) circle [radius=.07];
\draw[fill] (8,-2) circle [radius=.07];
\draw[fill] (8,0) circle [radius=.07];
\draw[fill] (8,2) circle [radius=.07];
\draw[fill] (8,-1) circle [radius=.07];

\end{tikzpicture}
\caption{(left) A 2-colored noncrossing arc diagram which is not completable. (middle) The complete 2-colored noncrossing arc diagram corresponding to the permutation 31452. (right) The complete 2-colored noncrossing arc diagram corresponding to the permutation 14523.}
\label{fig:2-coloredDiagram}
\end{figure} 

\begin{rem}\label{rem:whyDef}
    The definition of a 2-colored noncrossing arc diagram is motivated by representation theory. Indeed, we will show in Theorem \ref{Main_B} that 2-colored noncrossing arc diagrams correspond to semibrick pairs over preprojective algebras of type $A$ (Definition~\ref{def:preproj}) and that a semibrick pair is completable (resp. is a 2-term simple minded collection) if and only if the corresponding 2-colored noncrossing arc diagram is completable (resp. is complete). In particular, this will imply that for any $w \in A_n$, the pair $(\delta(w),\overline{\delta}(w))$ is a (complete) 2-colored noncrossing arc diagram.
    Below we sketch a combinatorial proof of this fact that is similar to the description of $\delta$ and $\overline{\delta}$.
    
    Given $w\in A_n$, we graph the points $(i, w_i)$.
    Then we connect consecutive points $(i, w_i)$ and $(i+1, w_{i+1})$ with a straight line segment that is oriented from left to right.
    Finally, we ``squish'' the entire graph into a vertical column.
    Now it is clear, regardless of color, saying that $\tar(\alpha) = \src(\beta)$ simply means that $\alpha$ precedes $\beta$ and is adjacent to it, hence $\alpha$ must pass to the left of $\beta$.
    
\end{rem}

\begin{rem}
    In the recent paper \cite{Miz2020}, Mizuno independently introduces so-called ``double arc diagrams''. We will show in Section \ref{sec:ThmD} that these ``double arc diagrams'' are precisely the complete 2-colored noncrossing arc diagrams. See Remark \ref{rem:miz_sec} for further discussion.
\end{rem}

We now show that complete 2-colored noncrossing arc diagrams are characterized precisely as those containing the maximum possible number of arcs.

\begin{thm}\label{thm:fullSizeComplete}
    Let $(\mathcal{G},\mathcal{R})$ be a 2-colored noncrossing arc diagram on $n+1$ nodes. Then $(\mathcal{G},\mathcal{R})$ is complete if and only if $|\mathcal{G}| + |\mathcal{R}| = n.$
\end{thm}

Before we jump into the proof of Theorem~\ref{thm:fullSizeComplete}, we will need the following technical lemma.

\begin{lem}\label{lem:trans_of_left}
Suppose that $\alpha_1, \alpha_2, \ldots, \alpha_k$ belong to a 2-colored arc diagram, and satisfy $\alpha_i$ is left of $\alpha_{i+1}$ for each $i=1,2,\ldots k-1$, where $k\ge 3$.
If $|\supp(\alpha_1)\cap \supp(\alpha_k)|>1$ then $\alpha_1$ is left of~$\alpha_k$.
\end{lem}
\begin{proof}
We proceed by induction.
First let $k=3$.
Note that since $\alpha_1$ and $\alpha_3$ do not cross (as they form a 2-colored arc diagram), it is enough to show either the first or second condition in Definition~\ref{def:leftof}.
Set $M$ equal to the smallest top endpoint among $\alpha_1$, $\alpha_2$, and $\alpha_3$.
Start by assuming that $M$ is the top endpoint of $\alpha_1$.
Our assumption that $|\supp(\alpha_1)\cap \supp(\alpha_i)|>1$ implies that $M\in \supp(\alpha_i)$, for $i=2,3$.
Observe that $M$ is not also the top endpoint of $\alpha_3$ (otherwise the top endpoint of $\alpha_2$ must be between $\alpha_1$ and $\alpha_3$, and hence be smaller than $M$).
Thus $M\in \suppo(\alpha_3)$.
If $M$ is also in $\suppo(\alpha_2)$, then since $\alpha_1$ is left of $\alpha_2$ it follows that $M$ is left of $\alpha_2$.
Remark~\ref{rmk:left_for_proof} says that $M$ is also left of $\alpha_3$, and we are done.
If $M$ is also a top endpoint of $\alpha_2$, then $\alpha_2$ is left of $\alpha_3$, it follows that $M$ is left of $\alpha_3$, and we are done.
A symmetric argument (using ``right'' instead of ``left'') also proves the case when $M$ is the top endpoint of $\alpha_3$ (but a top endpoint of $\alpha_1$).

If $M$ is a top endpoint of $\alpha_2$, but not of either $\alpha_1$ or $\alpha_3$, then it follows that $M$ is right of $\alpha_2$ and left of $\alpha_3$.
Since $\alpha_1$ and $\alpha_3$ do not cross, we must have that $\alpha_1$ is left of $\alpha_3$.

Now suppose that $k>3$.
Let $j$ be the largest index such that $2\le j<k$ and $|\supp(\alpha_1)\cap \supp(\alpha_j)|>1$.
By induction $\alpha_1$ is left of $\alpha_j$.
If we can show that $|\supp(\alpha_j)\cap \supp(\alpha_k)|>1$, then we also get $\alpha_j$ is left of $\alpha_k$ by induction, and we can complete the proof using the same argument as in the base case.
Assume that $|\supp(\alpha_j)\cap \supp(\alpha_k)|\le 1$.
Write $[m_j, M_j]$ for $\supp(\alpha_j)\cap \supp(\alpha_1)$, and $[m_k, M_k]$ for $\supp(\alpha_k)\cap \supp(\alpha_1)$.
Either $M_j<m_k$ or $M_k<m_j$.
Either $M_j<m_k$ or $M_k<m_j$.
We write the proof for the case where $M_j < m_k$ (the other case is the same).
Since $\alpha_j$ is the last arc before $\alpha_k$ whose support overlaps $\alpha_1$, it must be the case that $m_j$ is the bottom endpoint of $\alpha_1$, and the top endpoint for the next arc, $\alpha_{j+1}$ is less than or equal to $m_j$.
In fact, the top endpoint of each arc after $\alpha_j$ is less than or equal to $m_j$ (and at most two arcs are incident to $m_j$).
But in that case, $\supp(\alpha_{k-1})\cap \supp(\alpha_k)=\emptyset$, which contradicts our assumption that $\alpha_{k-1}$ is left of $\alpha_k$.
By this contradiction we conclude that $\supp(\alpha_j)\cap \supp(\alpha_k)\neq \emptyset$, thus completing the proof.

\end{proof}

\begin{proof}[Proof of Theorem~\ref{thm:fullSizeComplete}]
The ``only if'' part follows immediately from Remark \ref{rem:ascDes}, so we suppose that $|\mathcal{G}| + |\mathcal{R}| = n$. Consider the set of arcs in $\mathcal{G}\cup\mathcal{R}$ as a directed graph on $[n+1]$. Note that by~\ref{TC1}, each arc in $\mathcal{G}\cup\mathcal{R}$ is either red or green; i.e., it is contained in only one of $\mathcal{G}$ and $\mathcal{R}$.
In each connected component of this graph, order adjacent arcs so that $\alpha\prec \beta$ provided that $\tar(\alpha)=\src(\beta)$.

We claim that each connected component is a tree, so that the transitive closure of the relation described above is a chain.
By way of contradiction assume there is a connected component that contains a directed cycle, and write this cycle as $v_1\alpha_1v_2 \alpha_2 \ldots \alpha_k v_{k+1}=v_1$.
(An undirected cycle would violate either~\ref{C1} or~\ref{TC2}.)
Note this cycle is actually equal to the connected component, because each vertex has degree at most 2.
(If there were a vertex with degree 3 or more, then two red (or two green) arcs would violate~\ref{C1}).
Also, the cycle must contain both red and green arcs because arcs of a single color travel monotonically downward or monotonically upward.

We may choose $v_1$ so that it is the largest vertex in the cycle, and thus $\alpha_1$ is a green arc.
Let $i$ be the smallest positive integer such that $\alpha_i$ is a red arc.
Then \ref{TC3} implies that $\alpha_{i}$ passes to the right of $\alpha_{i-1}$.
Also, $\alpha_i$ does not cross $\alpha_1, \alpha_2, \ldots, \alpha_{i-1}$ in their interiors by~\ref{C1} and~\ref{TC1}.
In particular, if $\supp(\alpha_{i-2})\cap \supp(\alpha_i)\neq \emptyset$ then $\tar(\alpha_{i-2})=\src(\alpha_{i-1})$ is in $\suppo(\alpha_i)$.
Since $\alpha_i$ is right of $\alpha_{i-1}$, then we must have $\tar(\alpha_{i-2})=\src(\alpha_{i-1})$ is on the left side of $\alpha_i$.
Thus, $\alpha_i$ is also right of $\alpha_{i-1}$.
Continuing in this way, we see that $\alpha_i$ is right of each of the green arcs $\alpha_1, \alpha_2, \ldots, \alpha_{i-1}$.

By Lemma~\ref{lem:trans_of_left}, we see that each time we transition from a sequence of green arcs to red arcs, or red arcs to green arcs, the next arc is \emph{right} of the previous ones.

Since $v_1$ was chosen to be largest, the last arc, $\alpha_k$ must be a red arc.
Since $\tar(\alpha_k) = \src(\alpha_1)$, their supports must overlap.
The previous paragraph implies that $\alpha_k$ is right of $\alpha_1$.
However,~\ref{TC4} implies that $\alpha_k$ is also \emph{left} of $\alpha_1$, and this is a contradiction.
By this contradiction, we conclude that each connected component is a tree.

Next we claim that the directed graph consisting of all of the arcs in $\mathcal{G}\cup\mathcal{R}$ is connected.
Assume there are $l\ge1$ connected components, and write $k_i$ for the number vertices in the $i$-th connected component.
Then $n+1= k_1+k_2+\cdots + k_l$.
Each tree with $k_i$ vertices has $k_i-1$ edges/arcs.
Therefore:
\begin{eqnarray*}
(k_1-1)+(k_2-1)+\cdots + (k_l-1)
&=&(k_1+k_2+\cdots +k_l) - l\\
&=&(n+1) -l.
\end{eqnarray*}
Since we know there are $n$ arcs, we must have $l=1$.
This proves our second claim.

Finally, we observe that the ordering $\alpha\prec\beta$ induces a total order on $[n+1]$ where $\src(\alpha)\prec\tar(\alpha)=\src(\beta)\prec\tar(\beta)$.
Let $w$ be the resulting permutation, so that for each $i$, $w_i$ is the $i$-th smallest element of $[n+1]$ under $\prec$.
A pair $(i,j)$ is an descent of $w$ if and only if there is green arc $\alpha \in \mathcal{G}$ with $i= \src(\alpha)$ and $j=\tar(\alpha)$.
A pair $(i,j)$ is a ascent if and only if there is a red arc $\beta \in \mathcal{R}$ with $i=\tar(\beta)$ and $j=\src(\beta)$.
So we have constructed the desired permutation, and we conclude that $(\mathcal{G},\mathcal{R})$ is complete.
\end{proof}

\begin{rem}
    It does not follow immediately from Theorem \ref{thm:fullSizeComplete} that if $w \in A_n$ is a permutation, then $(\delta(w),\overline{\delta}(w))$ is a 2-colored noncrossing arc diagram. We shall see, however, that this is indeed the case as a consequence of Theorem~\ref{Main_B}.
\end{rem}

For the remainder of this section, we turn our attention to 2-colored noncrossing diagrams with fewer than $n$ arcs. More specifically, we wish to show that if an arc diagram is not completable, we can construct (potentially several) permutations in $A_n$ from the data in the diagram. To make this precise, we need the following definition.

\begin{define}\label{def:supportEquiv}
    Let $(\mathcal{G},\mathcal{R})$ and $(\mathcal{G}',\mathcal{R}')$ be 2-colored noncrossing arc diagrams on $n+1$ nodes. We say that $(\mathcal{G},\mathcal{R})$ and $(\mathcal{G}',\mathcal{R}')$ are \emph{support equivalent} if there exists bijections $\psi_g: \mathcal{G} \rightarrow \mathcal{G}'$ and $\psi_r:\mathcal{R}\rightarrow \mathcal{R}'$ such that for all $\alpha \in \mathcal{G}$ and $\beta \in \mathcal{R}$, we have $\supp(\alpha) = \supp(\psi_g(\alpha))$ and $\supp(\beta) = \supp(\psi_r(\beta)$.
\end{define}

The left and middle diagrams in Figure~\ref{fig:supportEquiv} are an example of 2-colored noncrossing arc diagrams which are support equivalent. We recall from Example~\ref{ex:arcDiagramNotComplete} that the left diagram is not completable. The middle diagram, however, is completable. Indeed, the permutation 14523 corresponds to the 2-colored noncrossing arc diagram on the right of Figure~\ref{fig:supportEquiv}. This diagram contains all of the arcs from the middle diagram plus an additional orange arc from 2 to 3.

 \begin{figure}
 \centering
\begin{tikzpicture}

\begin{scope}[decoration={markings,mark=at position 0.25 with {\arrow{<}},mark=at position 0.75 with {\arrow{<}}}]
\draw[very thick,blue,smooth,postaction=decorate](0,-1) to [out = 30, in = -30] (0,2);
\end{scope}

\begin{scope}[decoration={markings,mark=at position 0.25 with {\arrow{>}},mark=at position 0.75 with {\arrow{>}}}]
\draw[very thick,orange,dashed,smooth,postaction=decorate](0,-2) to [out = 150, in = -150] (0,1);
\end{scope}

\begin{scope}[decoration={markings,mark=at position 0.5 with {\arrow{>}}}]
\draw[very thick,orange,dashed,smooth,postaction=decorate] (0,1) to (0,2);
\end{scope}

\draw[fill] (0,1) circle [radius=.07];
\draw[fill] (0,-2) circle [radius=.07];
\draw[fill] (0,0) circle [radius=.07];
\draw[fill] (0,2) circle [radius=.07];
\draw[fill] (0,-1) circle [radius=.07];


\begin{scope}[decoration={markings,mark=at position 0.25 with {\arrow{>}},mark=at position 0.75 with {\arrow{>}}}]
\draw[very thick,orange,dashed,smooth,postaction=decorate](8,-2) to [out = 150, in = -90] (7,-0.5) to [out=90, in = -150] (8,1);
\end{scope}

\begin{scope}[decoration={markings,mark=at position 0.5 with {\arrow{>}}}]
\draw[very thick,orange,dashed,smooth,postaction=decorate] (8,1) to (8,2);
\draw[very thick,orange,dashed,smooth,postaction=decorate] (8,-1) to (8,0);
\end{scope}

\begin{scope}[decoration={markings,mark = at position 0.15 with {\arrow{<}},mark=at position 0.5 with {\arrow{<}},mark=at position 0.85 with {\arrow{<}}}]
\draw[very thick,blue,smooth,postaction=decorate](8,-1) to [out = 150, in = -150] (8,0.5) to [out = 30, in = -30] (8,2);
\end{scope}

\draw[fill] (8,1) circle [radius=.07];
\draw[fill] (8,-2) circle [radius=.07];
\draw[fill] (8,0) circle [radius=.07];
\draw[fill] (8,2) circle [radius=.07];
\draw[fill] (8,-1) circle [radius=.07];


\begin{scope}[decoration={markings,mark=at position 0.25 with {\arrow{>}},mark=at position 0.75 with {\arrow{>}}}]
\draw[very thick,orange,dashed,smooth,postaction=decorate](4,-2) to [out = 150, in = -90] (3,-0.5) to [out=90, in = -150] (4,1);
\end{scope}

\begin{scope}[decoration={markings,mark=at position 0.5 with {\arrow{>}}}]
\draw[very thick,orange,dashed,smooth,postaction=decorate] (4,1) to (4,2);
\end{scope}

\begin{scope}[decoration={markings,mark = at position 0.15 with {\arrow{<}},mark=at position 0.5 with {\arrow{<}},mark=at position 0.85 with {\arrow{<}}}]
\draw[very thick,blue,smooth,postaction=decorate](4,-1) to [out = 150, in = -150] (4,0.5) to [out = 30, in = -30] (4,2);
\end{scope}

\draw[fill] (4,1) circle [radius=.07];
\draw[fill] (4,-2) circle [radius=.07];
\draw[fill] (4,0) circle [radius=.07];
\draw[fill] (4,2) circle [radius=.07];
\draw[fill] (4,-1) circle [radius=.07];

\end{tikzpicture}
\caption{The left and middle diagrams are support equivalent. The middle diagram is completable, while the left diagram is not. The diagram on the right corresponds to the permutation $14523$.}\label{fig:supportEquiv}
\end{figure}
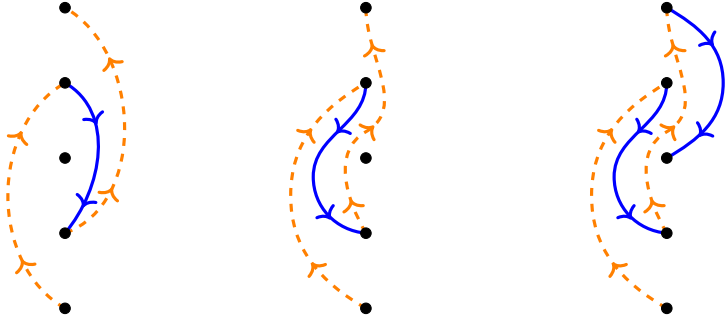

The following result shows that the previous example illustrated in Figure~\ref{fig:supportEquiv} is actually a general phenomenon.

\begin{thm}\label{thm:supportEquiv}
    Let $(\mathcal{G},\mathcal{R})$ be a 2-colored noncrossing arc diagram (on $n~+~1$ nodes). Then there exists a completable 2-colored noncrossing arc diagram $(\mathcal{G}',\mathcal{R}')$ which is support equivalent to~$(\mathcal{G},\mathcal{R})$.
\end{thm}

\begin{proof}
    Consider the set of arcs in $\mathcal{G}\cup\mathcal{R}$ as a directed graph on $[n+1]$. Now, as a consequence of \ref{C1} and \ref{TC2}, each node in this graph is the source of at most one arc and the target of at most one arc. Based on this observation, for $j \in [n+1]$, we denote by $\alpha_j$ the unique arc in $\mathcal{G}\cup\mathcal{R}$ with source $j$, if it exists. As in the proof of Theorem~\ref{thm:fullSizeComplete}, this gives a partial order on $[n+1]$ by taking the transitive closure of the relation $j \prec \tar(\alpha_j)$ for each arc $\alpha_j$. Choose some linear extension of this partial order, and let $w = w_1\cdots w_{n+1} \in A_n$ be the corresponding permutation.

     Now let $j \in [n+1]$. If $w_j > w_{j+1}$ (in the usual order on $\mathbb{N})$, we construct a green arc $\alpha'_j$ as follows:
     \begin{enumerate}
         \item The source of $\alpha'_j$ is $w_j$.
         \item The target of $\alpha'_j$ is $w_{j+1}$.
         \item For $w_k \in (w_{j+1},w_j)$, the node $w_k$ is right of $\alpha'_j$ if and only if $j<k$.
     \end{enumerate}
     If $w_j < w_{j+1}$ (in the usual order on $\mathbb{N}$), we construct a red arc $\beta'_j$ analogously.
     
     Now let $\mathcal{G}'' = \{\alpha'_j:w_j > w_{j+1}\}$ and $\mathcal{R}'' = \{\beta'_j: w_j < w_{j+1}\}$. We claim that $(\mathcal{G}'',\mathcal{R}'')$ is a complete 2-colored noncrossing arc diagram.
     
     First let $\alpha'_j,\alpha'_k \in \mathcal{G}''$ with $j < k$. We see that $\alpha'_j$ and $\alpha'_k$ satisfy \ref{C1} because $\alpha'_j$ is the only arc in $\mathcal{G}''$ with source $w_j$ and the only arc in $\mathcal{G}''$ with target $w_{j+1}$. To see that $\alpha'_j$ and $\alpha'_k$ satisfy \ref{C2}, let $w_i \in \supp(\alpha'_j)\cap \supp(\alpha'_k)$. If $w_i$ is right of $\alpha'_k$, then by construction, we have $k < i$. This means $j < i$, and so $w_i$ is right of $\alpha'_j$ as well. Analogously, if $w_i$ is left of $\alpha'_j$, then $w_i$ is left of $\alpha'_k$ as well. We conclude that we can draw $\alpha'_j$ and $\alpha'_k$ so that $\alpha'_j$ is always to the left of $\alpha'_k$ when their supports overlap. In particular, \ref{C2} is satisfied. This shows that $\mathcal{G}''$ is a (green) noncrossing arc diagram. The proof that $\mathcal{R}''$ is a (red) noncrossing arc diagram is analogous.
     
     Now let $\alpha'_j \in \mathcal{G}''$ and $\beta'_k \in \mathcal{R}''$. It is clear that  $\supp(\alpha'_j) \neq \supp(\beta'_k)$ (otherwise we would have both $j+1 = k$ and $k + 1 = j$). Moreover, we see that $\alpha'_j$ and $\beta'_k$ satisfy \ref{TC1} and \ref{TC2} by arguments analogous to those in the previous paragraph. To see that \ref{TC3} holds, suppose $\tar(\alpha'_j) = \src(\beta'_k)$, meaning $k = j+1$ and these arcs share a bottom endpoint. Now the top endpoint of $\beta'_k$ is $w_{j+2}$, and so if $\alpha'_j$ passes alongside this endpoint, it must pass on the left side. Likewise, if $\beta'_k$ passes alongside the top endpoint of $\alpha'_j$ (which is the node $w_j$), it must pass on the right. Finally, if $w_i \in \suppo(\alpha'_j)\cap \suppo(\beta'_k)$ and $w_i$ is left of $\alpha'_j$, then $i < j$. Since $k = j+1$ we have $w_i$ is also left of $\beta'_k$. We conclude that $\alpha'_j$ is left of $\beta'_k$, and so \ref{TC3} is satisfied. The argument that \ref{TC4} is satisfied is analogous.
     
     We have shown that $(\mathcal{G}'',\mathcal{R}'')$ is a 2-colored noncrossing arc diagram. The fact that it is complete then follows from Theorem \ref{thm:fullSizeComplete}. Now define $\mathcal{G}' = \{\alpha'_j:\alpha_{w_j} \in \mathcal{G}\}$ and define $\mathcal{R}'$ analogously. It follows that $(\mathcal{G}',\mathcal{R}')$ is a completable 2-colored noncrossing arc diagram. Moreover, this diagram is support equivalent to $(\mathcal{G},\mathcal{R})$ under the bijections $\alpha_{w_j} \mapsto \alpha'_j$ and $\beta_{w_k}\mapsto \beta'_k$.
\end{proof}

\section{Semibrick pairs for the type $A$ preprojective algebra}\label{sec:typeA}

In this section, we recall the definition of the preprojective algebra $\Pi_W$ associated to a finite Weyl group $W$. We then show that, in type $A$, there is a bijection between the semibrick pairs of $\mods \Pi_W$ and the 2-colored noncrossing arc diagrams constructed in Section~\ref{sec:arc-diagrams}. Moreover, we show that this bijection preserves completability. Finally, we deduce Corollary~\ref{Main_C} from this bijection.

We begin with the following definitions.

\begin{define}\label{def:preproj}
    Let $W$ be a finite Weyl group of type A, D, or E, and let $Q = (Q_0,Q_1)$ be a Dynkin quiver of the same type (with arbitrary orientation). Define an additional set of arrows
    $$Q_1^* = \left\{i\xrightarrow{\alpha^*} j \ \middle\vert \ \left(j \xrightarrow{\alpha} i\right) \in Q_1\right\}$$
    and let $\overline{Q} = (Q_0,Q_1\cup Q_1^*)$.
    \begin{enumerate}
        \item Let $x = \sum_{\alpha \in Q_1}(\alpha\alpha^* + \alpha^*\alpha)$. Then the \emph{preprojective algebra} of type $W$ is the bound quiver algebra $\Pi_W:= K\overline{Q}/(x)$, where $(x)$ is the two-sided ideal generated by $x$.
        \item Let $C = \bigcup_{\alpha \in Q_1} \{\alpha\alpha^*, \alpha^*\alpha\}$. We denote $RW:= K\overline{Q}/(C)$, where $(C)$ is the two-sided ideal generated by $C$.
    \end{enumerate}
\end{define}

\begin{ex}\label{ex:preproj}
Consider the following quivers:
$$
    \begin{tikzcd}[row sep=tiny]
        &&&&&&2\arrow[dr,"\alpha" below,shift right]\\
        \overline{Q}_A:&1 \arrow[r,"\alpha" below,shift right] & 2 \arrow[l,swap,"\alpha*" above,shift right]\arrow[r,"\beta*" below,shift right]&3,\arrow[l,swap,"\beta" above,shift right]&&\overline{Q}_D:&&1\arrow[ul,swap,"\alpha*" above,shift right]\arrow[r,"\beta*" below,shift right]\arrow[dl,swap,"\gamma*" above,shift right]&3\arrow[l,swap,"\beta" above,shift right].\\
        &&&&&&4\arrow[ur,"\gamma" below,shift right]
    \end{tikzcd}
    $$
    Then:
    \begin{eqnarray*}
        \Pi_{A_3} &=& K\overline{Q}_A/(\alpha\alpha^*, \beta\beta^*, \alpha^*\alpha+\beta^*\beta)\\
        RA_3 &=& K\overline{Q}_A/(\alpha\alpha^*, \beta\beta^*, \alpha^*\alpha, \beta^*\beta)\\
        \Pi_{D_4} &=& K\overline{Q}_D/(\alpha\alpha^*, \beta\beta^*, \gamma\gamma^*, \alpha^*\alpha+\beta^*\beta+\gamma^*\gamma)
    \end{eqnarray*}
    More generally, in type $A_n$, we identify the vertices of the corresponding quiver with $\{1,\ldots,n\}$ so that two vertices are joined by an arrow if and only if they are consecutive.
\end{ex}

We note that in general, $RA_n$ is a gentle algebra with no bands (since any cyclic word in the
alphabet $\overline{Q}_1$ would have a 2-cycle and all 2-cycles lie in $(C)$).
Thus, each indecomposable module over $RA_n$ is an orientation $Q_M$ of the full subquiver supporting $M$ (where, at each vertex, we place a copy of $K$, and each arrow acts by the identity map; we do not draw an arrow between $i$ and $i+1$ if both $a_i$ and $a_i^*$ act trivially.) 
Work of Butler and Ringel \cite{BR} implies that $\mods RA_n$ has finitely many indecomposables and each of these are bricks. 
(See also \cite[Propositions~4.1 and 4.2]{BCZ}.)

In type A, the next theorem allows us to dispense with $\Pi_{A_n}$, and work with the simpler algebra $RA_n$. We emphasize that this result holds only in type A, not in types D and E.

\begin{thm}\label{reduce_to_RAn}
The poset of torsion classes $\tors RA_n$ is isomorphic to $\tors \Pi_{A_n}$.
In particular,
\begin{enumerate}
\item there is a bijection from the set of bricks of $\Pi_{A_n}$ to the set of bricks of $RA_n$, and
\item a semibrick pair of $\Pi_{A_n}$ is completable if and only if the corresponding semibrick pair is completable for $RA_n$.
\end{enumerate}
\end{thm}

\begin{proof}
\cite[Theorem~0.2]{mizuno} says that the lattice of torsion classes of $\Pi_{A_n}$ is isomorphic to the weak order on the type $A_n$ Weyl group.
\cite[Theorem~4.3.8]{BCZ} says that the the lattice of torsion classes of $RA_n$ is also isomorphic to the weak order.
(We review the bijection involved in this isomorphism in Section~\ref{sec:arcs_to_bricks}.)
This proves the first item.
The second item follows from the first.
\end{proof}

\subsection{From arcs to bricks}\label{sec:arcs_to_bricks}
In this section we recall a bijection from the arcs we have worked with in Section~\ref{sec:arcs} to the bricks in $\mods RA_n$ as described in \cite[Section~4]{BCZ}.

We recall that each indecomposable module in $\mods RA_n$ is a brick and that for $M \in \mods RA_n$ a brick, there exist $p \leq q \in [n]$ so that the dimension vector of $M$ is $\underline{\dim}(M) = \sum_{j = p}^q e_j$. The \emph{support} of $M$ is then $\supp(M) = [p,q]$. When $\supp(M) = [1,n]$, we say that $M$ has \emph{full support}.

Recall that the arrows in $RA_n$ are of the form $a_i: i \rightarrow i+1$ and $a_i^*: i+1\rightarrow i$, for  $i=1,\dotsc, n-1$ (and each two cycle $a_ia_i^*$ and $a_i^*a_i$ is in the ideal defining $RA_n$).
Thus each indecomposable module $M$ is determined uniquely by its support and which one of the arrows $a_r$ or $a_r^*$ acts non-trivially on~$M$.

In order to establish the correspondence between arcs and bricks, we recall the following definition.

\begin{define}
    Let $\alpha$ and $\beta$ be arcs on $n+1$ nodes. We say that $\beta$ is a \emph{subarc} of $\alpha$ if both of the following conditions are satisfied:
    \begin{enumerate}
        \item $\supp(\beta)\subseteq \supp(\alpha)$; 
        \item $\alpha$ and $\beta$ pass on the same side of each node in $\suppo(\beta)$.
\end{enumerate}
\end{define}
\begin{rem}
    We emphasize that the definition of a subarc is agnostic to the color/orientation of the arcs. In particular, a green arc and a red arc can have a common subarc (as will be the case in many of our proofs). Thus we generally consider subarcs as having arbitrary color/orientation.
\end{rem}

    Let $\alpha$ be a (green or red) arc on $n+1$ nodes and write $\supp(\alpha) = [p,q]$. We define $\sigma(\alpha)$ to be the brick in $\mods RA_n$ with support $[p,q-1]$ such that for each $i\in \suppo(\alpha)$ 
\begin{enumerate}
\item If $\alpha$ passes to the right of $i$, then $a_{i-1}$ acts nontrivially on $\sigma(\alpha)$; and
\item If $\alpha$ passes to the left of $i$, then $a^*_{i-1}$ acts nontrivially on $\sigma(\alpha)$.
\end{enumerate}

The next proposition is a combination of \cite[Proposition 4.7 and Theorem~4.13]{BCZ}, and completely characterizes the semibricks in $\mods RA_n$.

\begin{prop}\label{prop: arc bij}
The map $\sigma$ from the set of green (resp. red) arcs on $n+1$ nodes to the set of bricks in $\mods RA_n$ is a bijection.
Moreover, a collection of green (resp. red) arcs is a noncrossing arc diagram if and only if the corresponding set of bricks is a semibrick.
\end{prop}

\begin{rem}
    An alternative combinatorial description of the bricks and  semibricks over preprojective algebras (and hence over $RA_n$) is given in \cite{asai_preproj}.
\end{rem}

As a consequence of Proposition~\ref{prop: arc bij} and Theorem \ref{torsion_to_simple}, we have the following commutative diagram:

\[
\begin{tikzcd}[row sep=large, column sep=large]
A_n \arrow[two heads, rightarrow]{d}{\delta} \arrow[rrd, dashed,to path=-| (\tikztotarget)] &  \\
\{\text{green noncrossing arc diagrams}\}  \arrow{r}{\sigma} & \{\D(\T): \, \T\in \tors RA_n\}\arrow{r}{\Filt\Fac(-)}&\tors RA_n\\
\end{tikzcd}
\]

One of the main results of \cite[Section~4.3]{BCZ} says that the composition in the above diagram is a poset isomorphism from the weak order on $A_n$ to the poset of torsion classes on $\tors RA_n$.

\begin{rem}\label{rem: arc bij}
There is an analogous commutative diagram where we replace $\delta$ with $\bar{\delta}$, ``green noncrossing arc diagrams'' with ``red noncrossing arc diagrams'', $\D(\T)$ with $\U(\T)$, and $\Filt\Fac(-)$ with $^\perp(-)$. In particular, for $w \in A_n$ there exists $\T \in \tors\Lambda$ so that $\sigma\circ\delta(w) = \D(\T)$ and $\sigma\circ\overline{\delta}(w) = \U(\T)$. By Theorem~\ref{torsion_to_simple}, this means $\sigma\circ\delta(w)\sqcup\sigma\circ\overline{\delta}(w)[1]$ is a 2-term simple minded collection. 
\end{rem}

As a consequence, we obtain a combinatorial criteria for deciding when a semibrick pair of $RA_n$ is completable.

\begin{prop}\label{prop:completability_criterion}
Let $\D\sqcup \U[1]$ be a semibrick pair in $\mods RA_n$. Then
\begin{enumerate}
    \item $\D\sqcup \U[1]$ is completable if and only if there exists a permutation $w \in A_n$ so that $\D\subseteq \sigma\circ\delta(w)$ and $\U\subseteq \sigma\circ\overline{\delta}(w)$.
    \item $\D\sqcup \U[1]$ is a 2-term simple minded collection if and only if there exists a permutation  $w \in A_n$ so that $\D =  \sigma\circ\delta(w)$ and $\U= \sigma\circ\overline{\delta}(w)$.
\end{enumerate}
\end{prop}

\subsection{Semibrick pairs and 2-colored noncrossing arc diagrams}\label{sec:ThmD}
In this section, we show that the correspondence between bricks and arcs from Section~\ref{sec:arcs_to_bricks} extends to a correspondence between semibrick pairs and 2-colored noncrossing arc diagrams; i.e., we prove Theorem~\ref{Main_B}. We then deduce Corollary~\ref{Main_C} (as Corollaries~\ref{cor:Main_C1} and~\ref{cor:Main_C2} below) from this correspondence and the results of Section~\ref{sec:arc-diagrams}.

We begin by giving the formal statement of Theorem~\ref{Main_B}.

\begin{thm}[Theorem~\ref{Main_B}]\label{arc_pairwise_compat}\
\begin{enumerate}
    \item There is a bijection between semibrick pairs for the algebra $RA_n$ and  2-colored noncrossing arc diagrams on $n+1$ nodes given by
$$\D\sqcup \U[1] \mapsto (\sigma^{-1}(\D),\sigma^{-1}(\U)).$$
    \item Let $\D\sqcup \U[1]$ be a semibrick pair in $\mods RA_n$. Then $\D\sqcup \U[1]$ is completable (resp. is a 2-term simple minded collection) if and only if $(\sigma^{-1}(\D),\sigma^{-1}(\U))$ is a completable (resp. complete) 2-colored noncrossing arc diagram.
\end{enumerate}
\end{thm}

We note that Theorem~\ref{arc_pairwise_compat} serves as our justification for the  definitions of 2-colored noncrossing arc diagrams and their completability and completeness. See Remark \ref{rem:whyDef}.

We break the bulk of the proof of Theorem~\ref{arc_pairwise_compat} into  Lemmas~\ref{lem:intersect_1}, \ref{lem:intersect_2}, \ref{lem:same_support}, \ref{lem:left}, and~\ref{lem:extBothWays}.
The arguments closely follow \cite[Section~4.2]{BCZ}, in which arc crossings are shown to correspond to nonzero homomorphisms between the corresponding bricks.
To that end we recall the following definition.

\begin{define}\label{def:pred_arc}
Let $\alpha$ be an arc on $n+1$ nodes and suppose that $\gamma$ is a subarc of $\alpha$.
We say that $\gamma$ is a \emph{predecessor closed subarc} if $\alpha$ does not pass to the right of the bottom endpoint of $\gamma$ nor to the left of its top endpoint.
The arc $\gamma$ is a \emph{successor closed subarc} if $\alpha$ does not pass to the left of the bottom endpoint of $\gamma$ nor to the right of its top endpoint.
\end{define}

The  next theorem follows from well-known work of Crawley-Boevey.
See \cite[Section~2]{CB} and \linebreak \cite[Proposition~4.7]{BCZ}.

\begin{thm}\label{thm:BC}
Let $\alpha$ be an arc on $n+1$ nodes and let $\gamma$ be a subarc of $\alpha$.
Let $S = \sigma(\alpha)$ be the brick corresponding to $\alpha$ and let $T = \sigma(\beta)$ be the brick corresponding to $\gamma$.
Then $\gamma$ is a predecessor (resp. successor) closed subarc of $\alpha$ if and only if $T$ is a quotient (resp. submodule) of $S$.
\end{thm}

In \cite{BDMHY}, the authors study the combinatorics of gentle algebras, and give an explicit basis of $\Ext_\Lambda^1(S,T)$ for any indecomposable modules $S$ and $T$.
In particular, they show that certain nonzero homomorphisms from $T$ to $S$ give rise to extensions with decomposable middle terms.
Such \linebreak homomorphisms are called \emph{two-sided graph maps}.
See \cite[Definition~2.6]{BDMHY}.
In the first of our four lemmas, we use the arc-analog of two-sided graph maps, which we call \emph{two-sided arc maps}.

In what follows, it will be useful to consider ``empty'' arcs. More precisely, for $i \in [n+1]$, we define the \emph{empty arc} at the node $i$ to have top and bottom endpoint both equal to $i$. We emphasize that empty arcs are not arcs in the sense of Definition~\ref{def:red_arc}. Thus for consistency, we will use the term ``possibly empty arc'' if we wish to consider both arcs (in the sense of Defintion~\ref{def:red_arc}) and empty arcs.

Suppose that $\gamma$ is a subarc of two distinct arcs, $\alpha$ and $\beta$.
Then $\gamma$ induces a decomposition of $\alpha$ into three possibly empty subarcs $\alpha_1, \gamma$, and $\alpha_2$ as follows:
The bottom endpoint of $\alpha_1$ is the same as the bottom endpoint of $\alpha$, and its top endpoint is the bottom endpoint of $\gamma$.
The bottom endpoint of $\alpha_2$ is equal to the top endpoint of $\gamma$, and its top endpoint is equal to the top endpoint of $\alpha$.
In the same way, $\gamma$ also decomposes $\beta$ into three possibly empty subarcs $\beta_1$, $\gamma$, and $\beta_2$. See Example \ref{ex:ext} and the middle diagram of Figure~\ref{fig:ext} for an example.

\begin{define}\label{def:two_sided}
In the notation of the preceding paragraph, we say that the pair of triples $((\alpha_1, \gamma, \alpha_2),(\beta_1, \gamma,\beta_2))$ is a \emph{two-sided arc map} provided that (a) $\gamma$ is a successor closed subarc of $\alpha$ and a predecessor closed subarc of $\beta$, (b) $\alpha_1$ and $\beta_1$ are not both empty, and (c) $\alpha_2$ and $\beta_2$ are not both empty.
\end{define}

Suppose that $S$ and $T$ are bricks with corresponding arcs $\alpha$ and $\beta$, respectively.
Assume that $\gamma$ is a subarc of both $\alpha$ and $\beta$ such that the triple  $((\alpha_1, \gamma, \alpha_2),(\beta_1, \gamma,\beta_2))$ is a two-sided arc map.
Let $\eta_1$ be the arc determined by the following:
\begin{enumerate}
\item its bottom endpoint equal to the bottom endpoint of $\alpha$;
\item its top endpoint equal to the top endpoint of $\beta$; 
\item it has $\alpha_1$, $\gamma$ and $\beta_2$ as subarcs;
\item $\alpha$ and $\eta_1$ pass on the same side of the bottom endpoint of $\gamma$; and
\item $\beta$ and $\eta_1$ pass on the same side of the top endpoint of $\gamma$.
\end{enumerate}
Define an arc $\eta_2$ symmetrically.
See the rightmost diagram of Figure~\ref{fig:ext}.

The next theorem follows immediately from \cite[Theorem~8.5]{BDMHY}. See also~\cite{schroer}.

\begin{thm}\label{thm:two_sided}
Suppose that $S \neq T$ are bricks in $\mods RA_n$ with corresponding arcs $\alpha = \sigma^{-1}(S)$ and $\beta = \sigma^{-1}(T)$. Suppose there exists a common subarc $\gamma$ of $\alpha$ and $\beta$ such that the triple $((\alpha_1,\gamma,\alpha_2),(\beta_1,\gamma,\beta_2))$ is a two-sided arc map. For $i \in \{1,2\}$, let $\eta_i$ be defined as in the previous paragraph and let $E_i = \sigma(\eta_i)$ be the corresponding brick.
Then
\begin{enumerate}
    \item There is a nonzero extension in $\Ext_\Lambda^1(S,T)$ whose middle term is $E_1\oplus E_2$.
    \item There is a nonzero morphism in $\Hom_\Lambda(T,S)$.
\end{enumerate}
In particular, each two-sided arc map for $\alpha$ and $\beta$ corresponds to unique extension in $\Ext_\Lambda^1(S,T)$. Moreover, if there is no two-sided arc map for $\alpha$ and $\beta$, then any nonzero extension in $\Ext_\Lambda^1(S,T)$ has an indecomposable middle term.
\end{thm}
An explicit example of a two-sided arc map is shown in Example~\ref{ex:ext}.

\begin{lem}\label{lem:intersect_1}
Let $S\neq T$ be bricks in $\mods RA_n$, let $\alpha = \sigma^{-1}(S)$ be the arc corresponding to $S$ and let $\beta = \sigma^{-1}(T)$ be the arc corresponding to $T$. Suppose $\alpha$ and $\beta$ do not share any endpoints. Then $S\sqcup T[1]$ is a semibrick pair (of rank 2) if and only if $\alpha$ and $\beta$ do not cross in their interiors.
\end{lem}

\begin{proof}
Assume that $\alpha$ and $\beta$ cross in their interiors.
\cite[Lemma~4.10]{BCZ} says that there exists an arc $\gamma$ that is a predecessor closed subarc of either $\alpha$ or $\beta$ which is also a successor closed subarc of the other.
If $\gamma$ is a predecessor closed subarc of $\alpha$ and a successor closed subarc of $\beta$, then $S$ surjects onto $\sigma(\gamma)$, and $\sigma(\gamma)$ maps into $T$.
This means $\Hom_\Lambda(S,T) \neq 0$, and so $S\sqcup T[1]$ is not a semibrick pair.

Otherwise, $\gamma$ must be a predecessor closed subarc of $\beta$ and a successor closed subarc of $\alpha$.
Define $\alpha_i$ and $\beta_i$ as in the paragraph preceding Definition~\ref{def:two_sided}.
We argue that $((\alpha_1, \gamma, \alpha_2),(\beta_1, \gamma,\beta_2))$ is a two-sided graph map.

Since $\alpha$ and $\beta$ do not share any endpoints, we cannot have $\alpha = \gamma = \beta$. If $\alpha \neq \gamma \neq \beta$, then at most one of $\alpha_1, \beta_1,\alpha_2,$ and $\beta_2$ is empty.
(One of these four arcs is empty when $\gamma$ shares a top or bottom endpoint with $\alpha$ or $\beta$.)
Now assume that $\gamma$ is equal to either $\alpha$ or $\beta$.
If $\gamma$ is equal to $\alpha$, then both $\beta_1$ and $\beta_2$ are nonempty.
The situation is the same when $\gamma$ is equal to $\beta$.
In all cases, $((\alpha_1, \gamma, \alpha_2),(\beta_1, \gamma,\beta_2))$ is a two-sided arc map.
Therefore, $\Ext_\Lambda^1(S,T)\ne 0$, and so $S\sqcup T[1]$ is not a semibrick pair.

    Now suppose that $\alpha$ and $\beta$ do not cross in their interiors. By Proposition \ref{prop: arc bij}, this means that $\Hom_\Lambda(S,T) = 0 = \Hom_\Lambda(T,S)$. It then follows from Theorem \ref{thm:two_sided} that the middle term of any nonzero extension in $\Ext_\Lambda^1(S,T)$ must be indecomposable (and thus a brick). This means that $\supp(S) \cap \supp(T) = \emptyset$ and there exists an arc $\gamma$ so that $\supp(S) \cup \supp(T) = \supp(\sigma(\gamma))$. Translated to arcs, we then have that $\supp(\alpha)\cap \supp(\beta)$ contains precisely one node which is an endpoint of both $\alpha$ and $\beta$, a contradiction. We conclude that $\Ext_\Lambda^1(S,T) = 0$ and so $S\sqcup T[1]$ is a semibrick pair.
\end{proof}

\begin{ex}\label{ex:ext}
    Let $\alpha$ be the solid blue arc (considered as a green arc) in the left diagram of Figure~\ref{fig:ext} and let $\beta$ be the dashed orange arc (considered as a red arc) in the left diagram of Figure~\ref{fig:ext}. These arcs correspond to the bricks {\footnotesize $\begin{matrix}24\\3\end{matrix}$} and ${\footnotesize\begin{matrix}2\\13\end{matrix}}$ respectively. The common subarc $\gamma$ is the black arc in the center diagram of Figure \ref{fig:ext}. This is a predecessor closed subarc of $\beta$ and a successor closed subarc of $\alpha$. The arcs $\alpha_2$ and $\beta_1$ in the factorizations of $\alpha$ and $\beta$ are the solid blue and dashed orange arcs in the center diagram of Figure \ref{fig:ext}. The arcs $\alpha_1$ and $\beta_2$ are both empty. The third figure shows the arcs $\eta_1$ and $\eta_2$ corresponding to the middle terms of the extension
    $${\footnotesize \begin{matrix}2\\13\end{matrix}}\hookrightarrow {\footnotesize\begin{matrix}\phantom{2}24\\13\end{matrix}}\sqcup{\footnotesize\begin{matrix}2\\3\end{matrix}}\twoheadrightarrow {\footnotesize \begin{matrix}24\\3\end{matrix}}.$$
This extension shows that $\footnotesize {\begin{matrix}24\\3\end{matrix}}\sqcup{\footnotesize\begin{matrix}2\\13\end{matrix}}[1]$ is not a semibrick pair.
\end{ex}

 \begin{figure}[h]
 \centering
\begin{tikzpicture}


\begin{scope}[decoration={markings,mark = at position 0.1 with {\arrow{>}},mark=at position 0.5 with {\arrow{>}},mark=at position 0.9 with {\arrow{>}}}]
\draw[very thick,orange,dashed,smooth,postaction=decorate](0,-2) to [out = 150, in = -150] (0,-0.5) to [out = 30, in = -30] (0,1);
\end{scope}

\begin{scope}[decoration={markings,mark = at position 0.15 with {\arrow{<}},mark=at position 0.5 with {\arrow{<}},mark=at position 0.85 with {\arrow{<}}}]
\draw[very thick,blue,smooth,postaction=decorate](0,-1) to [out = 30, in = -30] (0,0.5) to [out = 150, in = -150] (0,2);
\end{scope}

\draw[fill] (0,1) circle [radius=.07];
\draw[fill] (0,-2) circle [radius=.07];
\draw[fill] (0,0) circle [radius=.07];
\draw[fill] (0,2) circle [radius=.07];
\draw[fill] (0,-1) circle [radius=.07];


\draw[thick] (2,-1) to [out = 30, in = -30] (2,1);

\begin{scope}[decoration={markings,mark=at position 0.6 with {\arrow{<}}}] (2,1) to (2,2);
\draw[very thick,blue,postaction=decorate] (2,1) to (2,2);
\end{scope}

\begin{scope}[decoration={markings,mark=at position 0.5 with {\arrow{>}}}]
\draw[very thick,orange,dashed,postaction=decorate] (2,-2) to (2,-1);
\end{scope}

\draw[fill] (2,1) circle [radius=.07];
\draw[fill] (2,-2) circle [radius=.07];
\draw[fill] (2,0) circle [radius=.07];
\draw[fill] (2,2) circle [radius=.07];
\draw[fill] (2,-1) circle [radius=.07];


\draw[thick,black](4,-2) to [out = 150, in = -150] (4,-0.5) to [out = 30, in = -30] (4,0.5) to [out = 150, in = -150] (4,2);
\draw[thick,black](4,-1) to [out = 30, in = -30] (4,1);

\draw[fill] (4,1) circle [radius=.07];
\draw[fill] (4,-2) circle [radius=.07];
\draw[fill] (4,0) circle [radius=.07];
\draw[fill] (4,2) circle [radius=.07];
\draw[fill] (4,-1) circle [radius=.07];

\end{tikzpicture}
\caption{The arcs in Example \ref{ex:ext}}
\label{fig:ext}
\end{figure} 

\begin{lem}\label{lem:intersect_2}
Let $S\neq T$ be bricks in $\mods RA_n$. Let $\alpha = \sigma^{-1}(S)$ be the arc corresponding to $S$, let $\beta = \sigma^{-1}(T)$ be the arc corresponding to $T$, and
suppose $\alpha$ and $\beta$ share a bottom endpoint or a top endpoint. Further suppose that $\alpha$ and $\beta$ cross in their interiors. Then $S\sqcup T[1]$ is not a semibrick pair.
\end{lem}

\begin{proof}
By symmetry, we can assume that $\alpha$ and $\beta$ share a bottom endpoint (so that $\tar(\beta)=\src(\alpha)$).
We will argue that there is a nonzero map from $S$ to $T$, and so $S \sqcup T[1]$ is not a semibrick pair.
Without loss of generality, we may take the bottom endpoint to be the lowest one, labeled by $1$.
Observe that neither $\alpha$ nor $\beta$ has its top endpoint at $2$ (otherwise they could be drawn so that they do not cross).

Let $k$ be the smallest node such that $\alpha$ and $\beta$ pass on different sides of $k$.
(Such a $k$ exists because otherwise the two arcs can be drawn so that they never intersect in their interiors.)
We break into two cases:
In the first case, $\alpha$ passes on the right side of $k$ and $\beta$ passes on the left.
In this case, we can draw $\alpha$ and $\beta$ so that they do not intersect anywhere in their support along the vertices $1,2, \ldots k$.
When we do this, $\alpha$ lies strictly to the right of $\beta$.
Let $\gamma$ be the subarc of $\alpha$ with endpoints $1$ and $k$.
From our construction, $\gamma$ is also a subarc of $\beta$.
Indeed, $\gamma$ is a predecessor closed subarc of $\alpha$ and a successor closed subarc of $\beta$.
Thus, there exists a brick $M_\gamma$ such that $S$ maps onto $M_\gamma$ and $M_\gamma$ maps into $T$, giving rise to a nonzero homomorphism from $S$ to $T$.
Therefore, the first case is impossible.

In the second case, $\alpha$ passes on the left side of $k$ and $\beta$ passes on the right side.
Since the two arcs intersect in their interiors, there exists a node $j>k$ such that $\alpha$ does not pass to the left of $j$ and $\beta$ does not pass to the right side of $j$.
Take $j$ to be as small as possible (and still greater than~$k$).
Set $j'$ to be the largest element in the set 
\[\{x: k\le x<j\text{ and $\alpha$ passes to the left of $x$ and $\beta$ passes to the right of $x$}\}.\]
Note that $j'$ exists because this set is not empty ($k$ is a member).

Now consider the subarc $\gamma'$ of $\alpha$ with endpoints $j'$ and $j$.
For any point $i$ between $j'$ and $j$, if $\alpha$ passes to the left of $i$ then $\beta$ must also pass to the left (by the definition of $j'$).
If $\alpha$ passes to the right side of $i$ then $\beta$ must also pass to the right (by the definition of $j$).
Therefore $\gamma'$ is also a subarc of $\beta$.
Indeed, $\gamma'$ is a predecessor closed subarc of $\alpha$ and a successor closed subarc of $\beta$.
As in the previous case, there is a nonzero homomorphism from $S$ to $T$.
\end{proof}

\begin{ex}\label{ex:endptInt}
    Let $\alpha$ be the solid blue arc (considered as a green arc) in the left diagram of Figure~\ref{fig:endptInt} and let $\beta$ be the dashed orange arc (considered as a red arc) in the left diagram of Figure~\ref{fig:endptInt}. These arcs correspond to the bricks ${\tiny \begin{matrix}1\phantom{2}\\24\\3\end{matrix}}$ and ${\footnotesize \begin{matrix}13\\2\end{matrix}}$, respectively. The smallest node of which $\alpha$ and $\beta$ pass on opposite sides is $k = 3$. In this example, $\alpha$ passes on the right side of $k$ and $\beta$ passes on the left side of $k$. Moreover, this is the only node with this property. Thus the arc $\gamma'$ shown in the right diagram of Figure \ref{fig:endptInt} is a predecessor closed subarc of $\alpha$ and a successor closed subarc of $\beta$. This subarc corresponds to the nonzero morphism
    $${\tiny \begin{matrix}1\phantom{2}\\24\\3\end{matrix}}\rightarrow {\footnotesize \begin{matrix}13\\2\end{matrix}}$$
    which has image ${\footnotesize\begin{matrix}1\\2\end{matrix}}$. The existence of this morphism shows that ${\tiny \begin{matrix}1\phantom{2}\\24\\3\end{matrix}}\sqcup{\footnotesize\begin{matrix}13\\2\end{matrix}}[1]$ is not a semibrick pair.
\end{ex}

 \begin{figure}[h]
 \centering
\begin{tikzpicture}

\begin{scope}[decoration={markings,mark = at position 0.25 with {\arrow{>}},mark=at position 0.5 with {\arrow{>}},mark=at position 0.85 with {\arrow{>}}}]
\draw[very thick,orange,dashed,smooth,postaction=decorate](0,-2) to [out = 40, in = -30] (0,-0.5) to [out = 150, in = -150] (0,1);
\end{scope}

\begin{scope}[decoration={markings,mark = at position 0.15 with {\arrow{<}},mark=at position 0.5 with {\arrow{<}},mark=at position 0.85 with {\arrow{<}}}]
\draw[very thick,blue,smooth,postaction=decorate](0,-2) to [out = 20, in = -15] (0,0.5) to [out = 165, in = -150] (0,2);
\end{scope}

\draw[fill] (0,1) circle [radius=.07];
\draw[fill] (0,-2) circle [radius=.07];
\draw[fill] (0,0) circle [radius=.07];
\draw[fill] (0,2) circle [radius=.07];
\draw[fill] (0,-1) circle [radius=.07];


\draw[thick] (3,-2) to [out = 30, in = -30] (3,0);

\draw[fill] (3,1) circle [radius=.07];
\draw[fill] (3,-2) circle [radius=.07];
\draw[fill] (3,0) circle [radius=.07];
\draw[fill] (3,2) circle [radius=.07];
\draw[fill] (3,-1) circle [radius=.07];

\end{tikzpicture}
\caption{The arcs in Example \ref{ex:endptInt}}
\label{fig:endptInt}
\end{figure} 

\begin{lem}\label{lem:same_support}
    Let $S \neq T$ be bricks in $\mods RA_n$. Let $\alpha = \sigma^{-1}(S)$ be the arc corresponding to $S$, let $\beta = \sigma^{-1}(T)$ be the arc corresponding to $T$, and suppose that $\supp(\alpha) = \supp(\beta)$. Then $S\sqcup T$ is not a semibrick and $S\sqcup T[1]$ is not a semibrick pair.
\end{lem}

\begin{proof}
    The proof is similar to that of Lemma~\ref{lem:intersect_2}. In particular, we assume without loss of generality that the bottom endpoint of both $\alpha$ and $\beta$ is the node 1. Now, since we have assumed $S \neq T$, there exists a smallest node $k$ such that $\alpha$ and $\beta$ pass on different sides of $k$. (In particular, it must be the case that $|\supp(\alpha)| > 2$.) We break into two cases: In the first case, $\alpha$ passes on the right side of $k$ and $\beta$ passes to the left. Now let $\gamma$ be the subarc of $\alpha$ with endpoints 1 and $k$. As in the proof of Lemma~\ref{lem:intersect_2}, $\gamma$ is a predecessor closed subarc of $\alpha$ and a successor closed subarc of $\beta$. In particular, there is a nonzero homomorphism from $S$ to $T$.
    
    In the second case, $\alpha$ passes on the left side of $k$ and $\beta$ passes on the right side. Now, if $\alpha$ and $\beta$ cross in their interiors, we are in the situation of Lemma~\ref{lem:intersect_2}. We thus assume that $\alpha$ and $\beta$ do not cross in their interiors. We denote by $j$ the top endpoint of $\alpha$ and by $j'$ the largest node such that $\alpha$ passes to the left of $j'$ and $\beta$ passes to the right of $j'$. We consider the subarc $\gamma'$ of $\alpha$ with endpoints $j'$ and $j$. As in the proof of Lemma~\ref{lem:intersect_2}, we then have that $\gamma'$ is a predecessor closed subarc of $\alpha$ and a successor closed subarc of $\beta$. In particular, there is a nonzero homomorphism from $S$ to $T$.
\end{proof}

\begin{lem}\label{lem:left}
    Let $S\neq T$ be bricks in $\mods RA_n$, let $\alpha = \sigma^{-1}(S)$ be the arc corresponding to $S$, and let $\beta = \sigma^{-1}(T)$ be the arc corresponding to $T$.
    \begin{enumerate}
        \item If $\alpha$ and $\beta$ share a bottom endpoint and $\alpha$ is left of $\beta$, then $S \sqcup T[1]$ is a semibrick pair.
        \item If $\alpha$ and $\beta$ share a top endpoint and $\beta$ is left of $\alpha$, then $S\sqcup T[1]$ is a semibrick pair.
    \end{enumerate}
\end{lem}

\begin{proof}
    We prove only (1) as the proof of (2) is similar. Without loss of generality, we may take the bottom endpoint to be the lowest one, labeled by~1.
    
    Assume for a contradiction that there exists a nonzero morphism from $S$ to $T$. Then there exists an arc $\gamma$ which is a predecessor closed subarc of $\alpha$ and a successor closed subarc of $\beta$ by Theorem~\ref{thm:BC}. Moreover, there is a nonzero morphism from $S$ to $T$ with image $\sigma(\gamma)$.
    
    Denote by $b$ and $t$ be the bottom and top endpoint of $\gamma$, respectively. If $b$ is not the bottom endpoint of $\alpha$, then by definition $\alpha$ passes to the right of $b$ and either $\beta$ passes to the left of $b$ or $b$ is the bottom endpoint of $\beta$. This contradicts the fact that $\alpha$ is left of $\beta$, and so $b$ must be the bottom endpoint of $\alpha$. Analogous reasoning shows that $b$ is the bottom endpoint of $\beta$ as well. By symmetry, this also means that $t$ is the top endpoint of both $\alpha$ and $\beta$. We then have that $\supp(S) = \supp(\sigma(\gamma)) = \supp(T)$, and so $S = \sigma(\gamma) = T$, a contradiction.
    
    It remains to show that $\Ext^1_\Lambda(S,T) = 0$. Assume for a contradiction that this is not the case. Since $\supp(S) \cap \supp(T) \neq \emptyset$, the middle term of any nonzero extension in $\Ext^1_\Lambda(S,T)$ must be decomposable. By Theorem \ref{thm:BC}, this means there must exist a two-sided arc map for $\alpha$ and $\beta$. Let $((\alpha_1,\gamma,\alpha_2),(\beta_1,\gamma,\beta_2))$ be such a map.
    
    Write $\supp(\gamma) = [p,q]$ and let $r$ be the bottom endpoint of $\alpha$ and $\beta$. Note that since at least one of $\alpha_1$ and $\beta_1$ is nonempty, we must have that $r < p$. Now since $\gamma$ is a successor closed subarc of $\alpha$, we have that either $q$ is the top endpoint of $\alpha$ or $\alpha$ passes to the right of $q$. Likewise, since $\gamma$ is a predecessor closed subarc of $\beta$, we have that either $q$ is the top endpoint of $\beta$ or $\beta$ passes to the left of $q$. Since $\alpha$ is left of $\beta$, this is only possible if $q$ is the top endpoint of both $\alpha$ and $\beta$. It then follows that $\alpha_2$ and $\beta_2$ are both empty, a contradiction.
\end{proof}

\begin{lem}\label{lem:extBothWays}
    Let $S\neq T$ be bricks in $\mods RA_n$, let $\alpha = \sigma^{-1}(S)$ be the arc corresponding to $S$, and let $\beta = \sigma^{-1}(T)$ be the arc corresponding to $T$. Suppose that the top endpoint of $\alpha$ is equal to the bottom endpoint of $\beta$ or vice versa. Then $S\sqcup T[1]$ is not a semibrick pair.
\end{lem}

\begin{proof}
    Assume that the top endpoint of $\alpha$ is equal to the bottom endpoint of $\beta$ and denote this endpoint by $i$. Note that $1 < i < n+1$. Then there exist $j$ and $k$ so that the support of $S$ is $[j,i-1]$ and the support of $T$ is $[i,k]$. But then as there is an arrow $i-1\rightarrow i$ in the quiver of $RA_n$, we have that $\Ext_\Lambda^1(S,T) \neq 0$, and so $S\sqcup T[1]$ is not a semibrick pair. The case where the bottom endpoint of $\alpha$ is equal to the top endpoint of $\beta$ is completely analogous.
\end{proof}

We are now ready to prove Theorem~\ref{arc_pairwise_compat}.

\begin{proof}[Proof of Theorem~\ref{arc_pairwise_compat}]
We prove only (1), as (2) follows immediately from (1) and Proposition~\ref{prop:completability_criterion}.

Let $\D\sqcup \U[1]$ be a semibrick pair for $RA_n$. We know by Proposition~\ref{prop: arc bij} and Remark~\ref{rem: arc bij} that $\sigma^{-1}(\D)$ is a (green) noncrossing arc diagram and $\sigma^{-1}(\U)$ is a (red) noncrossing arc diagram. Thus let $S \in \D$ and $T \in \U$ and let $\alpha = \sigma^{-1}(S)$ and $\beta = \sigma^{-1}(T)$. We note that $\supp(\alpha) \neq \supp(\beta)$ as a consequence of Lemma~\ref{lem:same_support}. Thus we only need to show that the arcs $\alpha$ and $\beta$ satisfy the axioms in Definition~\ref{def:2-colored}(1). We have already shown in Lemmas \ref{lem:intersect_1} and \ref{lem:intersect_2} that $\alpha$ and $\beta$ satisfy~\ref{TC1}. Likewise, we have shown in Lemma \ref{lem:extBothWays} that $\alpha$ and $\beta$ satisfy~\ref{TC2}.

Now let us assume that $\alpha$ and $\beta$ share an endpoint.
We will consider the case where $\tar(\alpha)=\src(\beta)$ (i.e. $\alpha$ and $\beta$ have the same bottom endpoint).
Since we have shown that $\alpha$ and $\beta$ do not cross in their interiors, one of them lies to the left of the other.
By way of contradiction, we assume that $\beta$ is left of $\alpha$.

Suppose that $\alpha$ and $\beta$ pass on the same side of every point where their supports overlap.
If $\src(\alpha)>\tar(\beta)$ then $\alpha$ must pass along the right side of $\tar(\beta)$.
Thus $\beta$ is a predecessor closed subarc of $\alpha$.
Theorem~\ref{thm:BC} implies that there is an epimorphism $S\twoheadrightarrow T$, violating the definition of semibrick pair (see Figure \ref{fig:left_right}(a)).
If $\src(\alpha)<\tar(\beta)$ then $\beta$ must pass on the left side of $\src(\alpha)$.
Thus $\alpha$ is a successor closed subarc of $\beta$.
Theorem~\ref{thm:BC} then implies that there is a monomorphism $S\hookrightarrow T$, violating the definition of semibrick pair (see Figure \ref{fig:left_right}(b)).
Finally, suppose there is some node $y$ which $\alpha$ and $\beta$ pass on different sides, and take $y$ as small as possible.
Write $x$ for $\tar(\alpha)=\src(\beta)$.
Because $\beta$ is left of $\alpha$, we have $\beta$ passing on the left side of $y$ and $\alpha$ on the right.
Note that $\alpha$ and $\beta$ pass on the same side of each node between $x$ and $y$.
Therefore, the subarc (of both $\alpha$ and $\beta)$ with bottom endpoint $x$ and top endpoint $y$ is a predecessor closed subarc of $\alpha$ and a sucessor closed subarc of $\beta$.
As in the proof of Lemma~\ref{lem:intersect_2}, there exists a nonzero homomorphism $S \rightarrow T$, again violating the definition of semibrick pair (see Figure \ref{fig:left_right}(c)).
Therefore, if $\tar(\alpha)=\src(\beta)$, then $\alpha$ is left of $\beta$. We have thus shown that $\alpha$ and $\beta$ satisfy~\ref{TC3}. The proof that $\alpha$ and $\beta$ satisfy~\ref{TC4} is analogous.

So far, we have shown that if $\D\sqcup \U[1]$ is a semibrick pair for $RA_n$, then $(\sigma^{-1}(\D),\sigma^{-1}(\U))$ is a 2-colored noncrossing arc diagram. Now let  $(\mathcal{G},\mathcal{R})$ be a 2-colored noncrossing arc diagram. To complete the proof, it remains to show that $\sigma(\mathcal{G}) \sqcup \sigma(\mathcal{R})[1]$ is a semibrick pair.

We first observe that $\sigma(\mathcal{G})$ and $\sigma(\mathcal{R})$ are semibricks by Proposition \ref{prop: arc bij}. Thus let $\alpha \in \mathcal{G}$ and $\beta \in \mathcal{R}$. If $\alpha$ and $\beta$ do not share any endpoint, then $S\sqcup T[1]$ is a semibrick pair by Lemma \ref{lem:intersect_1}. Likewise, if $\alpha$ and $\beta$ do share an endpoint, this $S \sqcup T[1]$ is a semibrick pair by Lemma \ref{lem:left}. This concludes the proof.
\end{proof}

 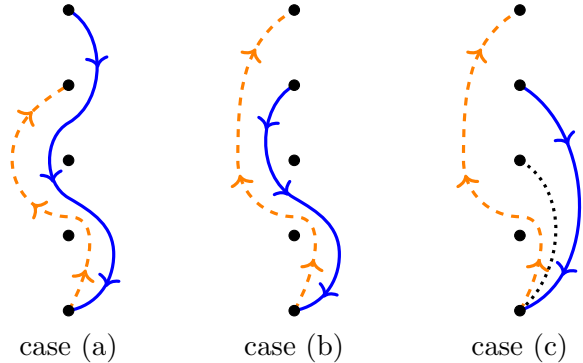
\begin{figure}[h]
 \centering
\begin{tikzpicture}

\begin{scope}[decoration={markings,mark = at position 0.15 with {\arrow{<}},mark=at position 0.5 with {\arrow{<}},mark=at position 0.85 with {\arrow{<}}}]
\draw[very thick,blue,smooth,postaction=decorate](0,-2) to [out = 10, in = 270] (0.6,-1.25) to [out = 90, in = -30] (0,-0.5) to [out = 150, in = -150] (0,0.5) to [out = 30, in = -30] (0,2);
\end{scope}

\begin{scope}[decoration={markings,mark = at position 0.15 with {\arrow{>}},mark=at position 0.5 with {\arrow{>}},mark=at position 0.85 with {\arrow{>}}}]
\draw[very thick,orange,dashed,smooth,postaction=decorate](0,-2) to [out = 60, in = 0] (0,-0.75) to [out = 180,in = -90] (-0.75,0) to [out = 90, in = -150] (0,1);
\end{scope}

\draw[fill] (0,1) circle [radius=.07];
\draw[fill] (0,-2) circle [radius=.07];
\draw[fill] (0,0) circle [radius=.07];
\draw[fill] (0,2) circle [radius=.07];
\draw[fill] (0,-1) circle [radius=.07];

\node at (0,-2.5) {case (a)};


\begin{scope}[decoration={markings,mark = at position 0.15 with {\arrow{<}},mark=at position 0.6 with {\arrow{<}},mark=at position 0.85 with {\arrow{<}}}]
\draw[very thick,blue,smooth,postaction=decorate](3,-2) to [out = 10, in = 270] (3.6,-1.25) to [out = 90, in = -30] (3,-0.5) to [out = 150, in = -150] (3,1);
\end{scope}

\begin{scope}[decoration={markings,mark = at position 0.15 with {\arrow{>}},mark=at position 0.5 with {\arrow{>}},mark=at position 0.85 with {\arrow{>}}}]
\draw[very thick,orange,dashed,smooth,postaction=decorate](3,-2) to [out = 60, in = 0] (3,-0.75) to [out=180,in=-90] (2.25,0) to [out=90,in=-150] (3,2);
\end{scope}

\draw[fill] (3,1) circle [radius=.07];
\draw[fill] (3,-2) circle [radius=.07];
\draw[fill] (3,0) circle [radius=.07];
\draw[fill] (3,2) circle [radius=.07];
\draw[fill] (3,-1) circle [radius=.07];

\node at (3,-2.5) {case (b)};


\begin{scope}[decoration={markings,mark = at position 0.25 with {\arrow{<}},mark=at position 0.75 with {\arrow{<}}}]
\draw[very thick,blue,smooth,postaction=decorate](6,-2) to [out = 20, in = -30] (6,1);
\end{scope}

\begin{scope}[decoration={markings,mark = at position 0.15 with {\arrow{>}},mark=at position 0.5 with {\arrow{>}},mark=at position 0.85 with {\arrow{>}}}]
\draw[very thick,orange,dashed,smooth,postaction=decorate](6,-2) to [out = 60, in = 0] (6,-0.75) to [out=180,in=-90] (5.25,0) to [out=90,in=-150] (6,2);
\end{scope}

\draw[very thick,black,dotted,smooth] (6,-2) to [out = 40,in = -30] (6,0);

\draw[fill] (6,1) circle [radius=.07];
\draw[fill] (6,-2) circle [radius=.07];
\draw[fill] (6,0) circle [radius=.07];
\draw[fill] (6,2) circle [radius=.07];
\draw[fill] (6,-1) circle [radius=.07];

\node at (6,-2.5) {case (c)};

\end{tikzpicture}
\caption{The three cases in the proof of Proposition \ref{arc_pairwise_compat}. In each diagram, the solid blue arc is $\alpha$ and the dashed orange arc is $\beta$. In the third diagram, the dotted black arc is the common subarc corresponding to the image of the morphism $S \rightarrow T$.}
\label{fig:left_right}
\end{figure} 

We are now ready to prove the first part of Corollary~\ref{Main_C}.

\begin{cor}[Corollary~\ref{Main_C}, Part 1]\label{cor:Main_C1}
Let $W$ be a Weyl group of type $A$ and consider a semibrick pair $\D\sqcup\U[1]$ for the preprojective algebra $\Pi_W$. Then $\D\sqcup\U[1]$ is a 2-term simple minded collection if and only if $|\D|+|\U|=n$.
\end{cor}

\begin{proof}
The ``only if'' direction follows immediately from \cite[Corollary 5.5]{KY} (see Proposition \ref{prop:maximal}). Thus assume that $|\D| + |\U| = n$. By Theorem \ref{reduce_to_RAn}, we can consider $\D\sqcup \U[1]$ as a semibrick pair in $\mods RA_n$. It then follows from Theorem~\ref{arc_pairwise_compat} and Theorem \ref{thm:fullSizeComplete} that $(\sigma^{-1}(\D),\sigma^{-1}(\U))$ is a complete 2-colored noncrossing arc diagram. Thus there exists $w \in A_n$ with $\sigma^{-1}(\D) = \delta(w)$ and $\sigma^{-1}(\U) = \overline{\delta}(w)$. Proposition \ref{prop:completability_criterion} then implies that $\D\sqcup \U[1]$ is complete.
\end{proof}

\begin{rem}\label{rem:miz_sec}
We note that Corollary~\ref{cor:Main_C1} is not equivalent to the statement proved independently and concurrently by Mizuno in \cite[Theorem~3.23]{Miz2020}. In that paper, Mizuno introduces ``double arc diagrams'', which are defined to correspond to permutations. On the other hand, the ``2-colored noncrossing arc diagrams'' considered in the present paper are defined to correspond only to semibrick pairs. To make this difference precise, let $w \in A_n$. Then the pair $(\delta(w),\overline{\delta}(w))$ coincides with one of Mizuno's double arc diagrams. On the other hand, we have not yet proved that $(\delta(w),\overline{\delta}(w))$ is a 2-colored noncrossing arc diagram. This is essentially the content of Corollary~\ref{cor:miz_sec}(1) below. Corollary~\ref{cor:miz_sec}(2) further establishes that completable 2-colored noncrossing arc diagrams are precisely those which appear as ``subdiagrams'' of Mizuno's double arc diagrams.
\end{rem}


\begin{cor}\label{cor:miz_sec}
    Let $\mathcal{G}$ be a green noncrossing arc diagram on $n+1$ nodes and let $\mathcal{R}$ be a red noncrossing arc diagram on $n+1$ nodes. Then:
    \begin{enumerate}
        \item There exists a permutation $w \in A_n$ so that $\mathcal{G} = \delta(w)$ and $\mathcal{R} = \overline{\delta}(w)$ if and only if $(\mathcal{G},\mathcal{R})$ is a 2-colored noncrossing arc diagram and $|\mathcal{G}| + |\mathcal{R}| = n$.
        \item There exists a complete 2-colored arc diagram $(\mathcal{G}',\mathcal{R}')$ with $\mathcal{G}\subseteq \mathcal{G}'$ and $\mathcal{R}\subseteq \mathcal{R}'$ if and only if $(\mathcal{G},\mathcal{R})$ is a completable 2-colored noncrossing arc diagram.
    \end{enumerate}
\end{cor}

\begin{proof}
    (1) Suppose first that there exists $w \in A_n$ so that $\mathcal{G} = \delta(w)$ and $\mathcal{R} = \overline{\delta}(w)$. In particular, this means $|\mathcal{G}| + |\mathcal{R}| = n$ by Remark \ref{rem:ascDes}. Therefore, by Propositions \ref{prop:completability_criterion} and \ref{prop:maximal}, we have that $\sigma(\mathcal{G})\sqcup \sigma(\mathcal{R})[1]$ is a 2-term simple minded collection. Proposition \ref{arc_pairwise_compat} then implies that $(\mathcal{G},\mathcal{R})$ is a 2-colored noncrossing arc diagram. The other implication follows immediately from Theorem \ref{thm:fullSizeComplete}.
    
    (2) Suppose $(\mathcal{G},\mathcal{R})$ is a completable 2-colored noncrossing arc diagram. Then Theorem~\ref{arc_pairwise_compat}(2) implies that $\sigma(\mathcal{G})\sqcup \sigma({\mathcal{R}})$ is completable; i.e., it is contained in some 2-term simple minded collection $\D\sqcup \U[1]$. It then follows immediately from Theorem~\ref{arc_pairwise_compat}(2) that $(\sigma^{-1}(\D),\sigma^{-1}(\U))$ is a complete 2-colored noncrossing arc diagram with $\mathcal{G}\subseteq \sigma^{-1}(\D)$ and $\mathcal{R} \subseteq \sigma^{-1}(\U)$. The reverse implication follows immediately from the definitions.
\end{proof}

As a final consequence of Theorem~\ref{arc_pairwise_compat}, we obtain the second part of Corollary~\ref{Main_C}.

\begin{cor}[Corollary \ref{Main_C}, Part 2]\label{cor:Main_C2}
Let $W$ be a Weyl group of type $A$ and consider a semibrick pair $\D\sqcup\U[1]$ for the preprojective algebra $\Pi_W$.  Let $$\mathcal{M} = \{\underline{\dim}(S):S \in \D\} \cup \{-\underline{\dim}(T):T \in \U\}.$$
    Then there exists a $c$-matrix $\mathcal{M}'$ for $\Pi_W$ with $\mathcal{M} \subseteq \mathcal{M}'$.
\end{cor}

\begin{proof}
    By Theorem \ref{reduce_to_RAn}, we can work over the algebra $RA_n$, rather than $\Pi_{W}$.
    Let $\D\sqcup \U[1]$ be a semibrick pair in $\mods RA_n$. Then by Proposition \ref{arc_pairwise_compat}, we have a 2-colored noncrossing arc diagram $(\sigma^{-1}(\D),\sigma^{-1}(\U))$. By Theorem \ref{thm:supportEquiv}, there then exists a completable 2-colored noncrossing arc diagram $(\mathcal{G},\mathcal{D})$ which is support equivalent to $(\sigma^{-1}(\D),\sigma^{-1}(\U))$. It is straightforward to show that $\sigma(\mathcal{G})\sqcup\sigma(\U)[1]$ is a completable semibrick pair whose bricks have the same dimension vectors as those in $\D\sqcup \U[1]$. This shows that these vectors are contained in a $c$-matrix.
\end{proof}

\begin{rem}\label{rem:notPair}
    We note that, although the dimension vectors of any semibrick pair for $\mods RA_n$ are contained in a $c$-matrix, the $c$-vectors of $\mods RA_n$ are not characterized by any pairwise conditions. Indeed, for $n = 2$, there are 2-term simple minded collections ${\tiny\begin{matrix}1\\2\end{matrix}}\sqcup2[1]$, ${\tiny\begin{matrix}2\\1\end{matrix}}\sqcup 1[1]$, and $1[1]\sqcup 2[1]$; however, the set $\{(1,1),(0,-1),(-1,0)\}$ cannot be contained in a $c$-matrix since it has size larger than 2.
\end{rem}

We conclude this section with an example.

\begin{ex}\label{ex:badArcDiagram}
    The three 2-colored noncrossing arc diagrams in Figure~\ref{fig:supportEquiv} correspond (from left to right) to the following semibrick pairs of $RA_4$:
    \begin{eqnarray*}
        \X_1 &=& {\tiny\begin{matrix}2\\3\\4\end{matrix}}\sqcup 4[1]\sqcup {\tiny\begin{matrix}3\\2\\1\end{matrix}}[1]\\
        \X_2 &=& {\scriptsize\begin{matrix}3\\42\end{matrix}}\sqcup 4[1]\sqcup {\tiny\begin{matrix}3\\2\\1\end{matrix}}[1]\\
        \X_3 &=&  {\scriptsize\begin{matrix}3\\42\end{matrix}}\sqcup 4[1]\sqcup {\tiny\begin{matrix}3\\2\\1\end{matrix}}[1]\sqcup 2[1].
        \end{eqnarray*}
        It follows that $\X_3$ is a 2-term simple minded collection, $\X_2$ is completable, and $\X_1$ is not completable. To see this final claim directly, denote $S = {\tiny\begin{matrix}2\\3\\4\end{matrix}}$. Applying the mutation formula from Definition~\ref{mutation}, we then obtain a semibrick pair
        $$\mu_{S}^+(\X_1) = {\scriptsize\begin{matrix}2\\3\end{matrix}}\sqcup{\tiny\begin{matrix}2\\3\\4\end{matrix}}[1]\sqcup{\tiny\begin{matrix}3\\2\\1\end{matrix}}[1].$$
        We observe that there is a nonzero morphism ${\tiny\begin{matrix}3\\2\\1\end{matrix}} \rightarrow {\scriptsize\begin{matrix}2\\3\end{matrix}}$ which has image $S_3$ (the simple representation at the vertex 3). It is straightforward that this morphism is a left $\Filt\left({\scriptsize\begin{matrix}2\\3\end{matrix}}\right)$-approximation, and so $\mu_{S}^+(\X_1)$ is not mutation compatible. By Corollary~\ref{cor:mutationcompatiblepassdown} and Theorem~\ref{thm:mutation}, is follows that $\X_1$ is not completable.
        
        Moreover, we recall that the left and middle diagrams diagrams in Figure~\ref{fig:supportEquiv} are support equivalent. As in the proof of Corollary~\ref{cor:Main_C2}, the bricks in $\X_1$ do indeed have the same (signed) dimension vectors as those in $\X_2$. These vectors are then included in the $c$-matrix consisting of the (signed) dimension vectors of the bricks in $\X_3$.
\end{ex}

\section{Pairwise completability for preprojective algebras}\label{sec:preprojective}


In this section, we show that if $W$ is a finite Weyl group, then $\Pi_W$ has the pairwise 2-simple minded completability property if and only if $W$ is type $A_1$, $A_2$, or $A_3$; that is, we prove Theorem~\ref{Main_D}. One direction of the proof will follow immediately from Corollary~\ref{Main_A}. The other will essentially follow from considering the cases $A_4$ and $D_4$.

In type A, the result can be seen either directly from the corresponding 2-colored noncrossing arc diagrams or from straightforward computation. We take both of these perspectives in Section~\ref{sec:pairwise_typeA}. In Section~\ref{sec:kstone}, we show how the pairwise 2-simple minded completability property can be simplified the class of ``$K$-stone algebras'', which includes the preprojective algebras of Dynkin type. This allows us to address types D and E in Section~\ref{sec:thmE}.

\subsection{Pairwise completability in type A}\label{sec:pairwise_typeA}

In this section, we fully determine which preprojective algebras of Dynkin type A have the pairwise 2-simple minded completability property.

\begin{thm}\label{thm:pairwise_typeA}
    Let $n$ be a positive integer. Then the preprojective algebra $\Pi_{A_n}$ has the pairwise 2-simple minded completability property if and only if $n \leq 3$.
\end{thm}

\begin{proof}
    If $n \leq 3$, then $\Pi_{A_n}$ has the pairwise 2-simple minded completability property as a consequence of Theorem~\ref{thm:3 vertex}. If $n \geq 4$, we consider the 2-colored noncrossing arc diagram $(\mathcal{G}, \mathcal{R})$ on $n+1$ nodes which contains the following arcs:
    \begin{itemize}
        \item An orange arc between the nodes $n-3$ and $n$ which passes to the left of the nodes $n-2$ and $n-1$.
        \item An orange arc between the nodes $n$ and $n+1$.
        \item A blue arc between the nodes $n-2$ and $n+1$ which passes to the right of the nodes $n-1$ and $n$.
        \item An orange arc between the nodes $i$ and $i + 1$ for each $i \in [n-4]$.
    \end{itemize}
    For $n = 4$, $(\mathcal{G},\mathcal{R})$ is the left diagram in Figure~\ref{fig:pairwise}. (It also appears in Figures~\ref{fig:2-coloredDiagram} and~\ref{fig:supportEquiv}). We showed in Example~\ref{ex:arcDiagramNotComplete} that $(\mathcal{G},\mathcal{R})$ is not completable, and the argument naturally extends to the case where $n > 4$.
    
    Again for $n = 4$, the three diagrams on the right of Figure~\ref{fig:pairwise} show that any pair of arcs in $(\mathcal{G},\mathcal{R})$ form a completable noncrossing arc diagram. For $n > 4$, this is still the case. This can be seen by adding orange arc between the nodes $i$ and $i + 1$ for $i \in [n-4]$ to all of the diagrams in Figure~\ref{fig:pairwise}. It the follows from Theorem~\ref{arc_pairwise_compat} that $\Pi_{A_n}$ fails to have the 2-simple minded completability property whenever $n \geq 4$.
\end{proof}

\begin{figure}
\begin{tikzpicture}

\begin{scope}[decoration={markings,mark=at position 0.25 with {\arrow{<}},mark=at position 0.75 with {\arrow{<}}}]
\draw[very thick,blue,smooth,postaction=decorate](0,-1) to [out = 30, in = -30] (0,2);
\end{scope}

\begin{scope}[decoration={markings,mark=at position 0.25 with {\arrow{>}},mark=at position 0.75 with {\arrow{>}}}]
\draw[very thick,orange,dashed,smooth,postaction=decorate](0,-2) to [out = 150, in = -150] (0,1);
\end{scope}

\begin{scope}[decoration={markings,mark=at position 0.5 with {\arrow{>}}}]
\draw[very thick,orange,dashed,smooth,postaction=decorate] (0,1) to (0,2);
\end{scope}

\draw[fill] (0,1) circle [radius=.07];
\draw[fill] (0,-2) circle [radius=.07];
\draw[fill] (0,0) circle [radius=.07];
\draw[fill] (0,2) circle [radius=.07];
\draw[fill] (0,-1) circle [radius=.07];

\begin{scope}[shift={(4,0)}]
\begin{scope}[decoration={markings,mark=at position 0.25 with {\arrow{<}},mark=at position 0.75 with {\arrow{<}}}]
\draw[very thick,blue,smooth,postaction=decorate](0,0) to [out = 30, in = -30] (0,2);
\end{scope}

\begin{scope}[decoration={markings,mark=at position 0.5 with {\arrow{<}}}]
\draw[very thick,blue,smooth,postaction=decorate](0,-1) to (0,0);
\end{scope}

\begin{scope}[decoration={markings,mark=at position 0.25 with {\arrow{>}},mark=at position 0.75 with {\arrow{>}}}]
\draw[very thick,orange,dashed,smooth,postaction=decorate](0,-2) to [out = 150, in = -150] (0,1);
\end{scope}

\begin{scope}[decoration={markings,mark=at position 0.5 with {\arrow{>}}}]
\draw[very thick,orange,smooth,dashed,postaction=decorate] (0,1) to (0,2);
\end{scope}

\draw[fill] (0,1) circle [radius=.07];
\draw[fill] (0,-2) circle [radius=.07];
\draw[fill] (0,0) circle [radius=.07];
\draw[fill] (0,2) circle [radius=.07];
\draw[fill] (0,-1) circle [radius=.07];
\end{scope}

\begin{scope}[shift={(7,0)}]
\begin{scope}[decoration={markings,mark=at position 0.25 with {\arrow{<}},mark=at position 0.75 with {\arrow{<}}}]
\draw[very thick,blue,smooth,postaction=decorate](0,-1) to [out = 30, in = -30] (0,2);
\end{scope}

\begin{scope}[decoration={markings,mark=at position 0.25 with {\arrow{>}},mark=at position 0.75 with {\arrow{>}}}]
\draw[very thick,orange,dashed,smooth,postaction=decorate](0,-2) to [out = 150, in = -150] (0,0);
\end{scope}

\begin{scope}[decoration={markings,mark=at position 0.5 with {\arrow{>}}}]
\draw[very thick,orange,dashed,smooth,postaction=decorate] (0,1) to (0,2);
\draw[very thick,orange,dashed,smooth,postaction=decorate] (0,0) to (0,1);
\end{scope}

\draw[fill] (0,1) circle [radius=.07];
\draw[fill] (0,-2) circle [radius=.07];
\draw[fill] (0,0) circle [radius=.07];
\draw[fill] (0,2) circle [radius=.07];
\draw[fill] (0,-1) circle [radius=.07];
\end{scope}

\begin{scope}[shift={(10,0)}]
\begin{scope}[decoration={markings,mark=at position 0.25 with {\arrow{<}},mark=at position 0.75 with {\arrow{<}}}]
\draw[very thick,blue,smooth,postaction=decorate](0,-1) to [out = 30, in = -20] (0,2);
\end{scope}

\begin{scope}[decoration={markings,mark=at position 0.25 with {\arrow{>}},mark=at position 0.75 with {\arrow{>}}}]
\draw[very thick,orange,smooth,dashed,postaction=decorate](0,-2) to [out = 150, in = -150] (0,1);
\draw[very thick,orange,dashed,smooth,postaction=decorate] (0,0) to [out = 30, in = -60] (0,2);
\end{scope}

\begin{scope}[decoration={markings,mark=at position 0.5 with {\arrow{<}}}]
\draw[very thick,blue,smooth,postaction=decorate] (0,0) to (0,1);
\end{scope}

\draw[fill] (0,1) circle [radius=.07];
\draw[fill] (0,-2) circle [radius=.07];
\draw[fill] (0,0) circle [radius=.07];
\draw[fill] (0,2) circle [radius=.07];
\draw[fill] (0,-1) circle [radius=.07];
\end{scope}
\end{tikzpicture}
\caption{The 2-colored noncrossing arc diagram on the left is not completable; however, deleting any of its three arcs results in a completable 2-colored noncrossing arc diagram. The corresponding semibrick pair is therefore pairwise completable, but not completable.}\label{fig:pairwise}
\end{figure}


\subsection{$K$-stone algebras}\label{sec:kstone}

In this section, we show how the pairwise 2-simple minded completability property can be simplified for so-called ``$K$-stone algebras''. We begin with the following from \linebreak \cite[Definition~4.29]{DIRRT}.

\begin{define}
    Let $S$ be a brick in $\mods\Lambda$. If $\End_\Lambda(S) \cong K$ and $\Ext_\Lambda^1(S,S) = 0$, then $S$ is called a \emph{$K$-stone}.
\end{define}
We will call an algebra \emph{$K$-stone} if all of its bricks are $K$-stones. It is straightforward that when $\Lambda$ is $K$-stone, a module $M \in \mods\Lambda$ is a brick if and only if $\dim(\Hom_\Lambda(M,M)) = 1$.

The following is the main theorem of this section.
\begin{thm}\label{thm:kstone}
    Let $\Lambda$ be a $\tau$-tilting finite $K$-stone algebra. Let $S$ and $T$ be bricks in $\mods\Lambda$ so that $S \sqcup T[1]$ is a semibrick pair. Then $S\sqcup T[1]$ is mutation compatible if and only if one of the following hold.
    \begin{enumerate}
        \item There are no nonzero morphisms $T \rightarrow S$; that is, $S \sqcup T$ is a semibrick.
        \item There is a monomorphism $T \hookrightarrow S$.
        \item There is an epimorphism $T \twoheadrightarrow S$.
    \end{enumerate}
\end{thm}

We begin building our proof by describing the left minimal $\Filt(S)$-approximation $g_{ST}^+: T\to S$ when $\dim(\Hom_\Lambda(T,S))=1$.

\begin{lem}\label{lem:mutation_comp}
    Let $\Lambda$ be a $\tau$-tilting finite $K$-stone algebra and let $\X=S\sqcup T[1]$ be a semibrick pair.
    If $\Ext_\Lambda^1(S,S) = 0$ and $\dim(\Hom_\Lambda(T,S)) = 1$, then any nonzero map $f: T\rightarrow S$ is a left-minimal $\Filt(S)$ approximation.
    Moreover:
    \begin{enumerate}
        \item if $f$ is a monomorphism, then $\X$ is singly left mutation compatible, and  $\mu^+_S(T[1])= \coker(f)$;
        \item if $f$ is an epimorphism, then $\X$ is singly left mutation compatible, and $\mu^+_S(T[1])$ is equal to $\ker(f)[1]$.
    \end{enumerate}
\end{lem}
\begin{proof}
    We prove the first statement.
    Then the remaining items follow from the definition of left-mutation.
    Note that each module in $\Filt(S)$ is isomorphic to $S^{n}$, the direct sum of $n$ copies of~$S$ where $n\ge1$.
    Let $f:T\to S$ be nonzero.
    Note that any other nonzero map $g:T\to S$ satisfies $g=\lambda f$, where $\lambda$ is a nonzero scalar.
    Therefore any map $T\to S^n$ factors through $f$ as $(\lambda_1f,\lambda_2f, \ldots, \lambda_nf)$.
    We have shown that $f:T\to S$ is a left $\Filt(S)$-approximation.
    Since $S$ is a brick, it is minimal.
\end{proof}

\begin{rem}\label{rem:mutation_comp}
Let $\Lambda$ be an arbitrary finite-dimensional algebra. Let $S\sqcup T[1]$ be a semibrick pair, and assume $\Ext_\Lambda^1(S,S) = 0$. If $\dim(\Hom_\Lambda(T,S))=0$, then the zero map $T\to 0$ is a minimal left $\Filt(S)$-approximation, and $\mu^+_S(T)=T[1]$.
\end{rem}

In the next proposition of Demonet--Iyama--Reading--Reiten--Thomas and the following lemma, we describe when $\dim(\Hom_\Lambda(T,S))=1$.

\begin{prop}\label{prop:kstone}
    Let $\Lambda$ be a $\tau$-tilting finite $K$-stone algebra. Let $S\sqcup T$ be a semibrick in $\mods\Lambda$. Then
    \begin{enumerate}
        \item \cite[Prop 4.33]{DIRRT} $\Filt(S\sqcup T)$ contains at most 4 (isoclasses of) bricks.
        \item \cite[Prop 4.33, 4.34]{DIRRT} There is at most one brick $R$ in $\Filt (S\sqcup T)$ so that $\Hom_\Lambda(S,R) \neq 0$ and $\Hom_\Lambda(R,T)\neq 0$. If such a brick exists, then $\dim(\Hom_\Lambda(S,R)) = 1 = \dim(\Hom_\Lambda(R,T))$ and there is an exact sequence $S \hookrightarrow R \twoheadrightarrow T$.
    \end{enumerate}
\end{prop}

\begin{lem}\label{lem:kstone}
    Let $\Lambda$ be a $\tau$-tilting finite $K$-stone algebra. Let $S$ and $T$ be bricks in $\mods\Lambda$ so that $S\sqcup T[1]$ is a semibrick pair. If there is a monomorphism $T\hookrightarrow S$ or an epimorphism $T \twoheadrightarrow S$, then $\dim(\Hom_\Lambda(T,S)) = 1$.
\end{lem}

\begin{proof}
    Suppose first that there is a monomorphism $f:T \hookrightarrow S$ and consider the short exact sequence
    \begin{equation}T \hookrightarrow S \twoheadrightarrow \coker f.\tag{$\star$}\end{equation}
    Observe that since $S$ is a brick, we must have $\Hom_\Lambda(\coker f, T) = 0$. We claim that in addition, $\Hom_\Lambda(T,\coker f) = 0$ and $\coker f$ is a brick. To see this, first apply the functor $\Hom_\Lambda(S,-)$ to the short exact sequence $(\star)$. This gives an exact sequence
    $$0 = \Hom_\Lambda(S,T) \rightarrow \Hom_\Lambda(S,S) \rightarrow \Hom_\Lambda(S, \coker f) \rightarrow \Ext_\Lambda^1(S,T) = 0,$$
    where the first and last terms are 0 since $S\sqcup T[1]$ is a semibrick pair. This means $K \cong \Hom_\Lambda(S,S) \cong \Hom_\Lambda(S,\coker f)$. Now apply the functor  $\Hom_\Lambda(-,\coker f)$ to the short exact sequence $(\star)$. This gives an exact sequence
    $$0 \rightarrow \Hom_\Lambda(\coker f,\coker f) \rightarrow \Hom_\Lambda(S,\coker f) \rightarrow \Hom_\Lambda(T,\coker f) \rightarrow \Ext_\Lambda^1(\coker f, \coker f).$$
    Since $\dim(\Hom_\Lambda(S,\coker f)) = 1$ we have $\dim(\Hom_\Lambda(\coker f,\coker f)) = 1$ and $\coker f$ is a brick. Since $\Lambda$ is $K$-stone, this implies that $\Ext_\Lambda^1(\coker f, \coker f) = 0$, and so $\Hom_\Lambda(T,\coker f) = 0$ as claimed. Proposition \ref{prop:kstone}(2) then implies $\dim(\Hom_\Lambda(T,S)) = 1$.
    
    Likewise, suppose that there is an epimorphism $f:T\twoheadrightarrow S$ and consider the short exact sequence
    \begin{equation}\ker f\hookrightarrow T \twoheadrightarrow S.\tag{$\dagger$}\end{equation}
    Since $T$ is a brick, we must have that $\Hom_\Lambda(S,\ker f) = 0$. We claim that in addition, $\ker f$ is a brick and $\Hom_\Lambda(\ker f,S) = 0$. To see this, first apply the functor $\Hom_\Lambda(-,T)$ to the short exact sequence $(\dagger)$. This gives an exact sequence
    $$0 = \Hom_\Lambda(S,T) \rightarrow \Hom_\Lambda(T,T) \rightarrow \Hom_\Lambda(\ker f,T) \rightarrow \Ext_\Lambda^1(S,T) = 0,$$
    where the first and last terms are 0 because $S \sqcup T[1]$ is a semibrick pair. This means $K \cong \Hom_\Lambda(T,T) \cong \Hom_\Lambda(\ker f, T)$. Now apply the functor $\Hom_\Lambda(\ker f,-)$ to the short exact sequence $(\dagger)$. This gives an exact sequence
    $$0 \rightarrow \Hom_\Lambda(\ker f,\ker f) \rightarrow \Hom_\Lambda(\ker f,T) \rightarrow \Hom_\Lambda(\ker f,S) \rightarrow \Ext_\Lambda^1(\ker f,\ker f).$$
    Now we know $\Hom_\Lambda(\ker f,\ker f) \neq 0$ and $\dim(\Hom_\Lambda(\ker f, T)) = 1$. Thus $\dim(\Hom_\Lambda(\ker f,\ker f)) = 1$ and $\ker f$ is a brick. Since $\Lambda$ is $K$-stone, this implies that $\Ext_\Lambda^1(\ker f,\ker f) = 0$ and thus $\Hom_\Lambda(\ker f, S) = 0$ as claimed. Propositon \ref{prop:kstone}(2) then implies that $\dim(\Hom_\Lambda(T,S)) = 1$.
\end{proof}

We now prove Theorem \ref{thm:kstone}.

\begin{proof}[Proof of Theorem \ref{thm:kstone}]
     Let $S$ and $T$ be bricks in $\mods\Lambda$ so that $S \sqcup T[1]$ is a semibrick pair.
     
     First, suppose that there are no nonzero morphisms $T \rightarrow S$. 
     By Remark~\ref{rem:mutation_comp}, $\mu^+_S(S\sqcup T[1]) = S[1] \sqcup T[1]$, so $S\sqcup T[1]$ is mutation compatible.
     If there exists $f:T\to S$ that is either a monomorphism or an epimorphism, then Lemma~\ref{lem:kstone} says that $\dim(\Hom_\Lambda(S,T))=1$.
     By Lemma~\ref{lem:mutation_comp}, $f$ is a left minimal $\Filt(S)$-approximation.
     
     If $f$ is a monomorphism, then $\mu^+_S(S\sqcup T[1]) = \coker f \sqcup S[1]$ and there is an epimorphism \linebreak $q:S\twoheadrightarrow \coker f$. Applying Lemma \ref{lem:mutation_comp} again, we see that $q$ is a minimal left $(\Filt \coker f)$-approximation. This means $\mu_{\coker f}^+\circ \mu_S^+(S\sqcup T[1]) = \coker f[1] \sqcup T[1]$, so $S\sqcup T[1]$ is mutation compatible.
     
     Likewise, if $f$ is an epimorphism, then $\mu^+_S(S\sqcup T[1]) = \ker f[1] \sqcup T[1]$, so $S\sqcup T[1]$ is mutation compatible.
     
     Now suppose that $S\sqcup T[1]$ is mutation compatible. Then by Theorems~\ref{thm:mutation} and \ref{torsion_to_simple}, there exists a torsion class $\T \in \tors\Lambda$ so that $T$ is a minimal extending module for $\T$ and $S$ is a minimal coextending module for $\T^\perp$. Now suppose there is a morphism $f:T \rightarrow S$ which is not mono or epi. Then, by Definition \ref{def:minextending}, $\image(f) \in \T\cap \T^\perp$; that is, $\image(f) = 0$. This completes the proof.
\end{proof}

As a consequence of Theorems~\ref{thm:kstone} and~\ref{thm:mutation}, we have the following corollary. This will be useful in proving Theorem~\ref{Main_D}.

\begin{cor}\label{cor:kstone}
 Let $\Lambda$ be a $\tau$-tilting finite $K$-stone algebra, and let $\X= \D\sqcup \U[1]$ be a semibrick pair.
 Then following are equivalent:
 \begin{enumerate}
 \item $\D\sqcup \U[1]$ is pairwise completable.
 \item For each $S\in \D$ and $T[1]\in \U[1]$ one of the following holds:
    \begin{enumerate}
          \item There are no nonzero morphisms $T\rightarrow S$; that is, $S\sqcup T$ is a semibrick.
          \item There is a monomorphism $f:T\hookrightarrow S$.
          \item There is an epimorphism $f: T\twoheadrightarrow S$.
    \end{enumerate}
 \end{enumerate}
\end{cor}

\subsection{Proof of Theorem~\ref{Main_D}}\label{sec:thmE}

We now study the pairwise 2-simple minded completability property for preprojective algebras of types D and E directly. We begin by recalling the following result, which allows us to use Corollary \ref{cor:kstone} in determining whether a semibrick pair is completable.

\begin{thm}\cite[Theorem 1.2]{IRRT}
    Let $W$ be a finite Weyl group. Then $\Pi_W$ is a $K$-stone algebra.
\end{thm}

We begin with type $D_4$.

\begin{prop}\label{prop:D4}
    The algebra $\Pi_{D_4}$ does not have the pairwise 2-simple minded completability property.
\end{prop}

\begin{proof}
    We use the presentation of $\Pi_{D_4}$ given in Example~\ref{ex:preproj}.
    Let $M = {\scriptsize\begin{matrix}1\\3\end{matrix}}$, $N = {\scriptsize\begin{matrix}1\\2\end{matrix}}$, $N' = {\tiny\begin{matrix}2\\1\\4\end{matrix}}$, and $E = {\tiny\begin{matrix}1\\23\\1\\4\end{matrix}}$.
    We see immediately that $M, N$, and $N'$ are all bricks and there is an exact sequence $N' \hookrightarrow E \twoheadrightarrow M$.
    We will demonstrate that $M\sqcup N\sqcup E[1]$ is a pairwise completable semibrick pair which is completable.
    
    We will first show that $\X' = N \sqcup N'[1] \sqcup M[1]$ is a semibrick pair. Indeed, we observe that $\Hom_\Lambda(N,N')$, $\Hom_\Lambda(N,M)$, $\Hom_\Lambda(N',M)$, and $\Hom_\Lambda(M,N')$ are all zero. This means that $\X'$ is a semibrick pair if and only if $\Ext_\Lambda^1(N,M) = 0 = \Ext_\Lambda^1(N,N')$.
    
    To see that $\Ext_\Lambda^1(N,M) = 0$, let $M\hookrightarrow F \twoheadrightarrow N$ be an exact sequence. For $i \in Q_0$, let $F_i$ be the vector space of $F$ at the vertex $i$. If $i \neq 1$, let $f_i$ be the linear transformation corresponding to the arrow from $i$ to 1, and let $g_i$ be the linear transformation corresponding to the arrow from 1 to $i$.
    Thus we have $F_1 \cong K^2$, $F_2,F_3 \cong K$, and $F_4 = 0$.
    Likewise, we have $f_4,g_4 = 0$, $f_2\circ g_2 = -f_3\circ g_3$, and $g_2\circ f_2 = 0 = g_3\circ f_3$.
    In particular, we see $f_2\circ g_2$ is not an isomorphism.
    Suppose first that $\dim(\ker(f_2\circ g_2)) = 1$.
    Thus we have $F \cong {\scriptsize\begin{matrix}1\\23\\1\end{matrix}}$, which contradicts that there is a monomorphism $M\hookrightarrow F$.
    We conclude that $f_2\circ g_2 = 0$.
    Moreover, as there are morphisms $M\hookrightarrow F$ and $F\twoheadrightarrow N$, it must be the case that $g_2, g_3 \neq 0$.
    This means $f_2 = 0 = f_3$, so we have either $F \cong S_1 \sqcup {\scriptsize\begin{matrix}1\\23\end{matrix}}$ or $F \cong M\sqcup N$. However, there is non nonzero morphism $M \rightarrow {\scriptsize\begin{matrix}1\\23\end{matrix}}$.\
    We conclude that $\Ext_\Lambda^1(N,M) = 0$.
    
    To see that $\Ext_\Lambda^1(N,N') = 0$, let $N' \hookrightarrow F\twoheadrightarrow N$ be an exact sequence. We define $F_i, f_i, g_i$ as before.
    Thus we have $F_1,F_2 \cong K^2, F_4 \cong K$, and $F_3 = 0$.
    Likewise, we have $f_3 = g_3 = 0$, $f_2\circ g_2 = -f_4\circ g_4$, and $g_2\circ f_2 = 0 = g_4\circ f_4$.
    Now, as there must be a morphisms $N'\hookrightarrow F$ and $F\twoheadrightarrow N$, we observe that $f_4 = 0$ and $g_4,f_2,g_2 \neq 0$.
    In particular, this means $f_2 \circ g_2 = 0$.
    The only possibility is then $F \cong N\sqcup N'$.
    We conclude that $\Ext_\Lambda^1(N,N') = 0$.
    
    We have now shown that $\X' = N\sqcup N'[1]\sqcup M[1]$ is a semibrick pair. Moreover, since the composition $N' \twoheadrightarrow S_2 \hookrightarrow N$ gives a nonzero map which is neither mono or epi, Corollary~\ref{cor:kstone} implies that $\X'$ is not completable. However, $\Hom_\Lambda(M,N) = 0$, so $\X'$ is singly right mutation compatible at $M$.
    
    Let $\X = \mu^-_{M}(\X')$ be the right mutation of $\X'$ at $M$. Now since $M \sqcup N'$ is a semibrick and our algebra is $K$-stone, it follows from Proposition \ref{prop:kstone}(2) that the morphism $M\rightarrow N'[1]$ corresponding to the short exact sequence $N' \hookrightarrow E \twoheadrightarrow M$ is a minimal right $(\Filt M)$-approximation. Since $\Hom_\Lambda(M,N) = 0$, this means that $\X = M\sqcup N\sqcup E[1]$. Finally, as there are epimorphisms $E \twoheadrightarrow M$ and $E \twoheadrightarrow N$, Corollary~\ref{cor:kstone} implies that $\X$ is pairwise completable.
    
    To summarize, we have shown that $\X$ is a pairwise completable semibrick pair. Moreover, we know that $\X'$ is not mutation compatible, and so $\X$ is not mutation compatible either by Corollary~\ref{cor:mutationcompatiblepassdown}. Corollary~\ref{cor:mutation} then implies that $\Pi_{D_4}$ does not have the pairwise 2-simple minded completability property.
\end{proof}

\begin{rem}
    In the proof of Proposition~\ref{prop:D4}, we note that Theorem~\ref{thm:wide} can also be used to deduce the fact that $\X = M\sqcup N \sqcup E[1]$ is not mutation compatible (and hence not completable). Indeed, the only wide subcategory containing all of $M, N$, and $E$ is $\mods\Pi_{D_4}$ itself, which has rank 4. We are thankful to an anonymous referee for pointing out this fact.
\end{rem}

We are now ready to prove our final main theorem.

\begin{thm}[Theorem \ref{Main_D}]\label{thm:preproj}
    Let $W$ be a finite simply laced Weyl group. Then $\Pi_W$ has the pairwise 2-simple minded completability property if and only if $\rk(\Pi_W) \leq 3$ (i.e. $W$ is of type $A_1, A_2$, or $A_3$).
\end{thm}

\begin{proof}
   Theorem~\ref{thm:pairwise_typeA} proves the result in type A, thus suppose $W$ is type D or E. It follows that the quiver of $\Pi_W$ contains a full subquiver of the form
$$
    \begin{tikzcd}[row sep=tiny]
        &2\arrow[dr,"\alpha" below,shift right]\\
        \overline{Q}_D:&&1\arrow[ul,swap,"\alpha*" above,shift right]\arrow[r,"\beta*" below,shift right]\arrow[dl,swap,"\gamma*" above,shift right]&3\arrow[l,swap,"\beta" above,shift right].\\
        &4\arrow[ur,"\gamma" below,shift right]
    \end{tikzcd}
    $$
    Moreover, the representations in $\mods \Pi_W$ whose support is contained in this full subquiver form a wide subcategory which is equivalent to $\mods \Pi_{D_4}$. It then follows from Theorem~\ref{thm:pairwise_wide} that $\Pi_W$ does not have the pairwise 2-simple minded completability property.
\end{proof}
    
\section{Discussion and future work}\label{sec:discussion}
    For algebras which are not $\tau$-tilting finite, there exist semibrick pairs (including pairwise completable semibrick pairs) $\D \sqcup \U[1]$ which are not 2-term simple minded collections such that $|\D| + |\U| \geq \rk(\Lambda)$. For example, consider the quiver $1 \rightarrow 2 \rightarrow 3 \rightarrow 4 \leftarrow 5 \leftarrow 6 \leftarrow 1$ of type~$\widetilde{A}_5$. In this case (so long as the field $K$ is infinite) there are semibricks of arbitrary (finite) size, formed by taking any finite collection of bricks generating homogeneous tubes. Moreover, the generators of the two tubes of rank 3 form a semibrick (of size $6 = \rk(\Lambda)$) which is pairwise completable (when considered as a semibrick pair) but is not a 2-term simple minded collection.

There are several $\tau$-tilting finite algebras which are derived equivalent to $\tau$-tilting infinite algebras (meaning their bounded derived categories are the same up to equivalence as triangulated categories). Thus, even in the $\tau$-tilting finite case, there can exist collections of at least $\rk(\Lambda)$ hom-orthogonal bricks in $\Db(\mods\Lambda)$ which are not (and are not contained in) simple minded collections. We conjecture that when such a collection of bricks is a semibrick pair, this is not the case. That is, if $\Lambda$ is $\tau$-tilting finite, then any semibrick pair $\D \sqcup \U[1]$ with $|\D| + |\U| = \rk(\Lambda)$ is actually a 2-term simple minded collection. This would imply the following.

\begin{conj}
    Let $\Lambda$ be a $\tau$-tilting finite algebra and let $\D \sqcup \U[1]$ be a semibrick pair. Then $\D \sqcup \U[1]$ is completable if and only if the smallest wide subcategory containing $\D\sqcup \U[1]$ has $|\D| + |\U|$ simple objects.
\end{conj}

We have shown this conjecture holds when $|\D| + |\U| \leq 3$ (Theorem \ref{thm:wide}), and the proof of the ``only if'' part holds in general. The recent work of Jin \cite{jin} may provide a framework for proving this result in general.



\section*{Acknowledgements}
The authors are thankful to Kiyoshi Igusa, Haibo Jin, Job Rock, Hugh Thomas, Gordana Todorov, and John Wilmes for insightful discussions and support. A large portion of this work is included in EH's Ph.D thesis, and a portion of this work was completed while EH was affiliated with the Norwegian University of Science and Technology (NTNU). EH thanks NTNU for their support and hospitality. The authors are also thankful to a pair of anonymous referees for their suggestions on how to improve this paper.

\section*{Declarations}
The authors have no competing interests to declare that are relevant to the content of this article.
Data sharing not applicable to this article as no datasets were generated or analyzed during the current study.

\bibliographystyle{amsplain}
\bibliography{pairwise2}

\end{document}